\documentclass[a4paper]{amsart}

\usepackage[utf8]{inputenc} 
\usepackage[T1]{fontenc}
\usepackage{tgtermes}
\usepackage{newtxmath}
\usepackage[english]{babel}
\usepackage{graphicx}
\usepackage{listings}
\usepackage{mathtools}
\usepackage{mathrsfs}
\usepackage{color}
\usepackage{geometry}
\usepackage{titlesec}
\usepackage{eufrak}

\newcommand{\RR}{\mathbb{R}}

\newcommand{\NN}{\mathbb{N}}

\newtheorem{theorem}{Theorem}[section]
\newtheorem{proposition}[theorem]{Proposition}
\newtheorem{Lem}[theorem]{Lemma}

\newtheorem{Cor}[theorem]{Corollary}
\newtheorem{Claim}[theorem]{Claim}

\theoremstyle{remark}
\newtheorem{Rq}[theorem]{Remark}

\titleformat{\section}[block]
{\centering\scshape\Large\bfseries}
{\normalfont\textbf{\thesection}.}{0.5em}{}

\titleformat{\subsection}[runin]
{\normalfont\bfseries}
{\thesubsection.}{0.5em}{}[.]
\titlespacing{\subsection}{1pc}{1.5ex plus .1ex minus .2ex}{0.5pc}

\titleformat{\subsubsection}[runin]
{\normalfont\bfseries}
{\thesubsubsection.}{0.5em}{}[.]
\titlespacing{\subsubsection}{1pc}{1.5ex plus .1ex minus .2ex}{0.5pc}
\numberwithin{equation}{section}

\titleformat{\paragraph}[runin]
{\normalfont\itshape}
{\theparagraph.}{0.5em}{}[.]
\titlespacing{\paragraph}{1pc}{1.5ex plus .1ex minus .2ex}{0.5pc}

\newcounter{Hyp}
\newcommand{\Hyp}[1]{\stepcounter{Hyp}\smallskip\textbf{(H\theHyp)}\quad \begin{minipage}[t]{0.9\textwidth} #1 \end{minipage}\smallskip}

\newcounter{Hypo}
\newcommand{\Hypo}[1]{\stepcounter{Hypo}\smallskip\textbf{(H'\theHypo)}\quad \begin{minipage}[t]{0.8\textwidth} #1 \end{minipage}\smallskip}

\title[Asymptotic $N$-soliton-like solutions of the NLKG equation]{On existence and uniqueness of asymptotic $N$-soliton-like solutions of the nonlinear Klein-Gordon equation}

\author{Xavier Friederich}
\address{Institut de Recherche Mathématique Avancée UMR 7501, Université de Strasbourg, Strasbourg, France}
\email{friederich@math.unistra.fr}

\begin{document}

\begin{abstract}
We are interested in solutions of the nonlinear Klein-Gordon equation (NLKG) in $\RR^{1+d}$, $d\ge 1$, which behave as a soliton or a sum of solitons in large time. In the spirit of other articles focusing on the supercritical generalized Korteweg-de Vries equations and on the nonlinear Schrödinger equations, we obtain an $N$-parameter family of solutions of (NLKG) which converges exponentially fast to a sum of $N$ given (unstable) solitons. For $N=1$, this family completely describes the set of solutions converging to the soliton considered; for $N\ge 2$, we prove uniqueness in a class with explicit algebraic rate of convergence. 
\end{abstract}

\maketitle

\let\thefootnote\relax
\let\thefootnote\relax

\section{Introduction}

\subsection{Setting of the problem}

We consider the following nonlinear Klein-Gordon equation
\begin{gather}\tag{NLKG}\label{NLKG}
\partial_t^2u=\Delta u-u+f(u),
\end{gather}
where $u$ is a real-valued function of $(t,x)\in\RR\times\RR^d$ and $f$ is a $\mathscr{C}^1$ real-valued function on $\RR$. This equation classically rewrites as the following first order system in time:
\begin{gather}\tag{NLKG'}\label{NLKG'}
\partial_tU=\left(\begin{array}{cc}
0&Id\\
\Delta-Id&0
\end{array}\right)U+\dbinom{0}{f(u)},
\end{gather}
where $U$ is the two-vector $\dbinom{u}{\partial_tu}$. \\

Let us denote by $F$ the unique primitive of $f$ on $\RR$ which vanishes in $0$. We make the following assumptions:
\begin{itemize}
\item if $d=1$, \\
\Hyp{
$f$ is odd and $f'(0)=0$.
}
\Hyp{
There exists $r>0$ such that $F(r)>\frac{1}{2}r^2$.
}

\item if $d\ge 2$,\\
\Hypo{
$f$ is a pure $H^1$-subcritical nonlinearity $r\mapsto \lambda |r|^{p-1}r$, with $\lambda>0$ and $p>1$ if $d=2$ and $p\in\left(1,\frac{d+2}{d-2}\right)$ if $d\ge 3$.
}
 \end{itemize}

Assumption \textbf{(H1)} for $d=1$ or assumption \textbf{(H'1)} for $d\ge 2$  on the nonlinearity $f$ ensures that the Cauchy problem is locally well-posed in the energy space $H^1(\RR^d)\times L^2(\RR^d)$ \cite{gv,nakamuraozawa}. It is even globally well-posed if one assumes further sufficient smallness on the initial condition.

Recall also that the following quantities are conserved for $H^1\times L^2$-solutions $(u,\partial_tu)$ of \eqref{NLKG'}:
\begin{itemize}
\item the energy $\frac{1}{2}\int_{\RR^d}\left\{(\partial_tu)^2+|\nabla u|^2+u^2-2F(u)\right\}(t,x)\;dx$
\item the momentum $\int_{\RR^d}\left\{\partial_tu\nabla u\right\}(t,x)\;dx$.
\end{itemize}

Moreover, the structure of the equation is left invariant under the action of $\RR\times\RR^d$ by (time and space) translation, and under the action of the Lorentz group $O(1,d)$ which consists of the linear automorphisms of $\RR^{1+d}$ that preserve the quadratic form $(t,x_1,\dots,x_d)\mapsto t^2-\sum_{i=1}^dx_i^2$. In other words, precising this latter action, for all $\beta\in\RR^d$ with $\RR^d$-euclidean norm $|\beta|<1$ and $\gamma:=\frac{1}{\sqrt{1-|\beta| ^2}}$, $u$ is a solution of \eqref{NLKG} if and only if
$$(t,x)\mapsto u\left(\Lambda_\beta(t,x)\right)$$ is still a solution to \eqref{NLKG},
where $\Lambda_\beta$ is the linear transformation with matrix
$$\left(\begin{array}{cc}
\gamma&-\gamma\beta\\
-\gamma \beta^\top&I_d+\frac{\gamma-1}{|\beta|^2}\beta^\top \beta
\end{array}\right)$$
in the canonical basis of $\RR^{1+d}$.
We observe in particular that $$\Lambda_\beta(t,x)=\left(\gamma(t-\beta x),\gamma(x-\beta t)\right)$$ if  $d=1$.
We refer to \cite{cotemunoz} for further details concerning the Lorentz transformations in all dimensions.

\bigskip 

It is well-known that \eqref{NLKG} admits a family of solitons indexed by two parameters: the velocity parameter $\beta\in\RR^d$ with $|\beta|<1$ and the translation parameter $x_0\in\RR^d$. Let $Q$ denote the unique (up to translation) positive $H^1$ solution of the following stationary elliptic problem, associated with \eqref{NLKG}:
\begin{equation}
\Delta Q-Q+f(Q)=0
\end{equation}
which we take as radial; for the record, existence of $Q$ follows from a standard result of Berestycki and Lions \cite{berestycki} due to \textbf{(H2)} or \textbf{(H'1)} and uniqueness has been proved in Kwong \cite{kwong} (in the case where $f(u)=|u|^{p-1}u$ is the particular power nonlinearity) and in Serrin and Tang \cite{serrin}. We recall that $Q$ and its partial derivatives up to order $3$ decay exponentially. Then for all $\beta\in\RR^d$ such that $|\beta|<1$, for all $x_0\in\RR^d$, the boosted ground state
$$Q_{\beta,x_0}:(t,x)\mapsto Q\left(pr\circ\Lambda_\beta(t,x-x_0)\right),$$ where $\gamma:=\frac{1}{\sqrt{1-|\beta|^2}}$ and $pr$ is the canonical projection $\RR^{1+d}\to \RR^{d}$ on the last $d$ coordinates, is a solution of \eqref{NLKG} known as \emph{soliton}. In the one-dimensional case, this soliton rewrites 
$$Q_{\beta,x_0}:(t,x)\mapsto Q\left(x-\beta t-x_0\right).$$

Soliton theory concerning \eqref{NLKG} has extensively been studied in many articles. One major result is linked to the classification of the solutions with energy near that of the ground state. 
Dynamics of the solutions $u$ of \eqref{NLKG} on the threshold energy $E(u)=E(Q)$ has been investigated in Duyckaerts and Merle \cite{dm1}. More generally, classification of the solutions with energy less than a quantity slightly larger than the energy of the ground state has been done by Nakanishi and Schlag \cite{nakanishi2011} and by Krieger, Nakanishi and Schlag \cite{krieger}.

Let us also mention that solitons of \eqref{NLKG} are known to be orbitally unstable in $H^1(\RR^d)$ by a general property by Grillakis, Shatah and Strauss \cite{gss}. \\

We further develop soliton analysis by exploring solutions which behave as a soliton or a sum of solitons as time goes to infinity.

For all $\beta\in\RR^d$ such that $|\beta|<1$ and $x_0\in\RR^d$, let us denote $$R_{\beta,x_0}(t,x):=\dbinom{Q_{\beta,x_0}(t,x)}{\partial_tQ_{\beta,x_0}(t,x)}=\dbinom{Q_{\beta,x_0}(t,x)}{-\beta\cdot\nabla Q_{\beta,x_0}(t,x)}.$$
When $x_0=0$, we will write $R_{\beta}$ instead of $R_{\beta,0}$ for the sake of simplification.\\

Drawing on the work by Grillakis, Shatah and Strauss \cite{gss,gss2}, Côte and Muñoz \cite{cotemunoz} have developed and proved spectral results adapted to the unstable dynamic around the (vector) soliton $R_\beta$. Essential properties which are needed in this paper, as well as the introduction of useful notations, are presented in the next subsection. Note that a similar spectral theory was firstly considered by Pego and Weinstein \cite{pego} in the context of the generalized Korteweg-de Vries equations.

Starting from this point of view, we are interested in solutions which converge to a soliton or a sum of solitons for large values of $t$; these solutions are classically known as \emph{multi-solitons}. \\

Let us consider an integer $N\ge 1$ and $2N$ parameters 
$$x_1,\dots,x_N\in\RR^d\quad\text{and}\quad  \beta_1,\dots,\beta_N\in\RR^d$$ such that 
$$\forall\;i=1,\dots,N,\quad |\beta_i|<1\qquad\text{and}\qquad\forall\;i\ne j,\quad \beta_i\ne\beta_j.$$

We recall the following theorem by Côte and Muñoz which states the existence of at least \emph{one} multi-soliton.

\begin{theorem}[\cite{cotemunoz}]\label{th_cotemunoz}
There exist $\sigma_0$, $t_0\in\RR$ and $C_0>0$, only depending on the sets $(\beta_i)_i$, $(x_i)_i$, and a solution $U=\dbinom{u}{\partial_tu}\in\mathscr{C}([t_0,+\infty),H^1(\RR^d)\times L^2(\RR^d))$ of \eqref{NLKG} such that for all $t\ge t_0$,
$$\left\|U(t)-\sum_{i=1}^NR_{\beta_i,x_i}(t)\right\|_{H^1\times L^2}\le C_0e^{-\sigma_0 t}.$$
\end{theorem}

 Let us mention that, dealing with complex valued solutions of \eqref{NLKG}, thus opening the possibi\-lity of considering stable solitons, Bellazzani, Ghimenti and Le Coz \cite{lecoz} obtained a similar existence result for \eqref{NLKG} in this particular stable case. We also notice that the previous theorem has been extended to solutions describing \emph{multi-bound states} by Côte and Martel \cite{cote_martel}, that is to multi-traveling waves made of any number $N$ of decoupled general (excited) bound states. In the present paper, we will however only focus on (real valued) multi-solitary waves in the above sense.

Since solitons are unstable, a solution of \eqref{NLKG} which behaves as a soliton in large time is not expected to be necessarily a soliton. One of our goals is thus to precise the dynamic of the flow of \eqref{NLKG} near a soliton. Similarly, the dynamic near a sum of solitons is also supposed to be more complex as time goes to infinity.

\subsection{Main results}

Given $N$ distinct velocity parameters, we aim at proving the existence of a whole family of multi-solitons which turns out to be the unique family of multi-solitons in a certain class of solutions. Our first result reads as follows.

\begin{theorem}\label{th_main_N}
Assume that $f$ is of class $\mathscr{C}^2$ and $0<|\beta_N|<\dots<|\beta_1|<1$.
There exist $\sigma>0$, $0<e_{\beta_1}<\dots<e_{\beta_N}$, $Y_{+,i}\in \mathscr{C}(\RR,H^1(\RR^d)\times L^2(\RR^d))\cap L^\infty(\RR,H^1(\RR^d)\times L^2(\RR^d))$ for $i=1,\dots,N$  and an $N$-parameter family $(\varphi_{A_1,\dots,A_N})_{(A_1,\dots,A_N)\in\RR^N}$ of solutions of \eqref{NLKG} such that, for all $(A_1,\dots,A_N)\in\RR^N$, there exist $t_0\in\RR$ and $C>0$  such that 
\begin{equation}\label{est_N_param}
\forall\;t\ge t_0,\qquad\left\|\Phi_{A_1,\dots,A_N}(t)-\sum_{i=1}^NR_{\beta_i,x_i}(t)-\sum_{i=1}^NA_ie^{-e_{\beta_i} t}Y_{+,i}(t)\right\|_{H^1\times L^2}\le Ce^{-(e_{\beta_N}+\sigma) t},
\end{equation} where $\Phi_{A_1,\dots,A_N}:=\dbinom{\varphi_{A_1,\dots,A_N}}{\partial_t\varphi_{A_1,\dots,A_N}}$. In addition, if $(A'_1,\dots,A'_N)\neq (A_1,\dots,A_N)$, then $\varphi_{A'_1,\dots,A'_N}\neq \varphi_{A_1,\dots,A_N}$. 
\end{theorem}

\begin{Rq}
The parameters $e_{\beta_i}$ and the functions $Y_{+,i}$ ($i=1,\dots,N$) are defined in Proposition \ref{prop_spectral} and in subsection \ref{subsection_multi-sol}. \\
One can moreover precise the value of $\sigma$ in Theorem \ref{th_main_N}; for this, we refer to \eqref{def_sigma}.
\end{Rq}

Our next result is concerned with the classification of multi-solitons. We aim at proving that any multi-soliton should belong to the family constructed in Theorem \ref{th_main_N} above. This is indeed the case, if one knows that the multi-soliton converges sufficiently fast to its profile. As it is stated below, a decay in a power of $t$ of degree larger than 3 is a sufficient rate. 

Let us emphasize that this is a much weaker assumption than exponential decay, which is natural in view of the convergence \eqref{est_N_param}. We also underline that we obtain for \eqref{NLKG} an effective rate. As a comparison, in the context of the nonlinear Schrödinger equations, uniqueness of a multi-soliton (in the stable case) was shown in a class of solutions with convergence faster than any power of $\frac{1}{t}$, see \cite{cf}: this is still much weaker than the exponential convergence, but not well behaved as what is proven here for \eqref{NLKG}. We nevertheless conjecture that the classification should hold in the most general case when no decay rate is assumed.

Here is the precise statement.

\begin{theorem}\label{th_main_N_class}
Under the assumptions of Theorem \ref{th_main_N} and keeping the same notations, if $u$ is a solution of \eqref{NLKG} such that 
\begin{equation}\label{cv_polynom}
\left\|U(t)-\sum_{i=1}^NR_{\beta_i,x_i}(t)\right\|_{H^1\times L^2}=\mathrm{O}\left(\frac{1}{t^\alpha}\right)\quad \text{as } t\to+\infty, 
\end{equation} where $U=\dbinom{u}{\partial_tu}$ and where $\alpha>3$, then there exist $A_1,\dots,A_N\in\RR$ and $t_0\in\RR$ such that for all $t\ge t_0$, $U(t)=\Phi_{A_1,\dots,A_N}(t)$.
\end{theorem}

\begin{Rq}\label{rq_dim}
Notice that Theorem \ref{th_main_N} and Theorem \ref{th_main_N_class} apply in dimension $d\le 5$ only. Indeed, assuming \textbf{(H'1)} and $f$ of class $\mathscr{C}^2$ forces to have $p>2$, hence $2<\frac{d+2}{d-2}$ if $d\ge 3$.
\end{Rq}

In the case where only one soliton is considered, one can moreover improve the preceding Theorem by completely characterizing solutions which converge to a soliton in large time. 

\begin{theorem}\label{th_main}
Let $\beta\in\RR^d$, $|\beta|<1$ and assume that $f$ is of class $\mathscr{C}^2$. \\
There exist $e_\beta>0$, $Y_{+,\beta}\in \mathscr{C}(\RR,H^1(\RR^d)\times L^2(\RR^d))$, and a one-parameter family $(u^A)_{A\in\RR}$ of solutions of \eqref{NLKG} such that for all $A\in\RR$, there exists $t_0=t_0(A)\in\RR$ such that for all $t\ge t_0$
\begin{equation}\label{est_U^A_thm}
\left\|U^A(t)-R_\beta(t)-Ae^{-e_\beta t}Y_{+,\beta}(t)\right\|_{H^1\times L^2}\le Ce^{-2e_\beta t},
\end{equation}
where $U^A:=\dbinom{u^A}{\partial_tu^A}$. In addition, if $A\neq A'$, then $u^A\neq u^{A'}$. \\
Moreover, if $u$ is a solution of \eqref{NLKG} such that 
\begin{equation}\label{cv_zero}
\|U(t)-R_\beta(t)\|_{H^1\times L^2}\to 0\quad \text{as } t\to+\infty,
\end{equation} where $U=\dbinom{u}{\partial_tu}$, then there exist $A\in\RR$ and $t_0\in\RR$ such that for all $t\ge t_0$, $U(t)=U^A(t)$.
\end{theorem}

\begin{Rq}
The parameter $e_\beta$ (which depends on $\beta$) and the function $Y_{+,\beta}$ are defined in Proposition \ref{prop_spectral}; they are intimately related to the spectral theory dealing with the flow around $R_\beta$.
\end{Rq}

It is interesting to remark that there are only three special solutions among the elements of the preceding family $(U^A)_{A\in\RR}$, up to translations in time and in space. This is the object of the following

\begin{Cor}\label{cor_special_solutions}
Consider the family of solutions $\left(U^A\right)_{A\in\RR}$ defined in Theorem \ref{th_main}.
\begin{enumerate}
\item If $A>0$, there exists $t_A\in\RR$ such that for all possible $t$, $U^A(t)=U^1(t+t_A,\cdot+\beta t_A)$.
\item If $A<0$, there exists $t_A\in\RR$ such that for all possible $t$, $U^A(t)=U^{-1}(t+t_A,\cdot+\beta t_A)$.
\item For all $t\in\RR$, $U^0(t)=R_\beta(t)$.
\end{enumerate}
\end{Cor}

\begin{Rq}
Let us observe that Remark \ref{rq_dim} is valid for Theorem \ref{th_main} and Corollary \ref{cor_special_solutions} too.
\end{Rq}

Theorem \ref{th_main} provides the behavior of the solutions converging to solitons at the order $\mathrm{O}\left(e^{-2e_\beta t}\right)$ in $H^1(\RR^d)\times L^2(\RR^d)$. The instability direction due to the existence of one negative eigenvalue (related to $e_\beta$) for the linearized operator around $Q_\beta$ yields infinitely many possibilities (described by the real line) to perturb the soliton $Q_\beta$ and to retrieve another solution of \eqref{NLKG} which however remains exponentially close in time (with decay rate $e_\beta$) to $Q_\beta$. Let us emphasize that this phenomenon appears quite naturally in the study of unstable solitons.  In a stable mode such as for the $L^2$-subcritical generalized Korteweg-de Vries (gKdV) equation, a solution which converges to a soliton as time goes to infinity is pretty well-known to be exactly the corresponding soliton.

What is more, Theorem \ref{th_main} is built exactly as in \cite{combet2010}. It reminds us of similar classification results, previously obtained for the $L^2$-supercritical gKdV equations by Combet \cite[Theorem 1.1]{combet2010} and for the three-dimensional cubic Schrödinger equation by Duyckaerts and Roudenko \cite[Proposition 3.1]{duyckaerts2010}, both inspired from pioneering works of Duyckaerts and Merle \cite{dm1,dm2}.

As far as the multi-soliton case is concerned, Theorem \ref{th_main_N} provides an $N$-parameter family of solutions to \eqref{NLKG} which behave as a sum of $N$ (differently boosted) solitons in large time. Here again, the existence result looks like that of Combet \cite{combetgKdV,combetNLS} and reinforces the idea that, in general, non-uniqueness holds for multi-solitons in an unstable context. Furthermore, Theorem \ref{th_main_N_class} opens the way to treat the question of the classification of multi-solitons for other models, at least in the restraint class of solutions with algebraic convergence in the sense of \eqref{cv_polynom}; on this matter and given the main existence results in \cite{combetNLS,lecoz} one can for instance explore the issue for the $L^2$-supercritical nonlinear Schrödinger equation or the nonlinear Klein-Gordon equation with complex valued solutions and stable solitons.  
\\

Regarding the global approach developed in this article, we shed new light on the construction of the one-parameter family $(U^A)_{A\in\RR}$ by a compactness procedure. Note that the special solution $U^{-1}$ in \cite{combet2010} has also been obtained in a first rigorous way by a compactness method (the starting point was the proof of instability of the soliton), but the question of obtaining $U^1$ (as well as the other solutions $U^A$ for $A>0$) by such method was raised (see Remark 4.15 in \cite{combet2010}) and, to our knowledge, remained unanswered. Obviously, our process can be thought and used as an alternative to prove similar results in the context of other partial differential equations which involve unstable solitons, and for which the spectral theory around the ground states is well understood (and actually analogous to the present case). This is a nice feature of our paper. 

A second interesting point to be discussed is about proving uniqueness of multi-solitary waves. Exploring the possibility of obtaining a "weak" monotonicity property like \eqref{monotonie_F_E} in Corollary \ref{cor_monotonie_F_E} would be a promising direction of research in order to classify multi-solitons of other models.

Besides, a significant issue would be to know to what extent the classification obtained in Theorem \ref{th_main} could be transcribed to the multi-soliton topic. Indeed, contrary to the gKdV setting \cite{combet2010} and especially as for the NLS case \cite{combetNLS,cf}, proving general uniqueness in the sense of \eqref{cv_zero}, where $R_\beta$ is replaced by a sum of several solitons $R_{\beta_i,x_i}$, still remains unclear. 

Another question to be addressed could be related to the potential generalizations of the previous theorems to multi-bound states.

\subsection{Outline and organization of the article}

In the spirit of \cite{martel} (dealing with the construction of multi-soliton solutions), our approach to constructing the families in Theorem \ref{th_main_N} and Theorem \ref{th_main} is based on backward uniform $H^1\times L^2$-estimates satisfied by well-chosen sequences of solutions of \eqref{NLKG} which aim to approximate the desired solutions. We entirely exploit a coercivity property available in the present matter (see Proposition \ref{prop_coercivity} stated below) in order to establish those estimates and, in fact, to obtain the expected exponential convergences to zero in time. We finally use continuity of the flow of \eqref{NLKG} for the weak $H^1\times L^2$-topology to obtain special solutions which fulfill \eqref{est_N_param} or \eqref{est_U^A_thm}.

Based on compactness and energy methods, the proof of Theorem \ref{th_main_N} follows more precisely the strategy of \cite{combetgKdV,combetNLS}. The construction of $\Phi_{A_1,\dots,A_N}$ is done by iteration, by means of Proposition \ref{prop_main} which roughly asserts that each multi-soliton can be perturbed slightly at the order $e^{-e_{\beta_j}t}$ around the soliton $R_{\beta_j,x_j}$. In order to establish this key proposition, we particularly rely on the topological ingredient set up originally by Côte, Martel and Merle \cite{cmm} for the construction of one given multi-soliton in unstable situations. 

For the construction of $U^A$ in the one-soliton case (Theorem \ref{th_main}), we also substantially rely on the spectral theory available for \eqref{NLKG}, built on the linearized operator of the flow around the boosted ground state. In particular, we point out that our approach differs from previous articles \cite{dm1,dm2,duyckaerts2010,combet2010} where the construction centers around the contraction principle. \\

Regarding the question of classification, from \eqref{est_U^A_thm} and orthogonality properties exposed in the next section, we notice that $A$ corresponds to the limit of $e^{e_\beta t}\langle U^A,Z_{-,\beta}\rangle$ as $t\to+\infty$ in Theorem \ref{th_main}. This is precisely useful for the uniqueness part of this theorem, where the goal is to identify $U$ with an element of the one-parameter family $(U^A)_{A\in\RR}$. Actually, to prove the second part of Theorem \ref{th_main}, we follow \cite{combet2010} in the first instance (up to the obtainment of an exponential control of $U-R_\beta$). Then, a refined version of the coercivity argument considered in \cite{combet2010}, and indeed an elementary but careful analysis of the available estimates, allows us to reach the conclusion, that is to show that $U$ equals some $U^A$ (already constructed in the first part of the theorem), and that without making use of any supplementary tool. We underline once again that we do not need any fixed point argument to conclude. 

Yet, we do not exclude the possibility to prove Theorem \ref{th_main} via the contraction principle; if the nonlinearity $f$ is sufficiently regular ($\mathscr{C}^{s+1}$ for instance), one could precise by this means the behavior of $U^A$ in large time, and indeed expand the solution at the order $\mathrm{O}\left(e^{-(s+1)e_\beta t}\right)$ in the Sobolev space $H^s$, in line with \cite{dm1}. For clarity purposes, we will not explore this path further and anyway, our description of the family $(U^A)_{A\in\RR}$ is sufficient to characterize solutions verifying \eqref{cv_zero}. 

Concerning Theorem \ref{th_main_N_class}, the identification of the solution satisfying \eqref{cv_polynom} is done step by step as in \cite{combetgKdV}. The core of the proof is the obtainment of an almost monotonicity property, inspired by Martel and Merle \cite{mm_waveeq}. By means of a technical lemma of analysis (we refer to Lemma \ref{lem_gen_diff} in Appendix), this monotonicity property allows us to see that any multi-soliton in the class with polynomial convergence to zero converges in fact exponentially (see subsection \ref{subsect_alm_mono_prop}), provided one assumes suitable integrability conditions in the neighborhood of $+\infty$ (and indeed $\alpha>3$). Such a "weak" monotonicity property has a priori been used so far only for the construction of multi-solitons or multi-bound states of the energy-critical wave equation \cite{mm_waveeq,Yuan_2019,xu_yuan}. We also underline that, by Lemma \ref{lem_gen_diff}, we directly obtain the adequate exponential convergence rate which allows us to identify $A_1$; this is in contrast with \cite{combetgKdV}. The monotonicity property is also used as a key ingredient to identify the other parameters $A_2,\dots,A_N$.
\\

This paper is organized as follows. In Section 2, we introduce essential notations and tools which are used throughout our article. In the following sections, we establish the proofs of our main results; for ease of writing, each proof will be made in dimension 1. Section 3 is devoted to the construction of the family of multi-solitons described in Theorem \ref{th_main_N}. Section 4 deals with the classification of multi-solitons, which is the object of Theorem \ref{th_main_N_class}. In Section 5, we study the existence part of Theorem \ref{th_main} focusing on one soliton. Section 6 aims at proving the second part of this latter theorem, that is general uniqueness of the one-parameter family previously constructed. In the appendix, we explain how to adapt the proof to all dimensions, we justify Corollary \ref{cor_special_solutions}, and we state and prove the lemma of analytic theory of differential equations used in Section 4.

As usual, $C$ denotes a positive constant which may depend on the soliton parameters and change from one line to the other, but which is always independent of $t$ and $x$.

\subsection{Acknowledgments}

The author would like to thank his supervisor Raphaël Côte for his constant encouragements and for fruitful discussions. The author is also grateful to Rémi Carles for his comments which improved the quality of this paper.

\section{Notations, review of spectral theory, and multi-solitons}

\subsection{Elements of spectral theory concerning \eqref{NLKG}}

For all $U=\dbinom{u_1}{u_2},V=\dbinom{v_1}{v_2}\in H^1(\RR^d)\times L^2(\RR^d)$, we define the scalar product:
$$\langle U,V\rangle:=\int_{\RR^d}(u_1v_1+u_2v_2)\;dx$$ and the energy norm
$$\|U\|_{H^1\times L^2}:=\left(\|u_1\|_{H^1}^2+\|u_2\|_{L^2}^2\right)^{\frac{1}{2}}.$$

Under assumption \textbf{(H'1)} (or \textbf{(H1)} and \textbf{(H2)} in the particular one-dimensional case), the operator $L:=-\Delta+Id-f'(Q)$ admits a unique simple negative eigenvalue, which we denote by $-\lambda_0$. The kernel of $L$ is spanned by $(\partial_{x_i}Q)_{i=1,\dots,d}$ \cite{maris,mcleod}.\\ Note that for a general nonlinearity $f$ and for $d\ge 2$, the operator $L$ possibly counts several and multiple negative eigenvalues. We refer to Côte and Martel \cite{cote_martel} for the detail of the spectral properties in this case.

With a slight abuse of notation, we still denote by $Q_\beta$ the function defined on $\RR^d$ by $Q_\beta(x):=Q(\gamma x)$, where $\gamma=\frac{1}{\sqrt{1-|\beta|^2}}$ so that for all $t$, $Q_\beta(x)=Q_\beta(t,x)$. In the sequel, we sometimes omit the variables $x$ and $t$ when there is no ambiguity (we work with functions which either depend on time or not). \\

For all $\beta\in \RR^d$ with $|\beta|<1$, we consider the matrix operator $$H_\beta:=\left(\begin{array}{cc}
-\Delta+Id-f'(Q_\beta)&-\beta\cdot\nabla\\
\beta\cdot\nabla&Id
\end{array}\right),$$ the matrix
$J:=\left(\begin{array}{cc}
0&1\\
-1&0
\end{array}\right)$, and the operator $$\mathscr{H}_\beta:=-H_\beta J=\left(\begin{array}{cc}
-\beta\cdot\nabla&\Delta-Id+f'(Q_\beta)\\
Id&\beta\cdot\nabla
\end{array}\right).$$
We define for all $j=1,\dots,d$
$$Z_{j,\beta}:=\dbinom{-\beta \cdot\nabla \left(\partial_{x_j}Q_\beta\right)}{\partial_{x_j}Q_\beta}.$$

\begin{proposition}[Côte and Muñoz \cite{cotemunoz}]\label{prop_spectral}
We have $\mathscr{H}_\beta Z_{j,\beta}=0$ for all $j\in\{1,\dots,d\}$ and there exist two functions $Z_{\pm,\beta}$ whose components decrease exponentially in space and such that $$\mathscr{H}_\beta Z_{\pm,\beta}=\pm e_\beta Z_{\pm,\beta}$$ where $e_\beta:=\frac{\sqrt{\lambda_0}}{\gamma}$. \\
Moreover there exist unique functions $Y_{\pm,\beta}$ (whose components are exponentially decreasing in space) such that 
$$H_\beta Y_{\pm,\beta}\in\mathrm{Span}\{Z_{\pm,\beta}\},\quad \langle JZ_{j,\beta},Y_{\pm,\beta}\rangle =0,\quad  \text{and} \quad \langle Y_{\pm,\beta},Z_{\mp,\beta}\rangle=1.$$ 
In addition, the following orthogonality properties hold:
$$\langle Y_{\pm,\beta},Z_{\pm,\beta}\rangle=0\quad \text{and}\quad\langle JZ_{0,\beta},Z_{\pm,\beta}\rangle =0.$$
\end{proposition}

The following coercivity property turns out to be a crucial tool in our paper.

\begin{proposition}[Almost coercivity of $H_\beta$; Côte and Muñoz \cite{cotemunoz}]\label{prop_coercivity}
There exists $\mu>0$ such that for all $V\in H^1(\RR^d)\times L^2(\RR^d)$, 
$$\langle H_\beta V,V \rangle \ge \mu\|V\|_{H^1\times L^2}^2-\frac{1}{\mu}\left[\langle V,Z_{+,\beta}\rangle^2+\langle V,Z_{-,\beta}\rangle^2+\sum_{j=1}^d\langle V,JZ_{j,\beta}\rangle^2\right].$$
\end{proposition}

\subsection{Multi-soliton results}\label{subsection_multi-sol}

Let us consider a set of $2N$ parameters as given in Theorem \ref{th_main_N} and the associated (vector) solitons $R_i=\dbinom{Q_i}{\partial_tQ_i}:=R_{\beta_i,x_i}$, $i=1,\dots,N$. 
We introduce moreover the vectors:
\begin{itemize}
\item $Y_{\pm,i}(t,x):=Y_{\pm,\beta_i}\left(pr\circ\Lambda_{\beta_i}(t,x-x_i)\right)$
\item $Z_{\pm,i}(t,x):=Z_{\pm,\beta_i}\left(pr\circ\Lambda_{\beta_i}(t,x-x_i)\right)$,
\end{itemize}
where $\gamma_i:=\frac{1}{\sqrt{1-|\beta_i|^2}}$.
We denote $e_i=e_{\beta_i}:=\frac{\sqrt{\lambda_0}}{\gamma_i}=\sqrt{\lambda_0(1-|\beta_i|^2)}$. 

In particular, let us observe that for all $i=1,\dots,N$, $Y_{\pm,i}$ belongs to $\mathscr{C}(\RR,H^1(\RR^d)\times L^2(\RR^d))\cap L^\infty(\RR,H^1(\RR^d)\times L^2(\RR^d))$. \\

There exists $\ell\in\RR^d$ such that 
$$\forall\;i\neq j, \quad\ell \cdot\beta_i\neq \ell\cdot\beta_j\qquad\text{and}\qquad\forall\;i,\quad |\ell\cdot\beta_i|<1;$$ we postpone the argument in the appendix. (If $d=1$, one can take obviously $\ell=1$.)
Let us consider the permutation $\eta$ of $\{1,\dots,N\}$ such that
$$-1<\ell\cdot\beta_{\eta(1)}<\dots<\ell\cdot\beta_{\eta(N)}<1.$$

We denote also
\begin{equation}\label{def_sigma}
\sigma:=\frac{1}{16}\min\left\{e_1,\gamma_N\min\{\ell\cdot(\beta_{\eta(2)}-\beta_{\eta(1)}),\dots,\ell\cdot(\beta_{\eta(N)}-\beta_{\eta(N-1)})\}\right\}>0.
\end{equation}

We can quantify the interactions between the solitons $R_{i}$ and the functions $Y_{\pm,i}$ and $Z_{\pm,i}$, for $i=1,\dots,N$ in terms of the parameter $\sigma$. This is the object of the following

\begin{proposition}\label{prop_interaction}
We have for all $i\neq j$, for all $k,l\in\{0,1,2\}$, and for all $t\ge 0$,
$$\left\langle \partial^{k}_xR_i(t),\partial^{l}_xR_j(t)\right\rangle=\mathrm{O}\left(e^{-4\sigma t}\right).$$
$$\left\langle Y_{\pm,i}(t),Y_{\pm,j}(t)\right\rangle=\mathrm{O}\left(e^{-4\sigma t}\right).$$
$$\left\langle Z_{\pm,i}(t),Z_{\pm,j}(t)\right\rangle=\mathrm{O}\left(e^{-4\sigma t}\right).$$
$$\left\langle Y_{\pm,i}(t),\partial^{l}_xR_{j}(t)\right\rangle=\mathrm{O}\left(e^{-4\sigma t}\right).$$
$$\left\langle Z_{\pm,i}(t),\partial^{l}_xR_{j}(t)\right\rangle=\mathrm{O}\left(e^{-4\sigma t}\right).$$
$$\left\langle Y_{\pm,i}(t),Z_{\pm,j}(t)\right\rangle=\mathrm{O}\left(e^{-4\sigma t}\right).$$
\end{proposition}

What is more, due to Theorem \ref{th_cotemunoz}, there exist $t_0\in\RR$ and $C>0$, only depending on the sets $(\beta_i)_i$, $(x_i)_i$, and a solution $\Phi_0=\dbinom{\varphi_0}{\partial_t\varphi_0}\in\mathscr{C}([t_0,+\infty),H^1(\RR^d)\times L^2(\RR^d))$ of \eqref{NLKG} such that for all $t\ge t_0$,
\begin{equation}\label{est_multi_sol}
\left\|\Phi_0(t)-\sum_{i=1}^NR_i(t)\right\|_{H^1\times L^2}\le Ce^{-4\sigma t}.
\end{equation}

When dealing with the multi-soliton case, we will need to consider in the present article the euclidean space $\left(\RR^k,|\cdot|\right)$ and euclidean balls and spheres of radius $r>0$ in $\RR^k$, $k=1,\dots,N$; in particular we define:
$$B_{\RR^k}(r):=\{x\in\RR^k|\;|x|\le r\}$$
$$S_{\RR^k}(r):=\{x\in\RR^k|\;|x|= r\}.$$

\section{Construction of a family of multi-solitons for $N\ge 2$}\label{sect_const_mult}

In this section, we give a detailed proof of Theorem \ref{th_main_N} in the one-dimensional case. \\
Let $N\ge 2$ and $x_1,\dots,x_N$, $\beta_1,\dots,\beta_N$ be $2N$ parameters as in Theorem \ref{th_main_N}. Denote by $\varphi$ a multi-soliton solution associated with these parameters, satisfying \eqref{est_multi_sol} and consider $\Phi:=\dbinom{\varphi}{\partial_t\varphi}$.

As it was firstly observed in \cite{combetgKdV}, the existence of $(\varphi_{A_1,\dots,A_N})_{(A_1,\dots,A_N)\in\RR^N}$ verifying \eqref{est_N_param} in Theorem \ref{th_main_N} is a consequence of the following crucial

\begin{proposition}\label{prop_main}
Let $j\in\{1,\dots,N\}$ and $A_j\in\RR$. Then there exist $t_0>0$, $C>0$, and a solution $u$ of \eqref{NLKG}, defined on $[t_0,+\infty)$ such that
$$\forall\;t\ge t_0,\qquad \|U(t)-\Phi(t)-A_je^{-e_j t}Y_{+,j}(t)\|_{H^1\times L^2}\le Ce^{-(e_j+\sigma) t},$$
where $U:=\dbinom{u}{\partial_tu}$.
\end{proposition}

By \eqref{est_multi_sol}, assuming that the preceding proposition holds, and considering $A_1,\dots,A_N\in\RR$, we indeed obtain a solution $\phi_{A_1}$ of \eqref{NLKG} and $t_1>0$ such that
$$\forall\;t\ge t_1,\qquad \|\Phi_{A_1}(t)-\Phi_0(t)-A_1e^{-e_1 t}Y_{+,1}(t)\|_{H^1\times L^2}\le Ce^{-(e_1+\sigma) t}$$
with obvious notations. We notice that $\phi_{A_1}$ is a multi-soliton. 
Now, assume that we have constructed, for some $j\in\{1,\dots,N-1\}$, a family of multi-solitons $\phi_{A_1},\dots,\phi_{A_1,\dots,A_j}$ such that there exists $t_j>0$ such that for all $k=1,\dots,j$,
$$\forall\;t\ge t_j,\qquad \|\Phi_{A_1,\dots,A_k}(t)-\Phi_{A_1,\dots,A_{k-1}}(t)-A_ke^{-e_k t}Y_{+,k}(t)\|_{H^1\times L^2}\le Ce^{-(e_k+\sigma) t},$$
where $\Phi_{A_1,\dots,A_{k-1}}=\Phi_0$ if $k=1$.
Hence we can apply Proposition \ref{prop_main} with $\Phi_{A_1,\dots,A_j}$ instead of $\Phi$ and there exist $\Phi_{A_1,\dots,A_{j+1}}$ and $t_{j+1}>0$ such that 
$$\forall\;t\ge t_{j+1},\qquad \|\Phi_{A_1,\dots,A_{j+1}}(t)-\Phi_{A_1,\dots,A_{j}}(t)-A_{j+1}e^{-e_{j+1} t}Y_{+,j+1}(t)\|_{H^1\times L^2}\le Ce^{-(e_{j+1}+\sigma) t}.$$
Thus, by induction on $j$, we obtain a family of multi-solitons $\phi_{A_1},\dots,\phi_{A_1,\dots,A_j}$ such that for all $t\ge t_0:=\max\{t_j|j=1,\dots,N\}$, 
\begin{align*}
\MoveEqLeft[2]
\left\|\Phi_{A_1,\dots,A_N}(t)-\Phi_{0}(t)-\sum_{j=1}^NA_je^{-e_j t}Y_{+,j}(t)\right\|_{H^1\times L^2}\\
&\le \sum_{j=1}^N\|\Phi_{A_1,\dots,A_j}(t)-\Phi_{A_1,\dots,A_{j-1}}(t)-A_je^{-e_j t}Y_{+,j}(t)\|_{H^1\times L^2}\\
&\le  C\sum_{j=1}^N e^{-(e_j+\sigma)t}.
\end{align*}
At this stage, we conclude to \eqref{est_N_param} in Theorem \ref{th_main_N}, using once more \eqref{est_multi_sol} and the triangular inequality. \\

Let us already justify also that for all $(A_1,\dots,A_N)\ne (A'_1,\dots,A'_N)$, we have 
$\varphi_{A_1,\dots,A_N}\ne \varphi_{A'_1,\dots,A'_N}$. \\

Assume for the sake of contradiction that $\varphi_{A_1,\dots,A_N}=\varphi_{A'_1,\dots,A'_N}$ for some $N$-uples $(A_1,\dots,A_N)\ne (A'_1,\dots,A'_N)$. Then, we denote
$$i_0:=\min\{i\in\{1,\dots,N\}|\;A_i'\ne A_i\}.$$
From the construction of $\varphi_{A_1,\dots,A_N}$, there exists $C>0$ such that for $t$ large
\begin{equation}\label{eq_aux_th_1}
\left\|\Phi_{A_1,\dots,A_N}(t)-\Phi_{A_1,\dots,A_{i_0-1}}(t)-\sum_{j=i_0}^NA_je^{-e_j t}Y_{+,j}(t)\right\|_{H^1\times L^2}\le  C\sum_{j=i_0}^N e^{-(e_j+\sigma)t}.
\end{equation}
Similarly there exists $C'>0$ such that for $t$ large
\begin{equation}\label{eq_aux_th_2}
\left\|\Phi_{A'_1,\dots,A'_N}(t)-\Phi_{A'_1,\dots,A'_{i_0-1}}(t)-\sum_{j=i_0}^NA'_je^{-e_j t}Y_{+,j}(t)\right\|_{H^1\times L^2}\le  C'\sum_{j=i_0}^N e^{-(e_j+\sigma)t}.
\end{equation}
Using that $\Phi_{A_1,\dots,A_N}(t)=\Phi_{A'_1,\dots,A'_N}(t)$ and $\Phi_{A_1,\dots,A_{i_0-1}}(t)=\Phi_{A'_1,\dots,A'_{i_0-1}}(t)$, we deduce from \eqref{eq_aux_th_1} and \eqref{eq_aux_th_2} that for all $t$ sufficiently large
$$e^{-e_{i_0}t}|A_{i_0}-A'_{i_0}|\le Ce^{-(e_{i_0}+\sigma)t}.$$
Hence, letting $t\to +\infty$, we obtain $A_{i_0}-A'_{i_0}=0$, which leads to a contradiction. \\
This ends the proof of Theorem \ref{th_main_N}.

\subsection{Compactness argument assuming uniform estimate}

The goal of this subsection is to explain how to prove Proposition \ref{prop_main}; for this, we follow the strategy of Combet \cite{combetgKdV} and Côte and Muñoz \cite{cotemunoz}, both inspired from pioneering work by Martel \cite{martel} and Côte, Martel and Merle \cite{cmm}. One key ingredient in the construction is the obtainment of uniform estimates satisfied by a sequence of approximating solutions of \eqref{NLKG}. 

We fix $j\in\{1,\dots,N\}$ and $A_j\in\RR$. Let $(S_n)_n$ be an increasing sequence of time such that $S_n\to +\infty$. Let us consider $\mathfrak{b}_n=(b_{n,k})_{j<k\le N}\in\RR^{N-j}$ the generic term of a sequence of parameters to be determined, and let $u_n$ be the maximal solution of \eqref{NLKG} such that
\begin{equation}\label{def_U_n}
U_n(S_n)=\Phi(S_n)+A_je^{-e_jS_n}Y_{+,j}(S_n)+\sum_{k>j}b_{n,k}Y_{+,k}(S_n),
\end{equation}
where $U_n:=\dbinom{u_n}{\partial_tu_n}$. 

Concerning $u_n$, we claim:

\begin{proposition}\label{prop_main_seq}
There exist $n_0\ge 0$ and $t_0>0$ (independent of $n$) such that for each $n\ge n_0$, there exists $\mathfrak{b}_n\in\RR^{N-j}$ with $|\mathfrak{b}_n|\le 2e^{-(e_j+2\sigma)t}$ and such that $U_n$ is defined on $[t_0,S_n]$ and satisfies
\begin{equation}\label{est_prop_main_seq}
\forall\;t\in[t_0,S_n],\qquad \|U_n(t)-\Phi(t)-A_je^{-e_j t}Y_{+,j}(t)\|_{H^1\times L^2}\le Ce^{-(e_j+\sigma) t}.
\end{equation}
\end{proposition}

The $b_n$ take the role of modulation parameters and are to be determined (if indeed possible) so that $U_n$ fulfills \eqref{est_prop_main_seq}, thus is a natural candidate in order to "approximate" the desired solution $U$ which is the object of Proposition \ref{prop_main_seq}. \\

We postpone the proof of the previous statement at the next subsection; for the time being, let us assume that Proposition \ref{prop_main_seq} is satisfied and let us show how it implies Proposition \ref{prop_main}. In fact, the existence of $U$ is due to the continuity of the flow of \eqref{NLKG} for the weak $H^1\times L^2$ topology. 
We explicit the construction of $U$ below, following the same strategy as \cite[paragraph 2.2, step 2]{cmm} or \cite[section 4]{cotemunoz}. 

\begin{proof}[Proof of Proposition \ref{prop_main}]
We observe that the sequence $(\|U_n(t_0)\|_{H^1\times L^2})_{n\in\NN}$ is bounded; thus there exist a subsequence of $(U_n(t_0))_{n\in\NN}$, say $(U_{n_k}(t_0))_{k\in\NN}$, and $U_0\in H^1(\RR)\times L^2(\RR)$ such that $(U_{n_k}(t_0))_{k\in\NN}$ converges to $U_0$ in the sense of the weak topology in $H^1(\RR)\times L^2(\RR)$. Let us consider $U$, defined as the maximal solution of \eqref{NLKG} such that $U(t_0)=U_0$. \\  
Let $t\ge t_0$. For $k$ sufficiently large, $S_{n_k}\ge t$ and thus $U_{n_k}$ is defined on $[t_0,t]$. By a standard result (we refer to \cite[Lemma 10]{cotemunoz} and \cite[Theorem 1.2]{tao}), $U$ is defined on $[t_0,t]$ and $(U_{n_k}(t))_k$ converges weakly to $U(t)$ in $H^1(\RR)\times L^2(\RR)$. \\
Moreover, by property of the weak limit, 
\begin{align*}
\|U(t)-\Phi(t)-A_je^{-e_j t}Y_{+,j}(t)\|_{H^1\times L^2}&\le \liminf_{k\to+\infty}\|U_{n_k}(t)-\Phi(t)-A_je^{-e_j t}Y_{+,j}(t)\|_{H^1\times L^2}\\
&\le C_0e^{-(e_j+\sigma) t}.
\end{align*}
\end{proof}

Now, the remainder of Section \ref{sect_const_mult} is devoted to the proof of Proposition \ref{prop_main_seq}. 

\subsection{Proof of Proposition \ref{prop_main_seq}}

For ease of reading, we will drop the index $n$ for the rest of this subsection (except for $S_n$), that is, we will write $U$ for $U_n$, $\mathfrak{b}$ for $\mathfrak{b}_n$, etc. 

Let us introduce the following variable (which depends on $n$)
$$W(t):=U(t)-\Phi(t)-A_je^{-e_j t}Y_{+,j}(t)$$ and for all $k\in\{1,\dots,N\}$, 
$$\alpha_{\pm,k}(t):=\langle W(t),Z_{\pm,k}(t)\rangle $$ (which depends on $\mathfrak{b}$ in particular by definition of $U=U_n$ \eqref{def_U_n}). \\
We denote also $\alpha_-(t):=(\alpha_{-,k}(t))_{j<k\le N}$.

\subsubsection{Modulated final data and strategy of the proof of Proposition \ref{prop_main_seq}}

We make the first step in order to determine the appropriate modulation parameter $\mathfrak{b}$. We obtain $\mathfrak{b}$ as the solution of a well-chosen equation; this is the object of the following

\begin{Lem}\label{lem_mod_fin}
There exists $n_0\ge 0$ such that for all $n\ge n_0$ and for all $\mathfrak{a}\in\RR^{N-j}$, there exists a unique $\mathfrak{b}\in\RR^{N-j}$ such that $\|\mathfrak{b}\|\le 2\|\mathfrak{a}\|$ and $\alpha_-(S_n)=\mathfrak{a}$.
\end{Lem}

\begin{proof}
Let us consider the linear application
$$\begin{array}{cccl}
\Psi:&\RR^{N-j}&\rightarrow&\RR^{N-j}\\
&\mathfrak{b}=(b_l)_{j<l\le N}&\mapsto&\left(\sum_{l>j}b_l\langle Y_{+,l}(S_n),Z_{-,l}(S_n)\rangle\right)_{j<k\le N}.
\end{array}$$
Its matrix in the canonical basis of $\RR^{N-j}$ has generic entry $\psi_{k,l}:=\langle Y_{+,j+l}(S_n), Z_{-,j+k}(S_n)\rangle$ where $(k,l)\in\{1,\dots,N\}^2$. \\
Since $\psi_{k,l}=1$ if $k=l$ and $|\psi_{k,l}|\le C_0e^{-\sigma S_n}$ for $k\ne l$, with $C_0>0$ independent of $n$, we have $\Psi=Id+M$ with $\| M\|\le \frac{1}{2}$ for large values of $n$. Thus $\Psi$ is invertible (for $n$ large) and $\|\Psi^{-1}\|\le 2$. \\
We deduce the content of Lemma \ref{lem_mod_fin} by taking $n_0$ large enough and by considering, for a given $\mathfrak{a}\in\RR^{N-j}$, the element $\mathfrak{b}:=\Psi^{-1}(\mathfrak{a})$.
\end{proof}

Roughly speaking, Lemma \ref{lem_mod_fin} reflects that estimate \eqref{est_prop_main_seq} is to be proven by choosing a relevant vector $\mathfrak{a}=a_-(S_n)$. \\
The reason why we determine $\mathfrak{b}$ according to the value of $\alpha_-(S_n)$ essentially comes from the directions $Z_{-,k}$, which yield "instability" in some sense (given Claim \ref{claim_alpha} below), and also from definition \eqref{def_T(a)} below. \\

At this stage, we notice that we already have:

\begin{Claim}\label{claim_mod_fin}
We have:
\begin{enumerate}
\item $\forall\;k\in\{1,\dots,N\},\quad |\alpha_{+,k}(S_n)|\le C|\mathfrak{b}|e^{-2\sigma S_n}$.
\item $\forall\;k\in\{1,\dots,j\},\quad |\alpha_{-,k}(S_n)|\le C|\mathfrak{b}|e^{-2\sigma S_n}$.
\item $\|W(S_n)\|_{H^1\times L^2}\le C|\mathfrak{b}|$.
\end{enumerate}
\end{Claim}

Let $t_0>0$ independent of $n$ to be chosen later and $\mathfrak{a}\in B_{\RR^{N-j}}(e^{-(e_j+2\sigma)S_n})$ to be determined. We consider the associated data $\mathfrak{b}$ given by Lemma \ref{lem_mod_fin} and $U$ defined in \eqref{def_U_n}. \\
Let us define
\begin{equation}\label{def_T(a)}
T(\mathfrak{a}):=\inf\{T\ge t_0\;|\:\forall\;t\in[T,S_n],\:\|W(t)\|_{H^1\times L^2}\le e^{-(e_j+\sigma)t} \text{ and } e^{(e_j+2\sigma)t}\alpha_{-}(t)\in B_{\RR^{N-j}}(1)\}.
\end{equation}

We observe that Proposition \ref{prop_main_seq} holds if for all $n$, we can find $\mathfrak{a}$ such that $T(\mathfrak{a})=t_0$. In the rest of the proof, our goal is thus to prove the existence of such an element $\mathfrak{a}$. \\

To this end, we will first of all improve the estimate on $\|W(t)\|_{H^1\times L^2}$ which falls within the definition of $T(\mathfrak{a})$. This is the object of the following subsection. Then, we will only need to care about the second condition, which implies a control of $\alpha_-(t)$; this is done in subsection \ref{subsect_topological_arg}. 

\subsubsection{Improvement of the estimate on $\|W\|_{H^1\times L^2}$}\label{subsect_improvement}

For notation purposes and ease of reading, we sometimes omit the index $n$ and also write $\mathrm{O}\left(G(t)\right)$ in order to refer to a function $g$ which a priori depends on $n$ and such that there exists $C\ge 0$ such that for all $n$ large and for all $t\in[t_n^*,S_n]$, $|g(t)|\le C|G(t)|$.

\begin{Lem}\label{lem_est_W}
There exists $K_0> 0$ such that for all $t\in[T(\mathfrak{a}),S_n]$,
$$\|W(t)\|_{H^1\times L^2}\le\frac{K_0}{t^{\frac{1}{4}}}e^{-(e_j+\sigma)t}.$$
\end{Lem}

The whole subsection consists of the proof of this lemma.

\paragraph{\underline{Step 1}: Estimates on $\alpha_{\pm,k}$}

Let us begin with the computation of the time derivative of $W$.

\begin{Claim}\label{claim_der_t_W}
We have for all $k\in\{1,\dots,N\}$,
\begin{equation}\label{egalite_der_W}
\begin{aligned}
\partial_tW=&\left(\begin{array}{cc}
0&Id\\
\partial_x^2-Id+f'(\varphi)&0
\end{array}\right)W+A_je^{-e_jt}\left[\left(\begin{array}{cc}
\beta_j\partial_x&Id\\
\partial_x^2-Id+f'(Q_k)&\beta_j\partial_x
\end{array}\right)Y_{+,j}+e_jY_{+,j}\right]\\
&+A_je^{-e_jt}\left(\begin{array}{cc}
0&0\\
f'(\phi)-f'(Q_k)&0
\end{array}\right)Y_{+,j}+\dbinom{0}{g},
\end{aligned}
\end{equation}
where $g:=f(u)-f(\varphi)-f'(\varphi)(u-\varphi)$ satisfies
$$\|g(t)\|_{L^\infty}=\mathrm{O}\left(\|u-\varphi\|_{H^1}^2\right).$$
\end{Claim}

\begin{proof}
Claim \ref{claim_der_t_W} follows from the fact that both $U$ and $\Phi$ satisfy \eqref{NLKG'} and is also a consequence of the following Taylor inequality ($f$ is $\mathscr{C}^2$)
$$|f(u)-f(\varphi)-f'(\varphi)(u-\varphi)|(t)\le C\|u(t)-\varphi(t)\|^2_{L^\infty}$$
and the Sobolev embedding $H^1(\RR)\hookrightarrow L^\infty(\RR)$.
\end{proof}

Now, we are in a position to prove the following estimate on $\alpha_{\pm,k}$.

\begin{Claim}\label{claim_alpha}
For all $k\in\{1,\dots,N\}$ and for all $t\in [T(\mathfrak{a}),S_n]$, we have
\begin{equation}
\left|\frac{d}{dt}\alpha_{\pm,k}(t)\mp e_k\alpha_{\pm,k}(t)\right|\le C\left(e^{-4\sigma t}\|W(t)\|_{H^1\times L^2}+\|W(t)\|_{H^1\times L^2}^2+e^{-(e_j+4\sigma)t}\right).
\end{equation}
\end{Claim}

\begin{proof}
Let $k\in\{1,\dots,N\}$. By means of \eqref{egalite_der_W} and since $\partial_tZ_{\pm,k}=-\beta_k\partial_xZ_{\pm,k}$, we compute
\begin{align*}
\frac{d}{dt}\alpha_{\pm,k}(t)=&\;\langle \partial_tW,Z_{\pm,k}\rangle + \langle W,\partial_tZ_{\pm,k}\rangle\\
=&\;\left\langle W, \left(\begin{array}{cc}
-\beta_k\partial_x&\partial_x^2-Id+f'(\varphi)\\
Id&-\beta_k\partial_x
\end{array}\right)
Z_{\pm,k}\right\rangle\\
&+A_je^{-e_jt}\left\langle
Y_{+,j},\left(\begin{array}{cc}
-\beta_j\partial_x&\partial_x^2-Id+f'(Q_k)\\
Id&-\beta_j\partial_x
\end{array}
\right)Z_{\pm,k}
\right\rangle\\
&+A_je^{-e_jt}\left[\left\langle Y_{+,j}, \left(\begin{array}{cc}
0&f'(Q_k)-f'(\varphi)\\
0&0
\end{array}
\right)Z_{\pm,k}
\right\rangle+e_j\langle Y_{+,j},Z_{\pm,k}\rangle\right]\\
&+\mathrm{O}\left(\|U-\Phi\|_{H^1\times L^2}^2\right).
\end{align*}
Let us notice first that
$$\left\langle W, \left(\begin{array}{cc}
-\beta_k\partial_x&\partial_x^2-Id+f'(\varphi)\\
Id&-\beta_k\partial_x
\end{array}\right)
Z_{\pm,k}\right\rangle=\langle W,\mathscr{H}_kZ_{\pm,k}\rangle+
\left\langle W, \left(\begin{array}{cc}
0&f'(\varphi)-f'(Q_k)\\
0&0
\end{array}\right)
Z_{\pm,k}\right\rangle.$$

We have $$\langle W,\mathscr{H}_kZ_{\pm,k}\rangle=\langle W,\pm e_kZ_{\pm,k}\rangle=\pm e_k\alpha_{\pm,k}$$ and
\begin{align*}
\MoveEqLeft[4]
\left|\left\langle W, \left(\begin{array}{cc}
0&f'(\varphi)-f'(Q_k)\\
0&0
\end{array}\right)
Z_{\pm,k}\right\rangle\right| \\
&\le \|W\|_{H^1\times L^2}\|\left(f'(\varphi)-f'(Q_k)\right)Z_{\pm,k}\|_{H^1\times L^2}\\
&\le C\left\|\varphi-\sum_{i=1}^NQ_i\right\|_{L^\infty}\|W\|_{H^1\times L^2}+C\|W\|_{H^1\times L^2}\sum_{i\ne k}\|Q_iZ_{\pm,k}\|_{H^1\times L^2}\\
&\le Ce^{-4\sigma t}\|W\|_{H^1\times L^2}.
\end{align*}

Similarly, we have 
$$
\left|\left\langle Y_{+,j}, \left(\begin{array}{cc}
0&f'(Q_k)-f'(\varphi)\\
0&0
\end{array}
\right)Z_{\pm,k}
\right\rangle\right|\le Ce^{-4\sigma t}
$$
and
\begin{align*}
\left\langle
Y_{+,j},\left(\begin{array}{cc}
-\beta_j\partial_x&\partial_x^2-Id+f'(Q_k)\\
Id&-\beta_j\partial_x
\end{array}
\right)Z_{\pm,k}
\right\rangle &= \langle Y_{+,j},\mathscr{H}_kZ_{\pm,k}\rangle+(\beta_k-\beta_j)\langle Y_{+,j},\partial_xZ_{\pm,k}\rangle \\
&= \pm\langle Y_{+,j},e_kZ_{\pm,k}\rangle+\mathrm{O}\left(e^{-4\sigma t}\right).
\end{align*}
Indeed, we notice that 
$$(\beta_k-\beta_j)\langle Y_{+,j},\partial_xZ_{\pm,k}\rangle=\begin{dcases}
0&\text{if } k=j\\
\mathrm{O}(e^{-4\sigma t})&\text{if } k\neq j.
\end{dcases}$$
Hence, we obtain
\begin{equation}
\begin{aligned}
\frac{d}{dt}\alpha_{\pm,k}(t)=&\pm e_k\alpha_{\pm,k}+ \mathrm{O}\left(e^{-4\sigma t}\|W\|_{H^1\times L^2}\right)\\
&+A_je^{-e_jt}\left[\pm\langle Y_{+,j},e_kZ_{\pm,k}\rangle+e_j\langle Y_{+,j},Z_{\pm,k}+\mathrm{O}\left(e^{-4\sigma t}+\|U-\Phi\|_{H^1\times L^2}^2\right)\right].
\end{aligned}
\end{equation}

Now, we observe that 
$$\pm e_k\langle Y_{+,j},Z_{\pm,k}\rangle +e_j\langle Y_{+,j},Z_{\pm,k}\rangle=\mathrm{O}(e^{-4\sigma t}).$$
This is clear if $k\neq j$ and for $k=j$, we have 
$$\pm e_j\langle Y_{+,j},Z_{\pm,j}\rangle +e_j\langle Y_{+,j},Z_{\pm,j}\rangle=\begin{dcases}
0+0=0& \text{if } \pm=+\\
-e_j+e_j=0& \text{if } \pm=-
\end{dcases};$$
indeed, we recall from Proposition \ref{prop_spectral}
$$\langle Y_{+,j},Z_{+,j}\rangle =0\quad\text{and}\quad\langle Y_{+,j},Z_{-,j}\rangle =1.$$
In addition, we have by the well-known inequality $(a+b)^2\le 2(a^2+b^2)$,
$$\|U-\Phi\|^2_{H^1\times L^2}\le C\left(\|W\|^2_{H^1\times L^2}+e^{-2e_j t}\right).$$
Considering that $2e_j\ge e_j+4\sigma$, we have thus finished the proof of the claim.
\end{proof}

\paragraph{\underline{Step 2}: Control of the stable directions}

\begin{Claim}\label{claim_alpha_2}
We have for all $k\in\{1,\dots,N\}$, for all $t\in[T(\mathfrak{a}),S_n]$,
\begin{equation}
|\alpha_{+,k}(t)|\le Ce^{-(e_j+4\sigma) t}.
\end{equation}
\end{Claim}

\begin{proof}
Due to Claim \ref{claim_alpha} and \eqref{def_T(a)}, we obtain
$$\left|\frac{d}{dt}\alpha_{+,k}(t)- e_k\alpha_{+,k}(t)\right|\le Ce^{-(e_j+4\sigma)t},$$
that is, for all $t\in[T(\mathfrak{a}),S_n]$,
$$\left|\left(e^{-e_kt}\alpha_{+,k}(t)\right)'\right|\le Ce^{-(e_j+e_k+4\sigma)t}.$$ Integrating, we deduce
that for all $t\in[T(\mathfrak{a}),S_n]$,
$$|e^{-e_kS_n}\alpha_{+,k}(S_n)-e^{-e_kt}\alpha_{+,k}(t)|\le Ce^{-(e_j+e_k+4\sigma)t}.$$ Thus,
$$|\alpha_{+,k}(t)|\le |\alpha_{+,k}(S_n)|+Ce^{-(e_j+4\sigma t)}.$$
From Claim \ref{claim_mod_fin} and Lemma \ref{lem_mod_fin}, we have
\begin{align*}
|\alpha_{+,k}(S_n)|&\le Ce^{-2\sigma S_n}|\mathfrak{b}| \\
&\le Ce^{-2\sigma S_n}e^{-(e_j+2\sigma)S_n}\\
&\le Ce^{-(e_j+4\sigma)t}.
\end{align*}
Consequently, Claim \ref{claim_alpha_2} indeed holds.
\end{proof}

\paragraph{\underline{Step 3}: Control of the unstable directions for $k\le j$} 

\begin{Claim}\label{claim_alpha_3}
We have for all $k\in\{1,\dots,j\}$, for all $t\in[T(\mathfrak{a}),S_n]$,
\begin{equation}
|\alpha_{-,k}(t)|\le Ce^{-(e_j+4\sigma) t}.
\end{equation}
\end{Claim}

\begin{proof}
As in the preceding step, we have for all $k\in\{1,\dots,N\}$ and $t\in[T(\mathfrak{a}),S_n]$,
\begin{equation}\label{est_alpha-k}
\left|\frac{d}{dt}\alpha_{-,k}(t)+ e_k\alpha_{-,k}(t)\right|\le Ce^{-(e_j+4\sigma)t},
\end{equation}
which writes also
$$\left|\left(e^{e_kt}\alpha_{-,k}(t)\right)'\right|\le Ce^{(e_k-e_j+4\sigma)t}.$$
For $k\le j$, we have $e_k\le e_j$, and so by integration, we obtain
$$|\alpha_{-,k}(t)|\le e^{e_k(S_n-t)}|\alpha_{-,k}(S_n)|+Ce^{-(e_j+4\sigma)t}.$$
But again from Claim \ref{claim_mod_fin} and Lemma \ref{lem_mod_fin}, we infer
\begin{align*}
e^{e_k(S_n-t)}|\alpha_{-,k}(S_n)|&\le Ce^{e_k(S_n-t)}e^{-2\gamma S_n}e^{-(e_j+2\sigma)S_n}\\
&\le Ce^{(S_n-t)(e_k-e_j)}e^{-e_jt}e^{-4\sigma S_n}\\
&\le Ce^{-(e_j+4\sigma)t}.
\end{align*}
Thus $$\forall\;k\in\{1,\dots,j\},\:\forall\;t\in[T(\mathfrak{a}),S_n],\qquad |\alpha_{-,k}(t)|\le Ce^{-(e_j+4\sigma )t}.$$
\end{proof}

\paragraph{\underline{Step 4}: Control of a Lyapunov functional satisfying a coercivity property}

Let us consider $$\begin{array}{cccl}
\psi:&\RR&\to& \RR\\
&x&\mapsto&\frac{2}{\pi}\mathrm{Arctan}\left(e^{-x}\right).
\end{array}$$
We define for all $k=1,\dots,N-1$,
$$\psi_k(t,x):=\psi\left(\frac{1}{\sqrt{t}}\left(x-\frac{\beta_{\eta(k)}+\beta_{\eta(k+1)}}{2}t-\frac{x_{\eta(k)}+x_{\eta(k+1)}}{2}\right)\right),$$
and then 
\begin{align*}
\phi_1(t)&=\psi_1(t) \\
\phi_k(t)&=\psi_k(t)-\psi_{k-1}(t) \text{ for all } k=2,\dots,N-1,\\
\phi_N(t)&=1-\psi_{N-1}(t).
\end{align*}
Recall that the permutation $\eta$ has been chosen so that $-1<\beta_{\eta(1)}<\dots<\beta_{\eta(N)}<1$. \\

Now, let us introduce for all $k\in\{1,\dots,N\}$
$$\mathcal{F}_{W,k}(t)=\int_{\RR}\left(w_1^2+(\partial_xw_{1})^2+w_2^2-f'(Q_{\eta(k)})w_1^2+2\beta_{\eta(k)}\partial_xw_1w_2\right)\phi_k\;dx,$$
and
$$\mathcal{F}_W(t):=\sum_{k=1}^N\mathcal{F}_{W,k}(t).$$

By means of Proposition \ref{prop_coercivity} and a usual localization argument \cite{tsai_gkdv,tsai}, we obtain that $\mathcal{F}_W$ is coercive on a subspace of $H^1\times L^2$ of finite codimension. More precisely, there exists $\mu>0$ such that

\begin{equation}
\mathcal{F}_W(t)\ge \mu\|W(t)\|^2_{H^1\times L^2}-\frac{1}{\mu}\sum_{k=1}^N\left(\langle W,\partial_xR_k\rangle^2+\langle W,Z_{+,k}\rangle^2+\langle W,Z_{-,k}\rangle^2\right).
\end{equation}

We state the following control about the derivative of $\mathcal{F}_W$:

\begin{Claim}
For $t_0$ large and for all $t\in [T(\mathfrak{a}),S_n]$,
\begin{equation}\label{der_F_W}
\left|\frac{d}{dt}\mathcal{F}_W(t)\right|\le \frac{C}{\sqrt{t}}\|W\|^2_{H^1\times L^2}.
\end{equation}
\end{Claim}

\begin{proof}
Let us rewrite $\mathcal{F}_{W,k}$ differently, using the notations developed in the introduction. Relying on integrations by parts, our computations lead to:
$$\langle (H_{\eta(k)}W)\phi_k,W\rangle=\mathcal{F}_{W,k}(t)-\frac{1}{2}\int_\RR w_1^2\partial_x^2\phi_k\;dx+\beta_{\eta(k)}\int_\RR w_1w_2\partial_x\phi_k\;dx.$$
Thus \begin{equation}
\mathcal{F}_{W,k}(t)=\langle (H_{\eta(k)}W)\phi_k,W\rangle+\mathrm{O}\left(\frac{1}{\sqrt{t}}\|W\|^2_{H^1\times L^2}\right).
\end{equation}

We immediately have
$$
\frac{d}{dt}\langle (H_{\eta(k)}W)\phi_k,W\rangle=\langle (H_{\eta(k)}W)\phi_k,\partial_tW\rangle +\langle \partial_t(H_{\eta(k)}W)\phi_k,W\rangle+\langle (H_{\eta(k)}W)\partial_t\phi_k,W\rangle.$$

Besides 
\begin{align*}
\MoveEqLeft[4]
\langle \partial_t(H_{\eta(k)}W)\phi_k,W\rangle+\langle (H_{\eta(k)}W)\partial_t\phi_k,W\rangle\\
&=\langle (H_{\eta(k)}\partial_t W)\phi_k,W\rangle+\beta_{\eta(k)}\int_\RR\partial_xQ_{\eta(k)}f''(Q_{\eta(k)})w_1^2\phi_k\;dx+\mathrm{O}\left(\frac{1}{\sqrt{t}}\|W\|^2_{H^1\times L^2}\right)\\
&=\langle H_{\eta(k)}(\partial_t W),W\phi_k\rangle+\beta_{\eta(k)}\int_\RR\partial_xQ_{\eta(k)}f''(Q_{\eta(k)})w_1^2\phi_k\;dx+\mathrm{O}\left(\frac{1}{\sqrt{t}}\|W\|^2_{H^1\times L^2}\right).
\end{align*}
Since $H_{\eta(k)}$ is a self-adjoint operator, we have
$$\langle H_{\eta(k)}(\partial_t W),W\phi_k\rangle=\langle \partial_t W,H_{\eta(k)}(W\phi_k)\rangle.$$

By a straightforward calculation, we have moreover
$$\langle H_{\eta(k)}(W\phi_k),\partial_tW\rangle= \langle (H_{\eta(k)}W)\phi_k,\partial_tW\rangle+\mathrm{O}\left(\frac{1}{\sqrt{t}}\|W\|_{H^1\times L^2}^2\right).$$
At this stage, we thus obtain
\begin{multline*}
\frac{d}{dt}\langle (H_{\eta(k)}W)\phi_k,W\rangle=2\langle (H_{\eta(k)}W)\phi_k,\partial_tW\rangle\\
+\beta_{\eta(k)}\int_\RR\partial_xQ_{\eta(k)}f''(Q_{\eta(k)})w_1^2\phi_k\;dx+\mathrm{O}\left(\frac{1}{\sqrt{t}}\|W\|^2_{H^1\times L^2}\right).
\end{multline*}

Now, by \eqref{egalite_der_W}, we write
$$
\left\langle (H_{\eta(k)}W)\phi_k,\partial_tW\right\rangle=I_1+I_2+I_3
$$
where $$I_1:=\left\langle \left(\begin{array}{cc}
T_{\eta(k)}&0\\
0&Id
\end{array}\right)W\phi_k,\left(\begin{array}{cc}
0&Id\\
\partial_x^2-Id+f'(\varphi)&0
\end{array}
\right)W\right\rangle$$ by denoting $T_i=-\partial_x^2+Id-f'(Q_{i})$ for all $i=1,\dots,N$,

$$I_2:=\beta_{\eta(k)}\left\langle\left(\begin{array}{cc}
0&-\partial_x\\
\partial_x&0\\
\end{array}\right)W\phi_k,\left(\begin{array}{cc}
0&Id\\
-T_{\eta(k)}+f'(\varphi)-f'(Q_{\eta(k)})&0\\
\end{array}\right)W\right\rangle,$$
and
\begin{multline*}
I_3:=A_je^{-e_jt}\left\langle (H_{\eta(k)}W)\phi_k,
\left(\begin{array}{cc}
\beta_j\partial_x&Id\\
\partial_x^2-Id+f'(Q_{j})&\beta_j\partial_x
\end{array}\right)Y_{+,j}\right\rangle\\
+\left\langle  (H_{\eta(k)}W)\phi_k,e_jY_{+,j}+\left(\begin{array}{cc}
0&0\\
f'(\varphi)-f'(Q_j)&0
\end{array}\right)Y_{+,j} \right\rangle.
\end{multline*}

Let us deal with $I_1$: we observe that
\begin{align*}
I_1&=\left\langle \left(\begin{array}{cc}
T_{\eta(k)}&0\\
0&Id
\end{array}\right)W\phi_k,\left(\begin{array}{cc}
0&Id\\
-T_{\eta(k)}+f'(\varphi)-f'(Q_{\eta(k)})&0
\end{array}\right)W\right\rangle\\
&=\left\langle W,\left(\begin{array}{cc}
T_{\eta(k)}&0\\
0&Id
\end{array}\right)\left(\begin{array}{cc}
0&Id\\
-T_{\eta(k)}&0
\end{array}\right)W\phi_k\right\rangle+\mathrm{O}\left(e^{-4\sigma t}\|W\|_{H^1 \times L^2}^2\right)\\
&=\left\langle W,\left(\begin{array}{cc}
0&T_{\eta(k)}\\
-T_{\eta(k)}&0
\end{array}\right)W\phi_k\right\rangle+\mathrm{O}\left(e^{-4\sigma t}\|W\|_{H^1 \times L^2}^2\right)\\
&=-\int_\RR w_1\partial_x^2w_2\phi_k\;dx+\int_\RR w_2\partial_x^2w_1\phi_k\;dx+\mathrm{O}\left(\frac{1}{\sqrt{t}}\|W\|_{H^1 \times L^2}^2\right)\\
&=\mathrm{O}\left(\frac{1}{\sqrt{t}}\|W\|_{H^1 \times L^2}^2\right).
\end{align*}

We have
\begin{align*}
I_2=&\;\beta_{\eta(k)}\left\langle\left(\begin{array}{cc}
0&-\partial_x\\
\partial_x&0
\end{array}\right)W\phi_k,\left(\begin{array}{cc}
0&Id\\
-T_{\eta(k)}&0
\end{array}\right)W\right\rangle+\mathrm{O}\left(e^{-(e_j+4\sigma)t}\|W\|_{H^1 \times L^2}\right)\\
=&-\beta_{\eta(k)}\int_\RR\partial_xw_2w_2\phi_k\;dx+\beta_{\eta(k)}\int_\RR\partial_xw_1(\partial_x^2w_1-w_1+f'(Q_{\eta(k)})w_1)\phi_k\;dx\\
&+\mathrm{O}\left(e^{-(e_j+4\sigma)t}\|W\|_{H^1 \times L^2}\right)\\
=&-\frac{\beta_{\eta(k)}}{2}\int_\RR w_1^2\partial_xQ_{\beta_{\eta(k)}}f''(Q_{\beta_{\eta(k)}})\phi_k\;dx+\mathrm{O}\left(\frac{1}{\sqrt{t}}\|W\|_{H^1 \times L^2}^2+e^{-(e_j+4\sigma)t}\|W\|_{H^1 \times L^2}\right).
\end{align*}

In addition, we have
\begin{align*}
I_3=&\;A_je^{-e_jt}\left\langle (H_{\eta(k)}W)\phi_k,
JZ_{+,j}+e_jY_{+,j}\right\rangle+\mathrm{O}\left(e^{-(4\sigma+e_j) t}\|W\|_{H^1\times L^2}\right)\\
=&\;A_je^{-e_jt}\left(\left\langle W\phi_k,
-\mathscr{H}_{\eta(k)}Z_{+,j}\right\rangle+e_j\langle W\phi_k,H_{\eta(k)}Y_{+,k}\rangle\right)\\
&+\mathrm{O}\left(\frac{1}{\sqrt{t}}\|W\|_{H^1 \times L^2}^2+e^{-(4\sigma+e_j) t}\|W\|_{H^1\times L^2}\right)\\
=&\mathrm{O}\left(\frac{1}{\sqrt{t}}\|W\|_{H^1 \times L^2}^2+e^{-(4\sigma+e_j) t}\|W\|_{H^1\times L^2}\right).
\end{align*}
Note that the last line of the previous equality is a consequence of the following observation: if $\eta(k)=j$, we have $\mathscr{H}_jZ_{+,j}=e_jZ_{+,j}$ and $H_jY_{+,j}=Z_{+,j}$ so that
$$A_je^{-e_jt}\left(\left\langle W\phi_k,
-\mathscr{H}_{\eta(k)}Z_{+,j}\right\rangle+e_j\langle W\phi_k,H_{\eta(k)}Y_{+,j}\rangle\right)=0.$$
If $\eta(k)\ne j$, we have
$$A_je^{-e_jt}\left(\left\langle W\phi_k,
-\mathscr{H}_{\eta(k)}Z_{+,j}\right\rangle+e_j\langle W\phi_k,H_{\eta(k)}Y_{+,j}\rangle\right)=\mathrm{O}\left(e^{-(e_j+4\sigma)t}\|W\|_{H^1 \times L^2}\right).$$

Gathering the preceding computations yields
$$\left|\frac{d}{dt}\mathcal{F}_{W,k}(t)\right|\le C\left(\frac{1}{\sqrt{t}}\|W\|^2_{H^1\times L^2}+e^{-(e_j+4\sigma)t}\|W\|_{H^1\times L^2}\right),$$
hence the expected claim, by summing on $k$.
\end{proof}

\paragraph{\underline{Step 5}: Control of the directions $\partial_xR_k$}

To obtain a control of the scalar products $\langle W,\partial_xR_k\rangle$ which is more precise than the a priori control by $\|W\|_{H^1\times L^2}$, let us introduce the following modulated variable $\tilde{W}$:

\begin{equation}\label{def_tilde_W}
\tilde{W}(t)=W(t)+\sum_{k=1}^Na_k(t)\partial_xR_k(t),
\end{equation}

where $a_k(t)\in\RR$, $k=1,\dots,N$ are chosen so that for all $l=1,\dots,N$, $\langle \tilde{W}(t),\partial_xR_l(t) \rangle=0$. Existence and uniqueness of the family $(a_k(t))_{k\in\{1,\dots,N\}}$ are justified by the fact that the (interaction) $N\times N$-matrix with generic entry $\langle \partial_xR_k(t),\partial_xR_l(t)\rangle$ is invertible for $t$ large enough.

Notice that \begin{equation}
|a_k(t)|\le C\|W(t)\|_{H^1\times L^2}\le Ce^{-(e_j+\sigma)t}.
\end{equation}

The functional $\mathcal{F}_{\tilde{W}}(t)$, defined as $\mathcal{F}_{W}(t)$ by changing $W$ in $\tilde{W}$, satisfies the following coercivity property:
\begin{equation}\label{coercivity_W_tild}
\|\tilde{W}\|_{H^1\times L^2}^2\le C\left(\mathcal{F}_{\tilde{W}}(t)+\sum_{k=1}^N\left(\langle \tilde{W},Z_{+,k}\rangle^2+\langle \tilde{W},Z_{-,k}\rangle^2\right)\right).
\end{equation}

We have
$$\mathcal{F}_{\tilde{W}}(t)\le \mathcal{F}_{W}(t)+\mathrm{O}\left(e^{-4\sigma t}\|W\|^2_{H^1\times L^2}\right)$$ and we have moreover by Proposition \ref{prop_interaction} and \eqref{def_tilde_W}.

$$\langle\tilde{W},Z_{\pm,k}\rangle^2 \le \alpha_{\pm,k}^2+e^{-2(e_j+5\sigma) t}.$$

\begin{Claim}[Estimate on $\tilde{W}$]\label{claim_tilde_W}
We have $$\forall\;t\in [T(\mathfrak{a}),S_n],\qquad\|\tilde{W}(t)\|^2_{H^1\times L^2}\le\frac{1}{\sqrt{t}}e^{-(2e_j+2\sigma)t}.$$
\end{Claim}

\begin{proof}
Let $t$ belong to $[T(\mathfrak{a}),S_n]$. We obtain by \eqref{coercivity_W_tild} and by integration of \eqref{der_F_W} on $[t,+\infty)$ (which is indeed possible by definition of $T(\mathfrak{a})$) that
$$
\|\tilde{W}(t)\|_{H^1\times L^2}^2\le  \frac{C}{\sqrt{t}}e^{-2(e_j+\sigma)t}+C\sum_{\pm,k}\alpha_{\pm,k}^2+Ce^{-2(e_j+4\sigma)t}.\\
$$
Using the estimate on $\alpha_{\pm,k}$ provided by the definition of $T(\mathfrak{a})$ and Claim \ref{claim_alpha_2}, we then infer: 
$$\|\tilde{W}(t)\|_{H^1\times L^2}^2\le \frac{C}{\sqrt{t}}e^{-2(e_j+\sigma)t}+Ce^{-(2e_j+4\sigma)t}.$$
This concludes the proof of the claim.
\end{proof}

\begin{Claim}[Control of the modulation parameters]\label{claim_mod_param}\label{claim_mod_par}
We have for all $k=1,\dots, N$, $$\forall\;t\in [T(\mathfrak{a}),S_n],\qquad|a_k(t)|\le\frac{C}{t^{\frac{1}{4}}}e^{-(e_j+\sigma)t}.$$
\end{Claim}

\begin{proof}
By definition of the modulation parameters $a_k$, we have $\langle \tilde{W},\partial_xR_k\rangle=0$. Thus, we have by differentiation with respect to $t$:
$$\langle \partial_t\tilde{W},\partial_xR_k\rangle+\langle \tilde{W},\partial_t\partial_xR_k\rangle=0.$$

By Proposition \ref{prop_interaction}, we have for $l\ne k$, $$\langle \partial_xR_l,\partial_xR_k\rangle=\mathrm{O}\left(e^{-4\sigma t}\right)$$ and for all $l$, 
$$\langle\partial_t\partial_xR_l,\partial_xR_k\rangle=\mathrm{O}\left(e^{-4\sigma t}\right).$$
We deduce that $$a_k'(t)\langle \partial_xR_k,\partial_xR_k\rangle+\langle \partial_tW,\partial_xR_k\rangle+\langle \tilde{W},\partial_t\partial_xR_k\rangle=\mathrm{O}\left(e^{-4\sigma t}\|W(t)\|_{H^1\times L^2}\right).$$

We have in addition
$$\left\langle W,\left(\begin{array}{cc}
0&\partial_x^2-Id+f'(\varphi)\\
Id&0
\end{array}\right)\partial_xR_k\right\rangle+\langle \tilde{W},\partial_{tx}R_k\rangle=\mathrm{O}\left(\|\tilde{W}\|_{H^1\times L^2}\right).$$

What is more,
$$\left|A_je^{-e_jt}e_j\langle Y_{+,j}, \partial_xR_k\rangle\right|\le Ce^{-(e_j+4\sigma)t},$$ again by Proposition \ref{prop_interaction}.

Hence, \begin{align*}
|a_k'(t)|&\le C\|\tilde{W}(t)\|_{H^1\times L^2}+Ce^{-(e_j+4\sigma)t}\\
&\le \frac{C}{t^{\frac{1}{4}}}e^{-(e_j+\sigma)t}+e^{-(e_j+3\sigma)t}.
\end{align*}

\end{proof}

Now, gathering \eqref{def_tilde_W}, Claim \ref{claim_tilde_W}, and Claim \ref{claim_mod_param}, we immediately deduce the expected estimate of $\|W\|_{H^1\times L^2}$, which ends the proof of Lemma \ref{lem_est_W}.

\subsubsection{Control of the unstable directions for $k>j$ and end of the proof}\label{subsect_topological_arg}

To control $\alpha_-=\left(\alpha_{-,k}\right)_{j<k\le N}$ and eventually obtain the following statement, we resort to a classical topological argument, already set up in \cite{combetgKdV} and initially developed by Côte, Martel and Merle \cite{cmm}. \\

\begin{Lem}\label{lem_existence_a}
For $t_0$ large enough, there exists $\mathfrak{a}\in B_{\RR^{N-j}}(e^{-(e_j+2\sigma)S_n})$ such that $T(\mathfrak{a})=t_0$.
\end{Lem}

The proof follows that of Combet \cite{combetgKdV}. We write it below for the sake of completeness.

\begin{proof}

We first choose $t_0$ sufficiently large such that $\frac{K_0}{\sqrt{t_0}}\le \frac{1}{2}$. Then, we have by  Lemma \ref{lem_est_W}
$$\forall\;t\in [T(\mathfrak{a}),S_n],\qquad \|W(t)\|_{H^1\times L^2}\le \frac{1}{2}e^{-(e_j+\sigma)t}.$$

Assume, for the sake of contradiction, that for all $\mathfrak{a}\in B_{\RR^{N-j}}(e^{-(e_j+2\sigma)S_n})$, $T(\mathfrak{a})>t_0$. As $\|W(T(\mathfrak{a}))\|_{H^1\times L^2}\le \frac{1}{2}e^{-(e_j+\sigma)T(\mathfrak{a})}$, by definition of $T(\mathfrak{a})$ and continuity of the flow, we have:
$$|\alpha_-(T(\mathfrak{a}))|=1.$$ (We recall that $\alpha_-(t)=(\alpha_{-,k}(t))_{j<k\le N}$.)
In other words, the map
$$\begin{array}{cccl}
\mathcal{M}:& B_{\RR^{N-j}}(e^{-(e_j+2\sigma)S_n})&\longrightarrow& S_{\RR^{N-j}}(e^{-(e_j+2\sigma)S_n})\\
&\mathfrak{a}&\longmapsto&e^{-(e_j+2\sigma)(S_n-T(\mathfrak{a}))}\alpha_-(T(\mathfrak{a}))
\end{array}$$
is well-defined. Now, we aim at showing that $\mathcal{M}$ is continuous and that its restriction to $S_{\RR^{N-j}}(e^{-(e_j+2\gamma)S_n})$ is the identity. \\

Let $T\in [T_0,T(\mathfrak{a})]$ be such that $W$ is defined on $[T,S_n]$ and, by continuity,
$$\forall\;t\in[T,S_n],\qquad \|W(t)\|_{H^1\times L^2}\le 1.$$

We consider, for all $t\in [T,S_n]$:
$$\mathcal{N}(t):=\mathcal{N}(\alpha_-(t))=\left\|e^{(e_j+2\sigma)t}\alpha_-(t)\right\|^2.$$

\begin{Claim}\label{claim_der_N}
For $t_0$ large enough, and for all $t\in [T,S_n]$ such that $\mathcal{N}(t)=1$, we have:
$$\mathcal{N}'(t)\le -(e_{j+1}-e_j-2\sigma).$$
\end{Claim}

\begin{proof}[Proof of Claim \ref{claim_der_N}]
Let us start from estimate \eqref{est_alpha-k}: for all $k\in\{j+1,\dots,N\}$, for all $t\in[T,S_n]$, 
$$\left|\frac{d}{dt}\alpha_{-,k}+e_k\alpha_{-,k}\right|\le Ce^{-(e_j+4\sigma)t}.$$
Thus we obtain for all $k\in\{j+1,\dots,N\}$, 
$$\alpha_{-,k}\frac{d}{dt}\alpha_{-,k}+e_{j+1}\alpha_{-,k}^2\le \alpha_{-,k}\frac{d}{dt}\alpha_{-,k}+e_k\alpha_{-,k}^2\le Ce^{-(e_j+4\sigma)t}|\alpha_{-,k}|.$$
Then, summing on $k\in\{j+1,\dots,N\}$ leads to
$$\left(|\alpha_-(t)|^2\right)'+2e_{j+1}|\alpha_-(t)|^2\le Ce^{-(e_j+4\sigma)t}|\alpha_-(t)|.$$
Therefore we can estimate:
\begin{align*}
\mathcal{N}'(t)&=e^{2(e_j+2\sigma)t}\left[2(e_j+2\sigma)|\alpha_{-}(t)|^2+\left(|\alpha_-(t)|^2\right)'\right]\\
&\le e^{2(e_j+2\sigma)t}\left[2(e_j+2\sigma)|\alpha_-(t)|^2-2e_{j+1}|\alpha_-(t)|^2+Ce^{-(e_j+4\sigma)t}|\alpha_-(t)|\right].
\end{align*}
Hence we have for all $t\in|T,S_n]$,
$$\mathcal{N}'(t)\le -\theta\mathcal{N}(t)+Ce^{e_jt}|\alpha_-(t)|,$$
where $\theta=2(e_{j+1}-e_j-2\sigma)>0$ by definition of $\sigma$. In particular, for all $\tau\in [T,S_n]$ satisfying $\mathcal{N}(\tau)=1$, we have:
$$\mathcal{N}'(\tau)\le -\theta+Ce^{e_j\tau}|\alpha_-(\tau)|\le -\theta+Ce^{e_j\tau}e^{-(e_j+2\sigma)\tau}\le -\theta+Ce^{-2\sigma t_0}.$$
Now, we fix $t_0$ large enough such that $Ce^{-2\sigma t_0}\le \frac{\theta}{2}$. Thus for all $\tau\in[T,S_n]$ such that $\mathcal{N}(\tau)=1$, we have
$$\mathcal{N}'(\tau)\le -\frac{\theta}{2}.$$
\end{proof}

Finally, we claim that $\mathfrak{a}\mapsto T(\mathfrak{a})$ is continuous. Indeed, let $\varepsilon>0$. By definition of $T(\mathfrak{a})$ and by Claim \ref{claim_der_N}, there exists $\delta>0$ such that for all $t\in [T(\mathfrak{a})+\varepsilon,S_n]$, $\mathcal{N}(t)<1-\delta$, and such that $\mathcal{N}(T(\mathfrak{a})-\varepsilon)>1+\delta$. But from continuity of the flow, there exists $\eta>0$ such that for all $\tilde{\mathfrak{a}}$ satisfying $\|\tilde{\mathfrak{a}}-\mathfrak{a}\|\le \eta$, we have
$$\forall\;t\in [T(\mathfrak{a})-\varepsilon,S_n],\qquad |\mathcal{N}(\tilde{\mathfrak{a}})-\mathcal{N}(\mathfrak{a})|\le\frac{\delta}{2}.$$
We finally deduce that $$T(\mathfrak{a})-\varepsilon\le T(\tilde{\mathfrak{a}})\le T(\mathfrak{a})+\varepsilon.$$ Hence, $\mathfrak{a}\mapsto T(\mathfrak{a})$ is continuous. \\

We then obtain that the map $\mathcal{M}$ is continuous. 
What is more, for $\mathfrak{a}\in S_{\RR^{N-j}}(e^{-(e_j+2\sigma)S_n})$, as $\mathcal{N}'(S_n)\le -(e_{j+1}-e_j-2\sigma)<0$, we then deduce by definition of $T(\mathfrak{a})$ that $T(\mathfrak{a})=S_n$, and thus, $\mathcal{M}(\mathfrak{a})=\mathfrak{a}$. \\
The existence of such a map $\mathcal{M}$ contradicts Brouwer's fixed point theorem. Thus, we have finished proving Lemma \ref{lem_existence_a}.

\end{proof}

\section{Classification under condition of the multi-solitons of (NLKG)}
 
Let $N\ge 2$ and $x_1,\dots,x_N$, $\beta_1,\dots,\beta_N$ be $2N$ parameters as in Theorem \ref{th_main_N}.
Let $U$ be a solution of \eqref{NLKG} such that
\begin{equation}
\left\|U(t)-\sum_{i=1}^NR_{\beta_i}(t)\right\|_{H^1\times L^2}=\mathrm{O}\left(\frac{1}{t^\alpha}\right)\qquad\text{as}\quad t\to+\infty
\end{equation} for some $\alpha>3$.

The goal of this section is to prove the existence of $A_1,\dots, A_N\in\RR$ such that $$U=\Phi_{A_1,\dots,A_N}.$$ Here again, we make the proof for $d=1$.

We denote by $\varphi$ a multi-soliton solution associated with these parameters, satisfying \eqref{est_multi_sol} and $\Phi:=\dbinom{\varphi}{\partial_t\varphi}$.
Let us consider $Z:=U-\Phi=\dbinom{z}{\partial_tz}$. Obviously,
$$\|Z(t)\|_{H^1\times L^2}=\mathrm{O}\left(\frac{1}{t^\alpha}\right),\qquad\text{as}\quad t\to +\infty.$$

Our first objective is to improve this comparison, and namely to pass from the polynomial decay to an exponential one.

\subsection{Exponential convergence to 0 at speed $e_1$ of $\|Z(t)\|_{H^1\times L^2}$}\label{subsec_exp_conv}

\subsubsection{Introduction of a new variable by modulation}

In a standard way, we modulate the variable $Z$ in order to obtain suitable orthogonality properties, making it possible to obtain crucial estimates when we apply the spectral theory available for \eqref{NLKG}.

\begin{Lem}\label{lem_mod_E}
There exists $t_0>0$ and $\mathscr{C}^1$ functions $a_i:[t_0,+\infty)\to\RR$ and $b_i:[t_0,+\infty)\to\RR$ for all $i=1,\dots,N$  such that, defining
$$E:=Z-\sum_{i=1}^Na_i\partial_xR_i-\sum_{i=1}^Nb_iY_{+,i},$$
we have for all $i=1,\dots,N$ and for all $t\ge t_0$:
\begin{align}
\left\langle E(t),\partial_xR_i(t)\right\rangle&=0 \label{aux_E_1}\\
\left\langle E(t),Z_{-,i}(t)\right\rangle &=0.  \label{aux_E_2}
\end{align}
Moreover, we have for all $i=1,\dots,N$:
\begin{align}
a_i(t)&=\frac{1}{\|\partial_xR_i\|^2}\langle Z(t),\partial_xR_i(t)\rangle+\mathrm{O}\left(e^{-4\sigma t}\|Z\|_{H^1\times L^2}\right)\\
b_i(t)&=\langle Z(t),Z_{-,i}(t)\rangle+\mathrm{O}\left(e^{-4\sigma t}\|Z\|_{H^1\times L^2}\right).
\end{align}
\end{Lem}

\begin{proof}
This lemma follows from the consideration of the system with unknown variables $a_i$ and $b_i$ which is obtained by replacing $E$ by its definition in \eqref{aux_E_1} and \eqref{aux_E_2}. See also \cite{cf} for similar considerations in the case of modulation for the nonlinear Schrödinger equations. 
\end{proof}

\subsubsection{Control of the $Z_{+,i}$ and $Z_{-,i}$ directions}

Define $\alpha_{\pm,i}:=\langle Z,Z_{\pm,i}\rangle$ for all $i=1,\dots,N$. We claim:

\begin{Lem}\label{lem_control_Z_pm_i}
The following bounds hold: for all $i=1,\dots,N$, for all $t\ge t_0$,
$$|\alpha_{\pm,i}'(t)\mp e_i\alpha_{\pm,i}(t)|\le C\left(e^{-4\sigma t}\|Z\|_{H^1\times L^2}+\|Z\|^2_{H^1\times L^2}\right).$$
\end{Lem}

\begin{proof}
The proof is in a similar fashion as that of Claim \ref{claim_alpha}. We note that
$$\partial_tZ=\left(\begin{array}{cc}
0&Id\\
\partial_x^2-Id+f'(\varphi)&0
\end{array}\right)Z+\dbinom{0}{f(u)-f(\varphi)-(u-\varphi)f'(\varphi)}$$
and that $|f(u)-f(\varphi)-(u-\varphi)f'(\varphi)|\le C|u-\varphi|^2\le C\|Z\|^2_{H^1\times L^2}$.
Thus for $i=1,\dots,N$, we have
\begin{align*}
\alpha_{\pm,i}'&=\langle\partial_tZ,Z_{\pm,i}\rangle+\langle Z,\partial_tZ_{\pm,i}\rangle\\
=&\;\left\langle \left(\begin{array}{cc}
0&Id\\
\partial_x^2-Id+f'(\varphi)&0
\end{array}\right)Z,Z_{\pm,i}\right\rangle-\beta_i\langle Z,\partial_xZ_{\pm,i}\rangle+\mathrm{O}\left(\|Z\|^2_{H^1\times L^2}\right)\\
=&\;\left\langle Z,\left(\begin{array}{cc}
-\beta_i\partial_x&\partial_x^2-Id+f'(Q_i)\\
Id&-\beta_i\partial_x
\end{array}\right)Z_{\pm,i}\right\rangle+
\left\langle Z,\left(\begin{array}{cc}
0&f'(\varphi)-f'(Q_i)\\
0&0
\end{array}\right)Z_{\pm,i}\right\rangle\\
&+\mathrm{O}\left(\|Z\|^2_{H^1\times L^2}\right)\\
=&\;\langle Z,\pm e_iZ_{\pm,i}\rangle+\mathrm{O}\left(e^{-4\sigma t}\|Z\|_{H^1\times L^2}+\|Z\|^2_{H^1\times L^2}\right).
\end{align*}
\end{proof}

\subsubsection{Control of the remaining modulation parameters}

\begin{Lem}\label{lem_control_modulation_parameters}
For all $i=1,\dots,N$, we have
\begin{equation}
|a_i'|\le C\left(\|E\|_{H^1\times L^2}+\|Z\|^2_{H^1\times L^2}\right).
\end{equation}
\end{Lem}

\begin{proof}
We do not detail the proof of this lemma which is similar to Claim \ref{claim_mod_par}. It suffices to start by differentiating the orthogonality relation $\langle E,\left(R_i\right)_x\rangle=0$ with respect to $t$ and then to control terms by means of $\|E\|_{H^1\times L^2}$.
\end{proof}

\subsubsection{Study of a Lyapunov functional}\label{subsect_alm_mono_prop}

Taking some inspiration in \cite{mm_waveeq,Yuan_2019,xu_yuan}, we consider for all $t\ge t_0$:
$$\mathcal{F}_z(t):=\int_\RR\left\{\partial_xz^2+\partial_tz^2+z^2-f'(\varphi)z^2\right\}\;dx+2\int_\RR\chi\partial_xz\partial_tz\;dx,$$
where $\chi$ is defined as follows.

To begin with, recall that the parameters are ordered in such a way: $-1<\beta_{\eta(1)}<\dots<\beta_{\eta(N)}<1$; let us denote, for some small $\delta>0$ which will be determined later:
\begin{align*}
\bar{l}_i&:=\beta_{\eta(i)}+\delta\left(\beta_{\eta(i+1)}-\beta_{\eta(i)}\right) \\
\underline{l}_i&:=\beta_{\eta(i)}-\delta\left(\beta_{\eta(i+1)}-\beta_{\eta(i)}\right). 
\end{align*}
We then define for all $t\ge t_0$ and for all $x\in\RR$:
$$\chi(t,x):=\begin{dcases}
\beta_{\eta(1)}&\text{if}\quad x\in (-\infty,\bar{l}_1t]\\
\beta_{\eta(i)}&\text{if}\quad  x\in [\underline{l}_it,\bar{l}_it]\\
\beta_{\eta(N)}&\text{if}\quad  x\in [\bar{l}_Nt,+\infty)\\
\frac{x}{(1-2\delta)t}-\frac{\delta}{1-2\delta}\left(\beta_{\eta(i)}+\beta_{\eta(i+1)}\right)&\text{if}\quad  x\in [\bar{l}_it,\underline{l}_{i+1}t], i\in\{1,\dots,N-1\}.
\end{dcases}$$
For all $t\ge t_0$, $\chi(t)$ is a piecewise $\mathscr{C}^1$ function.

Set $\Omega(t):=\bigcup_{i=1}^N(\bar{l}_it,\underline{l}_{i+1}t)$. It follows from the definition of $\chi$ that
\begin{align*}
\partial_t\chi(t,x)&=\partial_x\chi(t,x)=0 \quad \text{if } x\in\Omega(t)^\complement
\\
\partial_x\chi(t,x)&=\frac{1}{(1-2\delta)t},\quad \partial_t\chi(t,x)=-\frac{x}{(1-2\delta)t^2}\quad \text{if } x\in\Omega(t).
\end{align*}

\begin{Lem}
There exists $\gamma>0$ such that
\begin{multline}\label{est_der_Fz}
\mathcal{F}'_z(t)=2\int_{\Omega(t)}\partial_xz\partial_tz\partial_t\chi\;dx-\int_{\Omega(t)}\left\{(\partial_tz)^2+(\partial_xz)^2-z^2+f'(\varphi)z^2\right\}\partial_x\chi\;dx\\
+\mathrm{O}\left(e^{-\gamma t}\|Z\|^2_{H^1\times L^2}+\|Z\|^3_{H^1\times L^2}\right).
 \end{multline}
\end{Lem}

\begin{proof}
We essentially have to use the identity $\partial_t^2z=\partial_x^2z-z+f(u)-f(\varphi)$ in the expression of $\mathcal{F}'_z(t)$.\\
We compute 
\begin{equation}\label{eq_aux_1}
\begin{aligned}
\mathcal{F}'_z(t)=&\;2\int_\RR\left\{z_{tx}z_x+z_{tt}z_t+z_tz-f'(\varphi)z_tz-\frac{1}{2}\varphi_tf''(\varphi)z^2\right\}\;dx\\
&+2\int_\RR\chi_tz_xz_t\;dx+2\int_\RR\chi z_xz_{tt}\;dx+2\int_\RR\chi z_{xt}z_t\;dx\\
=&\;2\int_\RR z_t\left(-z_{xx}+z_{tt}+z-f'(\varphi)z\right)\;dx+2\int_\RR\chi_t z_xz_t\;dx\\
&+2\int_\RR\chi z_x\left(z_{xx}-z+f(u)-f(\varphi)\right)\;dx-\int_\RR z_t^2\chi_x\;dx-\int_\RR \varphi_tf''(\varphi)z^2\;dx
\end{aligned}
\end{equation}
Notice that 
\begin{equation}\label{eq_aux_2}
\int_\RR z_t\left(-z_{xx}+z_{tt}+z-f'(\varphi)z\right)\;dx=\int_\RR z_t\left(f(u)-f(\varphi)-f'(\varphi)z\right)\;dx=\mathrm{O}\left(\|Z\|_{H^1\times L^2}^3\right)
\end{equation}
and 
\begin{equation}\label{eq_aux_3}
\int_\RR\chi z_x\left(f(u)-f(\varphi)\right)\;dx=\int_\RR\chi z_xf'(\varphi)z\;dx+\mathrm{O}\left(\|Z\|_{H^1\times L^2}^3\right).
\end{equation} Hence, collecting \eqref{eq_aux_1}, \eqref{eq_aux_2}, and \eqref{eq_aux_3},

\begin{equation}\label{est_der_Fz_chi}
\begin{aligned}
\mathcal{F}'_z(t)=&\;2\int_\RR\chi_t z_xz_t\;dx-\int_\RR\chi_x(z_x^2+z_t^2-z^2)\;dx\\
&-\int_\RR z^2\left(\chi_xf'(\varphi)+\chi\varphi_xf''(\varphi)\right)\;dx-\int_\RR\varphi_tf''(\varphi)z^2\;dx+\mathrm{O}\left(\|Z\|_{H^1\times L^2}^3\right)\\
=&\;2\int_\RR z_xz_t\chi_t\;dx-\int_\RR\left\{z_x^2+z_t^2-z^2+f'(\varphi)z^2\right\}\chi_x\;dx\\
&-\int_\RR z^2f''(\varphi)(\varphi_t+\chi\varphi_x)\;dx+\mathrm{O}\left(\|Z\|_{H^1\times L^2}^3\right).
\end{aligned}
\end{equation}

Lastly, observe that
\begin{align*}
\int_\RR z^2f''(\varphi)(\varphi_t+\chi\varphi_x)\;dx&=\sum_{i=1}^N\int_\RR z^2f''(\varphi)\left((R_k)_t+\chi(R_k)_x\right)\;dx+\mathrm{O}\left(\|Z\|_{H^1\times L^2}^3\right)\\
&=I+J+\mathrm{O}\left(\|Z\|_{H^1\times L^2}^3\right),
\end{align*}
where
$$\begin{dcases}
I=\sum_{k=1}^N\int_{\Omega(t)}z^2f''(\varphi)\left((R_k)_t+\chi(R_k)_x\right)\;dx\\
J=\sum_{k=1}^N\sum_{i=1}^N\int_{\underline{l}_i(t)}^{\bar{l}_it}z^2f''(\varphi)\left((R_k)_t+\chi(R_k)_x\right)\;dx.
\end{dcases}$$

On the one hand,
\begin{equation}\label{est_J}
\begin{aligned}
J&=\sum_{i=1}^N\sum_{k=1}^N\int_{\underline{l}_i(t)}^{\bar{l}_it}z^2f''(\varphi)\left((R_k)_t+\chi(R_k)_x\right)\;dx\\
&=\sum_{i=1}^N\int_{\underline{l}_i(t)}^{\bar{l}_it}z^2f''(\varphi)\left((R_{\eta(i)})_t+\chi(R_{\eta(i)})_x\right)\;dx+\mathrm{O}\left(e^{-4\sigma t}\|Z\|_{H^1\times L^2}^2\right)\\
&=\mathrm{O}\left(e^{-4\sigma t}\|Z\|_{H^1\times L^2}^2\right).
\end{aligned}
\end{equation}
Indeed, for all $x\in [\underline{l}_i(t),\bar{l}_it]$, we have $(R_{\eta(i)})_t(t,x)+\chi(t,x)(R_{\eta(i)})_x(t,x)=0$.

On the other, for $x\in\Omega(t)$, there exists $k'\in\{1,\dots,N\}$ such that $\bar{l}_{k'}t\le x\le \underline{l}_{k'+1}t$. \\
Then $$(\bar{l}_{k'}-\beta_k)t\le x-\beta_kt\le (\underline{l}_{k'+1}-\beta_k)t$$ and thus
$$|x-\beta_kt|\ge \delta\min_{j=1,\dots,N-1}\{\beta_{\eta(j+1)}-\beta_{\eta(j)}\}t>0.$$
As a consequence, we have for all $x\in\Omega(t)$
\begin{equation}\label{est_dec_Rk_Omega}
\begin{aligned}
\left|(R_k)_x(t,x)\right|&\le Ce^{-\sigma|x-\beta_kt|}\\
&\le Ce^{-\gamma t},
\end{aligned}
\end{equation}
where $0<\gamma<\sigma\delta\min_{j=1,\dots,N-1}\{\beta_{\eta(j+1)}-\beta_{\eta(j)}\}$.

Hence 
\begin{equation}\label{est_I}
|I|\le Ce^{-\gamma t}\|Z\|_{H^1\times L^2}^2.
\end{equation}

Finally we gather \eqref{est_der_Fz_chi}, \eqref{est_J}, and \eqref{est_I} in order to obtain \eqref{est_der_Fz}.
\end{proof}

Let us introduce the components $\varepsilon$ and $\varepsilon_2$ of the vector 
$$E=\dbinom{\varepsilon:=z-\sum_i\left\{a_i\partial_xQ_i+b_i(Y_{+,i})_1\right\}}{\varepsilon_2:=\partial_tz-\sum_i\left\{a_i(-\beta_i\partial_{xx}Q_i)+b_i(Y_{+,i})_2\right\}}.$$

\begin{Cor}
We have
\begin{multline}
\mathcal{F}'_z(t)=2\int_{\Omega(t)}\varepsilon_x\varepsilon_2\chi_t\;dx-\int_{\Omega(t)}\left\{\varepsilon_2^2+\varepsilon_x^2-\varepsilon^2+f'(\varphi)\varepsilon^2\right\}\chi_x\;dx\\
+\mathrm{O}\left(e^{-\gamma t}\|Z\|_{H^1\times L^2}^2+\|Z\|_{H^1\times L^2}^3\right).\end{multline}
\end{Cor}

\begin{proof}
The corollary immediately follows from \eqref{est_der_Fz} and bounds for the derivatives of $R_k$ and $Y_{+,k}$ which are analogous to \eqref{est_dec_Rk_Omega}.
\end{proof}

 In the spirit of \cite[Proposition 4.2]{mm_waveeq}, we will state an almost monotonicity property satisfied by $\mathcal{F}_z$. Let us define
$$\mathcal{F}_{\varepsilon,\Omega(t)}(t):=\int_{\Omega(t)}\left\{\varepsilon_x^2+\varepsilon_2^2+2\chi\varepsilon_x\varepsilon_2\right\}\;dx.$$

Let $\lambda\in (1,\alpha-1)$. The choice of $\lambda$ is linked with the integrability of particular quantities and will appear naturally later.

\begin{proposition}\label{prop_der_F_z_F_E}
There exists $\delta>0$ and $t_0>0$ such that for all $t\ge t_0$, 
\begin{equation}
-\mathcal{F}'_z(t)\le \frac{\lambda}{t}\mathcal{F}_{\varepsilon,\Omega(t)}(t)+\mathrm{O}\left(e^{-\gamma t}\|Z\|_{H^1\times L^2}^2+\|Z\|_{H^1\times L^2}^3\right).
\end{equation}
\end{proposition}

\begin{proof}
Since $f'(0)=0$ and $f'$ is continuous, there exists $r_0>0$ such that for all $r\in[0,r_0]$, $|f'(r)|\le 1$. There exists $K>0$ (independent of $t$) such that for all $t\ge t_0$ and for all $x\in\Omega(t)$, $|\varphi(t,x)|\le Ke^{-\gamma t}.$
Even if it means increasing $t_0$, we can assume that $Ke^{-\gamma t}\le r_0$ for all $t\ge t_0$.

In addition $\chi_x\ge 0$. Thus, for $t\ge t_0$,
\begin{align*}
\MoveEqLeft[2]
-2\int_{\Omega(t)}\varepsilon_x\varepsilon_2\chi_t\;dx+\int_{\Omega(t)}\left\{\varepsilon_x^2+\varepsilon_2^2-\varepsilon^2+f'(\varphi)\varepsilon^2\right\}\chi_x\;dx\\
&\le -2\int_{\Omega(t)}\varepsilon_x\varepsilon_2\chi_t\;dx+\int_{\Omega(t)}\left\{\varepsilon_x^2+ \varepsilon_2^2\right\}\chi_x\;dx.
\end{align*}
Moreover, 
\begin{align*}
-2\int_{\Omega(t)}\varepsilon_x\varepsilon_2\chi_t\;dx+\int_{\Omega(t)}\left\{\varepsilon_x^2+ \varepsilon_2^2\right\}\chi_x\;dx&=\frac{1}{(1-2\delta)t}\int_{\Omega(t)}\left\{2\varepsilon_x\varepsilon_2\frac{x}{t}+\varepsilon_x^2+\varepsilon_2^2\right\}\;dx \\
&=\frac{1}{(1-2\delta)t}\left(\mathcal{F}_{\varepsilon,\Omega(t)}+2\int_{\Omega(t)}\left(\frac{x}{t}-\chi\right)\varepsilon_x\varepsilon_2\;dx\right).
\end{align*}
Now, for $x\in\Omega(t)$, we have
\begin{align*}
\left|\frac{x}{t}-\chi(t,x)\right|&\le \left|\frac{x}{t}\left(1-\frac{1}{1-2\delta}\right)\right|+\frac{\delta}{1-2\delta}\times 2\max_i|\beta_i| \\
&\le 2\delta\left(\left|\frac{x}{t}\right|+C\right)\\
&\le C\delta.
\end{align*}

Thus, $$\frac{2}{(1-2\delta)t}\int_{\Omega(t)}\left(\frac{x}{t}-\chi\right)\varepsilon_x\varepsilon_2\;dx=\mathrm{O}\left(\delta\|E\|_{H^1\times L^2}^2\right).$$
Noticing moreover that 
\begin{align*}
\mathcal{F}_{\varepsilon,\Omega(t)}&\ge \int_{\Omega(t)}\left\{\varepsilon_x^2+\varepsilon_2^2\right\}\;dx-2\|\chi(t)\|_{\infty}\int_{\Omega(t)}\varepsilon_x\varepsilon_2\;dx\\
&\ge  \int_{\Omega(t)}\left\{\varepsilon_x^2+\varepsilon_2^2\right\}\;dx-\|\chi(t)\|_{\infty} \int_{\Omega(t)}\left\{\varepsilon_x^2+\varepsilon_2^2\right\}\;dx\\
&\ge \left(1-\|\chi(t)\|_{\infty}\right)\|E\|_{H^1\times L^2(\Omega(t))}^2,
\end{align*}
we obtain 
$$\frac{2}{(1-2\delta)t}\int_{\Omega(t)}\left(\frac{x}{t}-\chi\right)\varepsilon_x\varepsilon_2\;dx=\mathrm{O}\left(\delta \mathcal{F}_{\varepsilon,\Omega(t)}\right).$$
Finally,
\begin{align*}
-\mathcal{F}_z'(t)&\le \frac{1}{(1-2\delta)t}(1+\mathrm{O}(\delta))\mathcal{F}_{\varepsilon,\Omega(t)}+\mathrm{O}\left(e^{-\gamma t}\|Z\|_{H^1\times L^2}^2+\|Z\|_{H^1\times L^2}^3\right)\\
&\le \frac{\lambda}{t}\mathcal{F}_{\varepsilon,\Omega(t)}+\mathrm{O}\left(e^{-\gamma t}\|Z\|_{H^1\times L^2}^2+\|Z\|_{H^1\times L^2}^3\right),
\end{align*}
provided $\delta$ is chosen small enough.
\end{proof}

We now introduce the functional 
$$\mathcal{F}_{\varepsilon}(t):=\int_\RR\left\{\varepsilon_x^2+\varepsilon_2^2+\varepsilon^2-f'(\varphi)\varepsilon^2\right\}\;dx+2\int_\RR\chi\varepsilon_x\varepsilon_2\;dx$$
and compare it with $\mathcal{F}_z$.

\begin{Lem}\label{lem_functional_F_E}
We have $$\mathcal{F}_\varepsilon(t)=\mathcal{F}_z(t)-2\sum_{i=1}^N\alpha_{-,i}(t)\alpha_{+,i}(t)+\mathcal{G}(t),$$
where $\mathcal{G}(t)=\mathrm{O}\left(e^{-\gamma t}\|Z\|_{H^1\times L^2}^2\right)$ and $\mathcal{G}'(t)=\mathrm{O}\left(e^{-\gamma t}\|Z\|_{H^1\times L^2}^2\right)$.
\end{Lem}

\begin{proof}
Let $M$ be the matrix $\left(\begin{array}{cc}
-\partial_x^2+Id-f'(\varphi)&0\\
0&Id
\end{array}\right)$. For all $i=1,\dots,N$, we have the decomposition
\begin{equation}
M=H_i+\left(\begin{array}{cc}
0&\beta_i\partial_x\\
-\beta_i\partial_x&0
\end{array}\right)+\left(\begin{array}{cc}
f'(Q_i)-f'(\varphi)&0\\
0&0
\end{array}\right).
\end{equation}
Then we infer
\begin{align*}
\mathcal{F}_\varepsilon(t)=&\;\langle ME,E\rangle+2\int_\RR\chi\varepsilon_x\varepsilon_2\;dx\\
=&\;\langle MZ,Z\rangle-2\left\langle M\left(\sum_{i=1}^N\left\{a_i(R_i)_x+b_iY_{+,i}\right\}\right),Z\right\rangle\\
&+\left\langle M\left(\sum_{i=1}^N\left\{a_i(R_i)_x+b_iY_{+,i}\right\}\right),\sum_{j=1}^N\left\{a_j(R_j)_x+b_jY_{+,j}\right\}\right\rangle\\
&+2\int_\RR\chi z_xz_t\;dx+2\int_\RR\chi\sum_{i=1}^N\left\{a_i(Q_i)_{xx}+b_i(Y_{+,i})_{1,x}\right\}\sum_{j=1}^N\left\{a_j(Q_j)_{xt}+b_j(Y_{+,j})_2\right\}\;dx\\
&-2\int_\RR\chi z_x\sum_{i=1}^N\left\{a_i(Q_i)_{xt}+b_i(Y_{+,i})_2\right\}\;dx-2\int_\RR\chi z_t\sum_{i=1}^N\left\{a_i(Q_i)_{xx}+b_i(Y_{+,i})_{1,x}\right\}\;dx\\
=&\; \mathcal{F}_z(t)-2\sum_{i=1}^Nb_i\langle Z_{+,i},Z\rangle +\tilde{\mathcal{G}}(t)\\
=&\; \mathcal{F}_z(t)-2\sum_{i=1}^N\alpha_{-,i}(t)\alpha_{+,i}(t)+\mathcal{G}(t),
\end{align*}
with 
\begin{align*}
MY_{+,i}&=Z_{+,i}+\beta_i\dbinom{(Y_{+,i})_{2,x}}{-(Y_{+,i})_{1,x}}+\dbinom{(f'(Q_i)-f'(\varphi))(Y_{+,i})_1}{0}\\
M(R_i)_x&=\dbinom{-\beta_i^2(Q_i)_{xxx}}{-\beta_i(Q_i)_{xx}}+\dbinom{(f'(Q_i)-f'(\varphi))(Q_i)_x}{0}.
\end{align*}
and
\begin{align*}
\tilde{\mathcal{G}}(t)=&-2\left\langle \sum_{i=1}^N\left\{a_i\dbinom{-\beta_i^2(Q_i)_{xxx}}{\beta_i(Q_i)_{xx}}+b_i\beta_i\dbinom{(Y_{+,i})_{2,x}}{-(Y_{+,i})_{1,x}}\right\},Z\right\rangle+\mathrm{O}\left(e^{-\gamma t}\|Z\|_{H^1\times L^2}^2\right)\\
&+\sum_{i=1}^N\left\langle a_i\dbinom{-\beta_i^2(Q_i)_{xxx}}{\beta_i(Q_i)_{xx}}+b_i\beta_i\dbinom{(Y_{+,i})_{2,x}}{-(Y_{+,i})_{1,x}}+b_iZ_{+,i},a_i(R_i)_x+b_iY_{+,i}\right\rangle\\
&+2\sum_{i=1}^N\int_{\RR}\chi\left(a_i(Q_i)_{xx}+b_i(Y_{+,i})_{1,x}\right)\left(-\beta_ia_i(Q_i)_{xx}+b_i(Y_{+,i})_{2}\right)\;dx\\
&-2\sum_{i=1}^N\int_\RR \chi z_x\left(-\beta_ia_i(Q_i)_{xx}+b_i(Y_{+,i})_2\right)\;dx-2\sum_{i=1}^N\int_\RR \chi z_t\left(a_i(Q_i)_{xx}+b_i(Y_{+,i})_{1,x}\right)\;dx.
\end{align*}
Integrating by parts, we obtain
\begin{align*}
\MoveEqLeft[4]
\int_\RR\chi z_x\left(-\beta_ia_i(Q_i)_{xx}+b_i(Y_{+,i})_2\right)\;dx\\
&=-\int_\RR \chi z\left(-\beta_ia_i(Q_i)_{xxx}+b_i(Y_{+,i})_{2,x}\right)\;dx-\int_\RR\chi_xz\left(-\beta_ia_i(Q_i)_{xx}+b_i(Y_{+,i})_2\right)\;dx.
\end{align*}
Consequently,
\begin{align*}
\tilde{\mathcal{G}}(t)=&-2\sum_{i=1}^Na_i\int_\RR(\chi-\beta_i)\left(-\beta_iz(Q_i)_{xxx}+z_t(Q_i)_{xx}\right)\;dx\\
&-2\sum_{i=1}^Nb_i\int_\RR(\chi-\beta_i)\left(-z(Y_{+,i})_{2,x}+z_t(Y_{+,i})_{1,x}\right)\;dx\\
&+2\sum_{i=1}^Na_i^2\beta_i\int_\RR(\beta_i-\chi)(Q_i)_{xx}^2\;dx+2\sum_{i=1}^Nb_i^2\int_\RR(\beta_i-\chi)(Y_{+,i})_1(Y_{+,i})_{2,x}\\
&+2\sum_{i=1}^N\int_{\Omega(t)}\chi_xz\left(-\beta_ia_i(Q_i)_{xx}+b_i(Y_{+,i})_2\right)\;dx\\
&+2\sum_{i=1}^Na_ib_i\int_\RR(\chi-\beta_i)\left(-\beta_i(Y_{+,i})_{1,x}(Q_i)_{xx}+(Q_i)_{xx}(Y_{+,i})_2\right)\;dx+\mathrm{O}\left(e^{-\gamma t}\|Z\|_{H^1\times L^2}^2\right)\\
=&\;\mathrm{O}\left(e^{-\gamma t}\|Z\|_{H^1\times L^2}^2\right).
\end{align*}
\end{proof}

We deduce from Proposition \ref{prop_der_F_z_F_E} and Lemma \ref{lem_functional_F_E} the following "weak" monotonicity property.

\begin{Cor}\label{cor_monotonie_F_E} We have for all $t\ge t_0$,
\begin{equation}\label{monotonie_F_E}
-\mathcal{F}_\varepsilon '(t)\le \frac{\lambda}{t}\mathcal{F}_{\varepsilon}(t)+\frac{C}{t}\sum_{i=1}^N\alpha_{+,i}^2+\mathrm{O}\left(e^{-\gamma t}\|Z\|_{H^1\times L^2}^2+\|Z\|_{H^1\times L^2}^{3}\right).
\end{equation}
\end{Cor}

\begin{proof}
From Lemma \ref{lem_control_Z_pm_i}, we obtain: for all $i=1,\dots,N$, for all $t\ge t_0$,
\begin{equation}
\left|(\alpha_{-,i}\alpha_{+,i})'\right|\le C\left(e^{-\gamma t}\|Z\|_{H^1\times L^2}^2+\|Z\|_{H^1\times L^2}^{3}\right).
\end{equation}

Thus, we have 
\begin{align*}
-\mathcal{F}_\varepsilon '(t)&=-\mathcal{F}'_z(t)+2\sum_{i=1}^N(\alpha_{-,i}\alpha_{+,i})'+\mathrm{O}\left(e^{-\gamma t}\|Z\|^2\right)\\
&\le \frac{\lambda}{t}\mathcal{F}_{\varepsilon,\Omega(t)}+\mathrm{O}\left(e^{-\gamma t}\|Z\|_{H^1\times L^2}^2+\|Z\|_{H^1\times L^2}^{3}\right).
\end{align*}

We now make use of the following property satisfied by $\mathcal{F}_\varepsilon$ which is a consequence of a localized version of Proposition \ref{prop_spectral} (we refer to \cite[proof of (4.12) and (4.21)]{mm_waveeq} for similar considerations in the case of the energy-critical wave equation):
\begin{equation}
\mathcal{F}_{\varepsilon,\Omega(t)}\le \mathcal{F}_{\varepsilon}(t)+C\sum_{i=1}^N\langle\varepsilon,Z_{+,i}\rangle^2,
\end{equation}
to deduce that
$$-\mathcal{F}_\varepsilon '(t)\le \frac{\lambda}{t}\mathcal{F}_{\varepsilon}(t)+\frac{C}{t}\sum_{i=1}^N\alpha_{+,i}^2+\mathrm{O}\left(e^{-\gamma t}\|Z\|_{H^1\times L^2}^2+\|Z\|_{H^1\times L^2}^{3}\right).$$
\end{proof}

We are now in a position to prove

\begin{proposition}\label{prop_Z_alpha_-,i}
Even if it means taking a larger $t_0$, we have for all $t\ge t_0$
\begin{equation}
\|Z(t)\|_{H^1\times L^2}\le C\sup_{t'\ge t}\sum_{i=1}^N|\alpha_{-,i}(t)|.
\end{equation}
\end{proposition}

\begin{proof}
Multiplying the estimate obtained in Corollary \ref{cor_monotonie_F_E} by $t^\lambda$, we have for all $t\ge t_0$,
$$-(t^\lambda\mathcal{F}_E)'(t)\le C\left(\sum_{i=1}^Nt^{\lambda-1}\alpha_{+,i}^2+t^\lambda e^{-\gamma t}\|Z\|_{H^1\times L^2}  ^2+t^\lambda\|Z\|_{H^1\times L^2}  ^3\right).$$
Since, $t^{\lambda-1}\alpha_{+,i}^2$, $t^\lambda\|Z\|_{H^1\times L^2}^3$ are integrable functions of $t$ on $[t_0,+\infty)$ and $t^\lambda\mathcal{F}_E(t)\to 0$ as $t\to+\infty$, we infer that
\begin{equation}
\mathcal{F}_E(t)\le C\left(\frac{1}{t^\lambda}\sum_{i=1}^N\int_t^{+\infty}t'^{\lambda-1}\alpha_{+,i}^2(t')\;dt'+\frac{C}{t^\lambda}\int_t^{+\infty}t'^\lambda\|Z(t')\|_{H^1\times L^2}^3\;dt'+e^{-\gamma t}\sup_{t'\ge t}\|Z(t')\|^2\right).
\end{equation}
By the coercivity property satisfied by $\mathcal{F}_E$, we thus obtain
\begin{align*}
\|E(t)\|_{H^1\times L^2}^2\le &\; C\left(\mathcal{F}_E(t)+\sum_{i=1}^N\alpha_{+,i}^2(t)+e^{-2\gamma t}\|Z(t)\|_{H^1\times L^2}^2\right)\\
\le &\; \frac{C}{t^\lambda}\sum_{i=1}^N\int_t^{+\infty}t'^{\lambda-1}\alpha_{+,i}^2(t')\;dt'+\frac{C}{t^\lambda}\int_t^{+\infty}t'^\lambda\|Z(t')\|_{H^1\times L^2}^3\;dt'\\
&+Ce^{-\gamma t}\sup_{t'\ge t}\|Z(t')\|_{H^1\times L^2}^2+C\sum_{i=1}^N\alpha_{+,i}^2(t).
\end{align*}
In other words, there exists $C\ge 0$ such that for all $t\ge t_0$,
\begin{equation}\label{ineq_E}
\begin{aligned}
\|E(t)\|_{H^1\times L^2}  \le & \frac{C}{t^{\frac{\lambda}{2}}}\left(\int_t^{+\infty}t'^{\lambda-1}\sum_{i=1}^N\alpha_{+,i}^2(t')\;dt'\right)^{\frac{1}{2}}+C\left(\sum_{i=1}^N\alpha_{+,i}^2(t)\right)^{\frac{1}{2}}\\
&+\frac{C}{t^{\frac{\lambda}{2}}}\left( \int_t^{+\infty}t'^{\lambda}\|Z(t')\|_{H^1\times L^2}^3\;dt' \right)^{\frac{1}{2}}+Ce^{-\frac{\gamma}{2}t}\sup_{t'\ge t}\|Z(t')\|_{H^1\times L^2}.
\end{aligned}
\end{equation}

From Lemma \ref{lem_control_Z_pm_i} we recall the estimate
$$|\alpha_{+,i}'(t)- e_i\alpha_{+,i}(t)|\le C\left(e^{-\gamma t}\|Z(t)\|_{H^1\times L^2}+\|Z(t)\|^2_{H^1\times L^2}\right)$$ which is equivalent to
$$|(e^{-e_it}\alpha_{+,i})'(t)|\le Ce^{-e_it}\left(e^{-\gamma t}\|Z(t)\|_{H^1\times L^2}+\|Z(t)\|^2_{H^1\times L^2}\right).$$
Integrating the preceding inequality (which is indeed possible), we deduce:
\begin{equation}\label{ineq_alpha_+,i}
|\alpha_{+,i}(t)|\le C\left(e^{-\gamma t}\sup_{t'\ge t}\|Z(t')\|_{H^1\times L^2}+\sup_{t'\ge t}\|Z(t')\|^2_{H^1\times L^2}\right).
\end{equation}

Now, 
$$\|Z(t)\|_{H^1\times L^2}\le \|E(t)\|_{H^1\times L^2}+C\sum_{i,\pm}|\alpha_{\pm,i}(t)|+C\sum_{i=1}^N|a_i(t)|+Ce^{-\gamma t}\|Z(t)\|_{H^1\times L^2}.$$
By Lemma \ref{lem_control_modulation_parameters}, it follows that
\begin{equation}
\|Z(t)\|_{H^1\times L^2}\le C\left(\|E(t)\|_{H^1\times L^2}+\sum_{\pm,i}|\alpha_{\pm,i}(t)|+\int_t^{+\infty}\left(\|E(t')\|_{H^1\times L^2}+\|Z(t')\|^2_{H^1\times L^2}\right)\;dt'\right).
\end{equation}
Observing that the quantity $\int_{t}^{+\infty}\|Z(t')\|_{H^1\times L^2}\;dt'$ makes sense and tends to 0 as $t\to +\infty$ (because $\alpha>1$) and that
$$\int_t^{+\infty} \|Z(t')\|^2_{H^1\times L^2}\;dt'\le \sup_{t'\ge t} \|Z(t')\|_{H^1\times L^2}\int_t^{+\infty} \|Z(t')\|_{H^1\times L^2}\;dt,$$ we deduce that for $t$ sufficiently large
\begin{equation}\label{ineq_Z_E}
\|Z(t)\|_{H^1\times L^2}\le C\left(\sup_{t'\ge t}\|E(t')\|_{H^1\times L^2}+\int_t^{+\infty}\|E(t')\|_{H^1\times L^2}\;dt'+\sup_{t'\ge t}\sum_{\pm,i}|\alpha_{\pm,i}(t')|\right).
\end{equation}

Now, we replace \eqref{ineq_alpha_+,i} in \eqref{ineq_E} and use \eqref{ineq_Z_E}. We notice that the following well-defined quantities tend to 0 as $t\to +\infty$ (because $\alpha>3$ and by the choice of $\lambda<\alpha-1$):
\begin{align*}
\int_t^{+\infty}t'^{\lambda-1}\sup_{t''\ge t'}\|Z(t'')\|_{H^1\times L^2}^2\;dt',\qquad &\int_{t}^{+\infty}\sup_{t''\ge t'}\|Z(t'')\|_{H^1\times L^2}\;dt'\\
\int_t^{+\infty}t'^{\lambda}\|Z(t')\|_{H^1\times L^2}\;dt',\qquad&\int_t^{+\infty}\frac{1}{u^{\frac{\lambda}{2}}}\left(\int_u^{+\infty}t'^{\lambda-1}\sup_{t''\ge t'}\|Z(t'')\|_{H^1\times L^2}^2\;dt'\right)^{\frac{1}{2}}\;du\\
\int_t^{+\infty}\frac{1}{u^{\frac{\lambda}{2}}}\left(\int_u^{+\infty}t'^{\lambda}\|Z(t')\|_{H^1\times L^2}\;dt'\right)^{\frac{1}{2}}\;du.
\end{align*}
We then obtain for $t$ sufficiently large
$$\|Z(t)\|_{H^1\times L^2}\le C\sup_{t'\ge t}\sum_{i=1}^N|\alpha_{-,i}(t')|.$$
\end{proof}

\begin{proposition}\label{prop_alpha_pm,i}
We have for all $t\ge t_0$, for all $i=1,\dots,N$,
\begin{equation}
|\alpha_{-,i}(t)|\le Ce^{-e_1t}.
\end{equation}
\end{proposition}

\begin{proof}
From Lemma \ref{lem_control_Z_pm_i} and Proposition \ref{prop_Z_alpha_-,i}, it results that for all $i=1,\dots,N$, for all $t\ge t_0$,
\begin{equation}
|\alpha_{-,i}'(t)+ e_i\alpha_{-,i}(t)|\le C\left( e^{-\sigma t}\sup_{t'\ge t}\sum_{j=1}^N|\alpha_{-,j}(t')|+\left(\sup_{t'\ge t}\sum_{j=1}^N|\alpha_{-,j}(t')|\right)^2\right).
\end{equation}
Then, for all $i=1,\dots,N$,
\begin{equation}
|\alpha_{-,i}'(t)\alpha_{-,i}(t)+ e_i\alpha_{-,i}^2(t)|\le C\left( e^{-\sigma t}\sup_{t'\ge t}\sum_{j=1}^N|\alpha_{-,j}(t')||\alpha_{-,i}(t)|+\sup_{t'\ge t}\left(\sum_{j=1}^N|\alpha_{-,j}(t')|\right)^2|\alpha_{-,i}(t)|\right).
\end{equation}
Let us denote $\mathcal{A}:=\sum_{j=1}^N\alpha_{-,j}^2$. Summing on $i=1,\dots,N$, we have in particular:
$$|\mathcal{A}'(t)+2e_1\mathcal{A}(t)|\le Ce^{-\sigma t}\left(\sup_{t'\ge t}\sum_{j=1}^N|\alpha_{-,j}(t')|\right)\sum_{i=1}^N|\alpha_{-,i}(t)|+C\left(\sup_{t'\ge t}\sum_{j=1}^N|\alpha_{-,j}(t')|\right)^2\sum_{i=1}^N|\alpha_{-,i}(t)|.$$
Noticing that $\left(\sum_{j=1}^N|\alpha_{-,j}|\right)^2\le C\mathcal{A}$, we obtain the existence of $c>0$ such that for all $t\ge t_0$,
$$|\mathcal{A}'(t)+2e_1\mathcal{A}(t)|\le C\left(e^{-\sigma t}+\mathcal{A}(t)^{\frac{1}{2}}\right)\sup_{t'\ge t}\mathcal{A}(t').$$

Lastly we observe that $\xi:t\mapsto e^{-\sigma t}+\mathcal{A}(t)^{\frac{1}{2}}$ is integrable on $[t_0,+\infty)$ since $$\mathcal{A}(t)^{\frac{1}{2}}=\mathrm{O}\left(\|Z(t)\|_{H^1\times L^2}^{\frac{1}{2}}\right)=\mathrm{O}\left(\frac{1}{t^{\frac{\alpha}{2}}}\right).$$
By Lemma \ref{lem_gen_diff} in Appendix, we obtain $\mathcal{A}(t)\le Ce^{-2e_1t}$ for $t\ge t_0$. Consequently, for all $i=1,\dots,N$,
$$|\alpha_{-,i}(t)|\le Ce^{-e_1t}.$$
\end{proof}

Gathering Propositions \ref{prop_Z_alpha_-,i} and \ref{prop_alpha_pm,i}, we deduce 

\begin{proposition}\label{prop_improve_Z}
There exists $C\ge 0$ such that for $t$ sufficiently large, $$\|Z(t)\|_{H^1\times L^2}\le Ce^{-e_1t}.$$  
\end{proposition}

\subsection{Identification of the solution}

Recall that we have constructed in Section \ref{sect_const_mult} a family of multi-solitons $(\phi_{A_1,\dots,A_N})$ such that for all $j=1,\dots,N$, for all $t\ge t_0$,
\begin{equation}
\left\|\Phi_{A_1,\dots,A_j}(t)-\Phi_{A_1,\dots,A_{j-1}}(t)-A_je^{-e_jt}Y_{+,j}(t)\right\|_{H^1\times L^2}\le Ce^{-(e_j+\sigma)t}.
\end{equation}
(We can always assume that $\sigma<\min\left\{e_1,\min_{j=2,\dots,N}\{e_j-e_{j-1}\}\right\}$).

Following the strategy of Combet \cite{combetgKdV}, our goal is to establish

\begin{proposition}\label{prop_dec_E_j}
For all $j=1,\dots,N$, there exist $C\ge 0$, $t_0\ge 0$, and $A_1,\dots,A_j\in\RR$ such that, defining
$$E_j:=U-\Phi_{A_1,\dots,A_j},$$ we have:
\begin{equation}
\|E_j(t)\|_{H^1\times L^2}\le Ce^{-e_jt}.
\end{equation}
Moreover, denoting $\alpha_{\pm,j,k}:=\langle E_j,Z_{\pm,k}\rangle$ for all $k=1,\dots,N$, we have
\begin{equation}
\forall\;k\in\{1,\dots,j\},\qquad e^{e_kt}\alpha_{-,j,k}(t)\to 0,\quad\text{as } t\to+\infty.
\end{equation}
\end{proposition}

\begin{proof}
We proceed by induction on $j$. First, we focus on the case where $j=1$.\\
We have $\|Z(t)\|_{H^1\times L^2}\le Ce^{-e_1t}$ by Proposition \ref{prop_improve_Z}.
Thus, by Lemma \ref{lem_control_Z_pm_i} and given that $\sigma<e_1$,
 $$\left|(e^{e_1t}\alpha_{-,1})'\right|\le Ce^{-\sigma t}.$$ Since $t\mapsto e^{-\sigma t}$ is integrable in $+\infty$, there exists $A_1\in\RR$ such that
$$e^{e_1t}\alpha_{-,1}(t)\to A_1,\qquad\text{as } t\to +\infty.$$

\noindent We then define $E_1:= U-\Phi_{A_1}$. We notice that $E_1=E+(\Phi-\Phi_{A_1})$ so that
\begin{align*}
\|E_1(t)\|_{H^1\times L^2}&\le \|E(t)\|_{H^1\times L^2}+\|\Phi(t)-\Phi_{A_1}(t)\|_{H^1\times L^2}\\
&\le Ce^{-e_1t}+\|\Phi_{A_1}(t)-\Phi(t)-A_1e^{-e_1t}Y_{+,1}(t)\|_{H^1\times L^2}+\|A_1e^{-e_1t}Y_{+,1}(t)\|_{H^1\times L^2}\\
&\le Ce^{-e_1t}.
\end{align*}
Moreover
\begin{align*}
\alpha_{-,1,1}=&\;\langle E,Z_{-,1}\rangle+\langle \Phi-\Phi_{A_1}-A_1e^{-e_1t}Y_{+,1},Z_{-,1}\rangle+A_1e^{-e_1t}\langle Y_{+,1},Z_{-,1}\rangle\\
=&-\alpha_{-,1}+A_1e^{-e_1t}+\mathrm{O}\left(e^{-(e_1+\sigma)t}\right)\\
=&\;\mathrm{o}\left(e^{-e_1t}\right),
\end{align*} the last line resulting from the definition of $A_1$.\\
Thus Proposition \ref{prop_dec_E_j} is true for $j=1$.\\

We now assume that there exist $A_1,\dots,A_{j-1}\in\RR$ such that $\|E_{j-1}(t)\|_{H^1\times L^2}\le Ce^{-e_{j-1}t}$ and for all $k=1,\dots,j-1$, $e^{e_kt}\alpha_{-,j-1,k}(t)\to 0$ as $t\to+\infty$. \\
Let us show 
\begin{Claim}
We have $$\|E_{j-1}(t)\|_{H^1\times L^2}\le Ce^{-e_jt}.$$
\end{Claim}
\begin{itemize}
\item
To prove this claim, we show that, if $\|E_{j-1}(t)\|\le Ce^{-\sigma_0t}$ with $e_{j-1}<\sigma_0<e_j-\sigma$, then
$$\|E_{j-1}(t)\|_{H^1\times L^2}\le Ce^{-(\sigma+\sigma_0)t}.$$

As 
\begin{align*}
\left|\alpha_{\pm,j-1,k}'(t)\mp e_k\alpha_{\pm,j-1,k}(t)|\right|&\le C\left(e^{-\sigma t}\|E_{j-1}(t)\|_{H^1\times L^2}+\|E_{j-1}(t)\|_{H^1\times L^2}^2\right)\\
&\le Ce^{-(\sigma+\sigma_0)t}
\end{align*}
(by the same calculations and arguments as those developed in the proof of Lemma \ref{lem_control_Z_pm_i}), we have
for all $k=1,\dots,j-1$,
$$\left|(e^{e_kt}\alpha_{-,j-1,k})'\right|\le Ce^{-(\sigma+\sigma_0-e_k)t}.$$
Since $t\mapsto e^{-(\sigma+\sigma_0-e_k)t}$ is integrable in the neighborhood of $+\infty$ (since $e_k\le e_{j-1}$), and by assumption, $e^{e_kt}\alpha_{-,j-1,k}(t)\to 0$ as $t\to+\infty$, we have by integration
$$ |\alpha_{-,j-1,k}(t)|\le Ce^{-(\sigma+\sigma_0)t}.$$
For all $k=j,\dots,N$, we have $\sigma+\sigma_0-e_k\le \sigma+\sigma_0-e_j<0$, thus by integration on $[t_0,t]$, we obtain
$$\left|e^{e_kt}\alpha_{-,j-1,k}(t)-e^{e_kt_0}\alpha_{-,j-1,k}(t_0)\right|\le Ce^{(e_k-\sigma_0-\sigma)t}.$$
Eventually, we obtain (by a "cut-and-paste" of the argument exposed in subsection \ref{subsec_exp_conv})
$$\|E_{j-1}(t)\|_{H^1\times L^2}\le C\sup_{t'\ge t}\sum_{k=1}^N|\alpha_{-,j-1,k}(t')|\le Ce^{-(\sigma_0+\sigma)t},$$ which is what was expected.
\item Now, from the preceding induction, there exists $\tilde{\sigma}_0\in\left(e_j-\sigma,e_j\right)$ such that
$$\|E_{j-1}(t)\|_{H^1 \times L^2}\le Ce^{-\tilde{\sigma}_0t},$$ from which we deduce
$$\left|(e^{e_kt}\alpha_{-,j-1,k})'\right|\le Ce^{-(\sigma+\tilde{\sigma}_0-e_k)t}.$$

Now, for $k\in\{1,\dots,j-1\}$, $e_k-\sigma-\tilde{\sigma}_0\le e_{j-1}-\sigma-e_j<0$, we thus have
$$|\alpha_{-,j-1,k}(t)|\le Ce^{-(\sigma_0+\tilde{\sigma})t}\le Ce^{-e_jt}.$$
For $k=j$, we have $|(e^{e_jt}\alpha_{-,j-1,j})'|\le Ce^{(e_j-\tilde{\sigma}_0-\sigma)t}.$ Thus, there exists $A_j\in\RR$ such that
$$e^{e_jt}\alpha_{-,j-1,j}(t)\to A_j,\quad\text{as } t\to+\infty.$$

\noindent For $k\in\{j+1,\dots,N\}$, we have $\sigma+\tilde{\sigma}_0-e_k<\sigma+e_{j}-e_k<0$, thus by integration
\begin{align*}
|\alpha_{-,j-1,k}(t)|&\le Ce^{-e_kt}+Ce^{-(\tilde{\gamma}_0+\gamma)t}\\
&\le Ce^{-e_jt}.
\end{align*}
Hence, $$\|E_{j-1}(t)\|_{H^1 \times L^2}\le C\sup_{t'\ge t}\sum_{k=1}^N|\alpha_{-,j-1,k}(t')|\le Ce^{-e_jt}.$$
\end{itemize}
Let us conclude the proof of Proposition \ref{prop_dec_E_j}. We define at this stage
$E_j:=U-\Phi_{A_1,\dots,A_j}$. We immediately have
$$E_j(t)=E_{j-1}(t)+\Phi_{A_1,\dots,A_{j-1}}(t)-\Phi_{A_1,\dots,A_j}(t).$$
Then,
\begin{align*}
\|E_{j}(t)\|_{H^1\times L^2}\le & \;\|E_{j-1}(t)\|_{H^1\times L^2}+\|\Phi_{A_1,\dots,A_j}(t)-\Phi_{A_1,\dots,A_{j-1}}(t)-A_je^{-e_jt}Y_{+,j}(t)\|_{H^1\times L^2}\\
&+\|A_je^{-e_jt}Y_{+,j}(t)\|_{H^1\times L^2}\\
\le &\; Ce^{-e_jt}.
\end{align*}
What is more,
\begin{align*}
\alpha_{-,j,k}(t)&=\langle E_j(t),Z_{-,k}(t)\rangle\\
&=\alpha_{-,j-1,k}(t)-A_je^{-e_jt}\langle Y_{+,j},Z_{-,k}\rangle+\mathrm{O}\left(e^{-(e_j+\sigma)t}\right).
\end{align*}
For $k=1,\dots,j-1$, we have:
\begin{align*}
e^{e_kt}|\alpha_{-,j,k}(t)| &\le Ce^{e_kt}|\alpha_{-,j-1,k}(t)|+\mathrm{O}\left(e^{-(e_j-e_k+\sigma)t}\right).\\
&\le Ce^{(e_k-e_j)t}\underset{t\to+\infty}{\longrightarrow} 0.
\end{align*}
For $k=j$, $$e^{e_jt}\alpha_{-,j,j}(t)=e^{e_jt}\alpha_{-,j-1,j}(t)-A_j+\mathrm{O}(e^{-\sigma t})\underset{t\to+\infty}{\longrightarrow} 0.$$
This finishes the induction argument.
\end{proof}

Finally we obtain that $U=\Phi_{A_1,\dots,A_N}$ by means of

\begin{Cor}
For $t$ sufficiently large, $\|E_{N}(t)\|_{H^1\times L^2}=0$.
\end{Cor}

\begin{proof}
As in the preceding proofs, the following bounds hold:
\begin{equation}\label{eq_aux_1'}
\|E_N(t)\|_{H^1\times L^2}\le C\sup_{t'\ge t}\sum_{i=1}^N|\alpha_{-,N,i}(t')|
\end{equation}
and
\begin{equation}\label{eq_aux}
|\alpha_{-,N,i}'(t)+e_i\alpha_{-,N,i}(t)|\le C\left(e^{-\sigma t}\|E_N(t)\|_{H^1\times L^2}+\|E_N(t)\|_{H^1\times L^2}^2\right).
\end{equation}
We observe that $t\mapsto e^{e_it}\left(e^{-\sigma t}\|E_N(t)\|_{H^1\times L^2}+\|E_N(t)\|_{H^1\times L^2}^2\right)$ is integrable on $[t_0,+\infty)$; this is due to the fact that $\|E_N(t)\|_{H^1\times L^2} \le Ce^{-e_Nt}$. Since for all $i=1,\dots,N$, $e^{e_it}\alpha_{-,N,i}(t)\to 0$ as $t\to+\infty$, we obtain by integration of \eqref{eq_aux} on $[t,+\infty)$:
$$|\alpha_{-,N,i}(t)|\le C\left(e^{-\sigma t}\sup_{t'\ge t}\|E_N(t')\|_{H^1\times L^2}+\sup_{t'\ge t}\|E_N(t')\|_{H^1\times L^2}^2\right).$$
Then, using \eqref{eq_aux_1'}, we obtain 
$$\sup_{t'\ge t}\|E_N(t')\|_{H^1\times L^2}\le C\left(e^{-\sigma t}\sup_{t'\ge t}\|E_N(t')\|_{H^1\times L^2}+\sup_{t'\ge t}\|E_N(t')\|_{H^1\times L^2}^2\right).$$
This implies that $\|E_N(t)\|_{H^1\times L^2}=0$ for $t$ sufficiently large.
\end{proof}

\section{Construction of a one-parameter family of solutions converging to a soliton}\label{sect_constr_one_par}

The goal of this section is to prove the existence part in Theorem \ref{th_main}. Once again we restrict our focus to $d=1$.

\subsection{Outline of the construction}

Let $A\in\RR$.  \\

Let $(S_n)_{n\in\NN}$ be an increasing sequence of real numbers which tends to $+\infty$ and, for all $n\in\NN$, define $u_n$ as the maximal solution of \eqref{NLKG} such that 
\begin{equation}\label{def_u_n}
U_n(S_n)=R_\beta(S_n)+Ae^{-e_\beta S_n}Y_{+,\beta}(S_n),
\end{equation} with obvious notations.

We aim at proving the following key proposition:

\begin{proposition}\label{prop_est_uniform}
There exist $t_0\ge 0$ and $C_0\ge 0$ such that for $n$ large,
\begin{equation}
\forall\;t\in[t_0,S_n],\qquad \left\|U_n(t)-R_\beta(t)-Ae^{-e_\beta t}Y_{+,\beta}(t)\right\|_{H^1\times L^2}\le C_0e^{-2e_\beta t}.
\end{equation}
\end{proposition}

To this end, we will set up a bootstrap argument and show

\begin{proposition}\label{prop_boot}
There exist $\alpha_0>0$, $C_0>0$, and $t_0\ge 0$ such that for $n$ sufficiently large,\\
if there exists $t_n^*\in[t_0,S_n]$ such that for all $t\in [t_n^*,S_n]$, 
\begin{equation}\label{est_a_priori}
\|U_n(t)-R_\beta(t)-Ae^{-e_\beta t}Y_{+,\beta}(t)\|_{H^1\times L^2}\le \alpha_0,
\end{equation}
then for all $t\in [t_n^*,S_n]$, 
\begin{equation}\label{est_a_posteriori}
\|U_n(t)-R_\beta(t)-Ae^{-e_\beta t}Y_{+,\beta}(t)\|_{H^1\times L^2}\le C_0e^{-2e_\beta t}.
\end{equation}
\end{proposition}

Let us show how to deduce Proposition \ref{prop_est_uniform} from Proposition \ref{prop_boot}. 

\begin{proof}[Proof of Proposition \ref{prop_est_uniform}]
Assume momentarily that Proposition \ref{prop_boot} holds true. Let us consider $\alpha_0$ and $C_0$ as in Proposition \ref{prop_boot} and suppose (even if it means enlarging $t_0$) that $C_0e^{-2e_\beta t_0}\le \frac{\alpha_0}{2}$. \\
We define for all $n$ such that $S_n>t_0$:
$$t^*_n:=\inf\left\{t\in[t_0,S_n],\:\forall\;\tau\in[t,S_n],\:\left\|U_n(\tau)-R_\beta(\tau)-Ae^{-e_\beta \tau}Y_{+,\beta}(\tau)\right\|_{H^1\times L^2}\le \alpha_0\right\}.$$
By \eqref{def_u_n} and by continuity in time of $U_n$, $R_\beta$, and $Y_{+,\beta}$, $t^*_n$ is indeed well-defined and we necessarily have $t_0\le t^*_n<S_n$. \\
Since \eqref{est_a_priori} implies \eqref{est_a_posteriori}, for all $t\in[t^*_n,S_n]$, 
\begin{align*}
\left\|U_n(t)-R_\beta(t)-Ae^{-e_\beta t}Y_{+,\beta}(t)\right\|_{H^1\times L^2}&\le C_0e^{-2e_\beta t}\\
&\le C_0e^{-2e_\beta t_0}\\
&\le \frac{\alpha_0}{2}.
\end{align*}
Let us assume for the sake of contradiction that $t^*_n>t_0$ for some $n$. Then, observing the preceding inequality, we obtain (again by continuity in time of $U_n$, $R_\beta$, and $Y_{+,\beta}$) the existence  of $\tau_n>0$ such that $t^*_n-\tau_n\ge t_0$ and for all $t\in[t^*_n-\tau_n,S_n]$, 
$$\left\|U_n(t)-R_\beta(t)-Ae^{-e_\beta t}Y_{+,\beta}(t)\right\|_{H^1\times L^2}\le \frac{3\alpha_0}{4}<\alpha_0. $$
This contradicts the definition of $t^*_n$ as an infimum. Hence $t^*_n=t_0$ and \eqref{est_a_priori} (and thus \eqref{est_a_posteriori}) holds on $[t_0,S_n]$ for all $n$. \\
This achieves the proof of Proposition \ref{prop_est_uniform}.
\end{proof}

The existence of $u^A$ (and $U^A$), as stated in Theorem \ref{th_main} is a consequence of Proposition \ref{prop_est_uniform} and the continuity of the flow of \eqref{NLKG} for the weak $H^1\times L^2$ topology. We will not detail the construction of $U^A$ considering that it is a sort of cut and paste of what was done in order to prove Proposition \ref{prop_main} in the context of multiple solitons. \\

Similarly, we do not repeat the arguments exposed at the beginning of section \ref{sect_const_mult} which justify that the map $A\mapsto  u^A$ is one-to-one. We devote the next subsection to the proof of Proposition \ref{prop_boot}. 


\subsection{Proof of Proposition \ref{prop_boot}} 
We assume that $U_n(t)$ is defined on some interval $[t_n^*,S_n]$ and satisfies \eqref{est_a_priori}. We want to show that \eqref{est_a_posteriori} holds, provided that the parameters $\alpha_0$ and $t_0$ are well chosen. \\

In this subsection again, for notation purposes and ease of reading, we sometimes omit the index $n$ and also write $\mathrm{O}\left(G(t)\right)$ in order to refer to a function $g$ which a priori depends on $n$ and such that there exists $C\ge 0$ (independent of $n$) such that for all $n$ large and for all $t\in[t_n^*,S_n]$, $|g(t)|\le C|G(t)|$.

\subsubsection{Step 1: Set up of a modulation argument}

\begin{Lem}\label{lem_modulation}
For $t_0\ge 0$ sufficiently large and $\alpha_0>0$ sufficiently small, there exists a unique $\mathscr{C}^1$ function $x:[t_n^*,S_n]\to\RR$ such that if we set 
$$W_n(t):=U_n(t)-\tilde{R}_\beta(t)-Ae^{-e_\beta t}\tilde{Y}_{+,\beta}(t),$$
with $\tilde{R}_\beta(t):=R_\beta(t,\cdot-x(t))$ and $\tilde{Y}_{+,\beta}(t):=Y_{+,\beta}(t,\cdot-x(t))$,
then for all $t\in[t_n^*,S_n]$,
\begin{equation}\label{prop_ortho}
\left\langle W_n(t),\partial_x\tilde{R}_\beta(t)\right\rangle=0.
\end{equation}
Moreover there exists $K_1>0$ such that for all $t\in[t_n^*,S_n]$, 
\begin{equation}
\|W_n(t)\|_{H^1\times L^2}+|x(t)|\leq K_1\alpha_0,
\end{equation}
\begin{equation}\label{est_x_dot}
|\dot{x}(t)|\leq K_1\left(\|W_n(t)\|_{H^1\times L^2}+e^{-2e_\beta t}\right).
\end{equation}
\end{Lem}

\begin{Rq}
Notice that by uniqueness of the function $x$ and by definition of $u_n$ (see \eqref{def_u_n}), we have $W_n(S_n)=0$ and $x(S_n)=0$.
\end{Rq}

\begin{proof}
The existence of $x$ such that \eqref{prop_ortho} is granted and the existence of $K_2>0$ such that
$$\|W_n(t)\|_{H^1\times L^2}+|x(t)|\leq K_2\alpha_0$$
are standard consequences of the implicit function theorem.\\
 Now, let us prove \eqref{est_x_dot}. For this, we notice that $W=W_n$ satisfies the following equation:
\begin{equation}\label{eq_W}
\begin{aligned}
\partial_tW=&\left(\begin{array}{cc}
0&Id\\
\partial_x^2-Id+f'(\tilde{Q}_\beta)&0
\end{array}\right)\left(W+Ae^{-e_\beta t}\tilde{Y}_{+,\beta}\right)+\dbinom{0}{f(u)-f(\tilde{Q}_\beta)-f'(\tilde{Q}_\beta)(u-\tilde{Q}_\beta)}\\
&+\dot{x}(t)\partial_x\tilde{R}_\beta-Ae^{-e_\beta t}\left(\partial_t\tilde{Y}_{+,\beta}-\dot{x}(t)\partial_x\tilde{Y}_{+,\beta}-e_\beta \tilde{Y}_{+,\beta}\right),
\end{aligned}
\end{equation}
where $\tilde{Q}_\beta(t,x)=Q_\beta(t,x-x(t))$.

Since $\frac{d}{dt}\langle W,\partial_x\tilde{R}_\beta\rangle=0$, we have:
$$\langle \partial_tW,\partial_x\tilde{R}_\beta\rangle+\langle W,\partial_{tx}\tilde{R}_\beta-\dot{x}\partial_x^2\tilde{R}_\beta\rangle=0.$$
Observing moreover that $$f(u)-f(\tilde{Q}_\beta)-f'(\tilde{Q}_\beta)(u-\tilde{Q}_\beta)=\mathrm{O}\left(\|U-\tilde{R}_\beta\|_{H^1\times L^2}^2\right)=\mathrm{O}\left(\|W(t)\|^2_{H^1\times L^2}+e^{-2e_\beta t}\right)$$ by Taylor formula ($f$ is $\mathscr{C}^2$), we have thus: 
\begin{align*}
0=&\left\langle \left(\begin{array}{cc}
0&Id\\
\partial_x^2-Id+f'(\tilde{Q}_\beta)&0
\end{array}\right)W,\partial_x\tilde{R}_\beta\right\rangle+\mathrm{O}\left(\|W(t)\|^2_{H^1\times L^2}+e^{-2e_\beta t}\right)+\dot{x}(t)\|\partial_xR_\beta\|^2_{H^1\times L^2}\\
&+Ae^{-e_\beta t}\left\langle (\beta+\dot{x})\partial_x\tilde{Y}_{+,\beta}+e_\beta\tilde{Y}_{+,\beta},\partial_x\tilde{R}_\beta\right\rangle-(\beta+\dot{x})\langle W,\partial_x^2\tilde{R}_\beta\rangle\\
=&\;\dot{x}\left(\|\partial_xR_\beta\|^2_{H^1\times L^2}+Ae^{-e_\beta t}\langle \partial_x\tilde{Y}_{+,\beta},\partial_x\tilde{R}_\beta\rangle-\langle W,\partial_x^2\tilde{R}_\beta\rangle\right)+\mathrm{O}\left(\|W(t)\|^2_{H^1\times L^2}+e^{-2e_\beta t}\right)\\
&+\left\langle\left(\begin{array}{cc}
\beta\partial_x&Id\\
\partial_x^2-Id+f'(\tilde{Q}_\beta)&\beta\partial_x
\end{array}\right)W,\partial_x\tilde{R}_\beta\right\rangle+Ae^{-e_\beta t}\left\langle\left(\begin{array}{cc}
\beta\partial_x&Id\\
\partial_x^2-Id+f'(\tilde{Q}_\beta)&\beta\partial_x
\end{array}\right)\tilde{Y}_{+,\beta},\partial_x\tilde{R}_\beta\right\rangle.
\end{align*}
Notice that we have used $\langle Y_{+,\beta},\partial_xR_\beta\rangle=0$ (see Proposition \ref{prop_spectral}). We now observe that 
$$\left(\begin{array}{cc}
\beta\partial_x&Id\\
\partial_x^2-Id+f'(\tilde{Q}_\beta)&\beta\partial_x
\end{array}\right)=-J\tilde{H}_\beta,$$ where $\tilde{H}_\beta$ is the matrix operator defined like $H_\beta$ by replacing $Q_\beta$ by $\tilde{Q}_\beta$, that is
\begin{equation}\label{def_H_beta}
\tilde{H}_\beta:=\left(\begin{array}{cc}
-\partial_x^2+Id-f'(\tilde{Q}_\beta)&-\beta\partial_x\\
\beta\partial_x&Id
\end{array}
\right).
\end{equation} We obtain:
$$
\left\langle\left(\begin{array}{cc}
\beta\partial_x&Id\\
\partial_x^2-Id+f'(\tilde{Q}_\beta)&\beta\partial_x
\end{array}\right)\tilde{Y}_{+,\beta},\partial_x\tilde{R}_\beta\right\rangle =\left\langle \tilde{Y}_{+,\beta},\tilde{H}_\beta J\partial_x\tilde{R}_\beta\right\rangle =\left\langle \tilde{Y}_{+,\beta},0\right\rangle =0.
$$
Finally we can choose $t$ large enough such that $$Ae^{-e_\beta t}\left|\langle \partial_x\tilde{Y}_{+,\beta},\partial_x\tilde{R}_\beta\rangle\right|\le \frac{1}{4}\|\partial_xR_\beta\|^2_{H^1\times L^2}$$ and we can take $\alpha_0>0$ sufficiently small such that 
$$|\langle W,\partial_x^2\tilde{R}_\beta\rangle|\le K_2\alpha_0\|\partial_x^2\tilde{R}_\beta\|_{H^1\times L^2}\le \frac{1}{4}\|\partial_xR_\beta\|^2_{H^1\times L^2}$$ 
and such that $K_2\alpha_0\le 1$ (then $\|W(t)\|^2_{H^1\times L^2}\le \|W(t)\|_{H^1\times L^2}$).
Consequently for $t$ large and $\alpha_0$ small, we have:
\begin{align*}
|\dot{x}(t)|&\le \frac{2}{\|\partial_xR_\beta\|^2_{H^1\times L^2}}\left|\left\langle W,\left(\begin{array}{cc}
\beta\partial_x&\partial_x^2-Id+f'(\tilde{Q}_\beta)\\
Id&\beta\partial_x
\end{array}\right)\partial_x\tilde{R}_\beta\right\rangle\right|+\mathrm{O}\left(\|W(t)\|_{H^1\times L^2}+e^{-2e_\beta t}\right)\\
&\le K_3\left(\|W(t)\|_{H^1\times L^2}+e^{-2e_\beta t}\right),
\end{align*} for some constant $K_3>0$.
Finally, take $K_1:=\max(K_2,K_3)$ to obtain Lemma \ref{lem_modulation}.
\end{proof}

\subsubsection{Step 2: Control of particular directions}

Let us denote 
\begin{equation}
\alpha_{\pm,\beta}(t):=\langle W(t),\tilde{Z}_{\pm,\beta}(t)\rangle,
\end{equation}
where $\tilde{Z}_{\pm,\beta}(t):=Z_{\pm,\beta}(t,\cdot-x(t))$.

\begin{Lem}\label{lem_der_alpha_pm}
We have for all $t\in [t_n^*,S_n]$,
$$\left|\frac{d\alpha_{+,\beta}}{dt}(t)-e_\beta\alpha_{+,\beta}(t)\right|+\left|\frac{d\alpha_{-,\beta}}{dt}(t)+e_\beta\alpha_{-,\beta}(t)\right|=\mathrm{O}\left(\|W(t)\|_{H^1\times L^2}^2+e^{-2e_\beta t}\right).$$
\end{Lem}

\begin{proof}
Observe that 
\begin{align*}
\partial_t\tilde{Z}_{\pm,\beta}(t,x)&=-(\beta+\dot{x}(t))\partial_xZ_{\pm,\beta}(t,x-x(t))\\
&=-(\beta+\dot{x}(t))\partial_x\tilde{Z}_{\pm,\beta}(t,x).
\end{align*}
Thus
\begin{equation}\label{lem_alpha_aux_1}
\frac{d\alpha_{\pm,\beta}}{dt}=\langle \partial_tW,\tilde{Z}_{\pm,\beta}\rangle+\langle W,-\beta \partial_x\tilde{Z}_{\pm,\beta}\rangle + \langle W,-\dot{x}\partial_x\tilde{Z}_{\pm,\beta}\rangle.
\end{equation}
By \eqref{est_x_dot}, we obtain
\begin{equation}\label{lem_alpha_aux_2}
\langle W,-\dot{x}\partial_x\tilde{Z}_{\pm,\beta}\rangle=\mathrm{O}\left(\|W(t)\|_{H^1\times L^2}^2+e^{-2e_\beta t}\right).
\end{equation}
In addition, defining $\tilde{\mathscr{H}}_\beta$  like $\mathscr{H}_\beta$ by replacing $Q_\beta$ by $\tilde{Q}_\beta$, we have by \eqref{eq_W}
\begin{equation}\label{lem_alpha_aux_3}
\begin{aligned}
\langle \partial_tW,\tilde{Z}_{\pm,\beta}\rangle+\langle W,-\beta \partial_x\tilde{Z}_{\pm,\beta}\rangle =&\;\langle W,\tilde{\mathscr{H}}_\beta\tilde{Z}_{\pm,\beta}\rangle + \dot{x}\langle \partial_x\tilde{R}_\beta,\tilde{Z}_{\pm,\beta}\rangle\\
&+Ae^{-e_\beta t}\left\langle \tilde{Y}_{+,\beta},\tilde{\mathscr{H}}_\beta\tilde{Z}_{\pm,\beta}\right\rangle+\mathrm{O}\left(\|W(t)\|_{H^1\times L^2}^2+e^{-2e_\beta t}\right)\\
&+Ae^{-e_\beta t}\left(\dot{x}\langle \partial_x\tilde{Y}_{+,\beta},\tilde{Z}_{\pm,\beta}\rangle+e_\beta\langle \tilde{Y}_{+,\beta},\tilde{Z}_{\pm,\beta}\rangle\right).
\end{aligned}
\end{equation}

Since $$\langle \tilde{Y}_{+,\beta}, \tilde{Z}_{\pm,\beta}\rangle=\langle Y_{+,\beta}, Z_{\pm,\beta}\rangle=\begin{dcases}
1 & \text{if we take } \pm=- \\
0 & \text{if we take } \pm=+, 
\end{dcases}$$
we always obtain
\begin{equation}\label{lem_alpha_aux_4}
Ae^{-e_\beta t}\left\langle \tilde{Y}_{+,\beta},\tilde{\mathscr{H}}_\beta\tilde{Z}_{\pm,\beta}\right\rangle+Ae^{-e_\beta t}e_\beta\langle \tilde{Y}_{+,\beta},\tilde{Z}_{\pm,\beta}\rangle=0.
\end{equation}
We have in addition $$|e^{-e_\beta t}\dot{x}|\leq \frac{1}{2}\left(e^{-2e_\beta t}+\dot{x}^2\right)=\mathrm{O}\left(\|W(t)\|_{H^1\times L^2}^2+e^{-2e_\beta t}\right)$$
and $\langle \partial_x\tilde{R}_{\beta},\tilde{Z}_{+,\beta}\rangle$ by Proposition \ref{prop_spectral}.

Hence, gathering \eqref{lem_alpha_aux_1}, \eqref{lem_alpha_aux_2}, \eqref{lem_alpha_aux_3}, and \eqref{lem_alpha_aux_4}, we infer:
$$
\frac{d\alpha_{\pm,\beta}}{dt}=\pm e_\beta\langle W,\tilde{Z}_{\pm,\beta}\rangle+\mathrm{O}\left(\|W(t)\|_{H^1\times L^2}^2+e^{-2e_\beta t}\right),$$
which proves Lemma \ref{lem_der_alpha_pm}.
\end{proof}

\subsubsection{Step 3: Exponential control of $\|W_n\|_{H^1\times L^2}$}

Let us introduce the functional 
\begin{equation}\label{def_F_W}
\mathcal{F}_W(t):=\langle \tilde{H}_\beta W(t),W(t)\rangle.
\end{equation} (We recall that $\tilde{H}_\beta$ is defined in \eqref{def_H_beta}.)

\begin{Lem}[Control of $\mathcal{F}_W$]\label{lem_control_F}
There exists $C>0$ such that for $t_0$ sufficiently large, we have for all $n$ such that $S_n\ge t_0$, for all $t\in[t_n^*,S_n]$:
$$\left|\mathcal{F}_W(t)\right|\le C\left(\|W(t)\|^3_{H^1\times L^2}+e^{-3e_\beta t}+e^{-e_\beta t}|\alpha_{+,\beta}(t)|\right).$$
\end{Lem}

\begin{proof}
We have
\begin{equation}\label{lem_aux_F_W_1}
\begin{aligned}
\mathcal{F}_W(t)&=\langle \tilde{H}_\beta(U-\tilde{R}_\beta),U-\tilde{R}_\beta\rangle -2Ae^{-e_\beta t}\langle \tilde{H}_\beta\tilde{Y}_{+,\beta},W\rangle\\
&=\langle \tilde{H}_\beta(U-\tilde{R}_\beta),U-\tilde{R}_\beta\rangle -2Ae^{-e_\beta t}\alpha_{+,\beta}.
\end{aligned}
\end{equation}
Notice that we have used $ \tilde{H}_\beta\tilde{Y}_{+,\beta}=\tilde{Z}_{+,\beta}$ and $\langle\tilde{Y}_{+,\beta},\tilde{Z}_{+,\beta}\rangle=0$. \\

Now let us focus on the quadratic term $\langle \tilde{H}_\beta(U-\tilde{R}_\beta),U-\tilde{R}_\beta\rangle$. This term rewrites $$\langle \tilde{H}_\beta(U-\tilde{R}_\beta),U-\tilde{R}_\beta\rangle=\int_\RR\left\{\varepsilon_2^2+(\partial_x\varepsilon_1)^2+\varepsilon_1^2-f'(\tilde{Q}_\beta)\varepsilon_1^2+2\beta\varepsilon_2\partial_x\varepsilon_1\right\}(t,x)\;dx,$$
where $\varepsilon_1$ and $\varepsilon_2$ are defined as follows:
 $$\varepsilon_1(t,x):=u(t,x)-Q_\beta(t,x-x(t)) \quad\text{and}\quad \varepsilon_2(t,x):=\partial_tu(t,x)-\partial_tQ_\beta(t,x-x(t)).$$
In a compact manner, we can write:
$$\varepsilon_1:=u-\tilde{Q_\beta}\quad\text{and}\quad \varepsilon_2:=\partial_tu+\beta\partial_x\tilde{Q_\beta}$$ since $\partial_tQ_\beta(t,x)=-\beta\partial_xQ_\beta(t,x)$.

We observe that 
\begin{align*}
\MoveEqLeft[4]
\int_\RR\left\{(\partial_tQ_\beta(t,x-x(t)))^2+(\partial_xQ_\beta(t,x-x(t)))^2+(Q_\beta(t,x-x(t)))^2\right\}\;dx\\
&=\int_\RR\left\{(\partial_tQ_\beta(t,x))^2+(\partial_xQ_\beta(t,x))^2+(Q_\beta(t,x))^2\right\}\;dx
\end{align*}
and
$$\int_\RR\partial_xQ_\beta(t,x-x(t))\partial_tQ_\beta(t,x-x(t))\;dx=\int_\RR\partial_xQ_\beta(t,x)\partial_tQ_\beta(t,x)\;dx.$$

Considering that $u_n$ and $Q_\beta$ are solutions of \eqref{NLKG}, the energy and the momentum of these solutions as defined in introduction are conserved. Thus, there exists $C_n\in\RR$ such that
\begin{align*}
\MoveEqLeft[4]
\langle \tilde{H}_\beta(U-\tilde{R}_\beta),U-\tilde{R}_\beta\rangle (t)\\
=&\;C_n+2\int_\RR\left\{F(u(t,x))-F(\tilde{Q}_\beta(t,x))-\frac{1}{2}f'(\tilde{Q}_\beta(t,x))\varepsilon_1(t,x)^2\right\}\;dx\\
&-2\int_\RR\partial_tu(t,x)\left(-\beta\partial_x\tilde{Q}_\beta(t,x)\right)\;dx-2\int_\RR\partial_xu(t,x)\partial_x\tilde{Q}_\beta(t,x)\;dx-2\int_\RR u(t,x)\tilde{Q}_\beta(t,x)\;dx\\
&-2\beta\int_\RR\partial_tu(t,x)\partial_xQ_\beta(t,x-x(t))\;dx-2\beta\int_\RR\partial_xu(t,x)\partial_tQ_\beta(t,x-x(t))\;dx.
\end{align*}

Now, we notice that $$\int_\RR\partial_xu(t,x)\partial_x\tilde{Q}_\beta(t,x)\;dx=-\int_\RR u(t,x)\partial_x^2\tilde{Q}_\beta(t,x)\;dx$$ and
$$\partial_x\tilde{Q}_\beta(t,x)=\partial_xQ_\beta(t,x-x(t)).$$
Thus
\begin{equation}\label{lem_aux_F_W_2}
\begin{aligned}
\MoveEqLeft[4]
\langle \tilde{H}_\beta(U-\tilde{R}_\beta),U-\tilde{R}_\beta\rangle (t)\\
=&\;  C_n+2\int_\RR\left\{F(u)-F(\tilde{Q}_\beta)-f(\tilde{Q}_\beta)\varepsilon_1-\frac{1}{2}f'(\tilde{Q}_\beta)\varepsilon_1^2\right\}(t,x)\;dx\\
&+2\int_\RR u(t,x)\left((1-\beta^2)\partial_x^2\tilde{Q}_\beta(t,x)-\tilde{Q}_\beta(t,x)+f(\tilde{Q}_\beta(t,x))\right)\;dx.
\end{aligned}
\end{equation}
The last integral is zero by the equation satisfied by $Q_\beta$. By means of Taylor inequality, we claim
\begin{equation}\label{lem_aux_F_W_3}
\begin{aligned}
\MoveEqLeft[6]
\int_\RR\left\{F(u(t,x))-F(\tilde{Q}_\beta(t,x))-f(\tilde{Q}_\beta(t,x))\varepsilon_1(t,x)-\frac{1}{2}f'(\tilde{Q}_\beta(t,x))\varepsilon_1(t,x)^2\right\}\;dx\\
&=\mathrm{O}\left(\|U-\tilde{R}_\beta(t)\|_{H^1\times L^2}^3\right)\\
&=\mathrm{O}\left(\|W(t)\|_{H^1\times L^2}^3+e^{-3e_\beta t}\right).
\end{aligned}
\end{equation}
Hence, collecting \eqref{lem_aux_F_W_1}, \eqref{lem_aux_F_W_2}, and \eqref{lem_aux_F_W_3}, we have:
$$\mathcal{F}_W(t)=C_n+\mathrm{O}\left(\|W(t)\|^3_{H^1\times L^2}+e^{-3e_\beta t}+e^{-e_\beta t}|\alpha_{+,\beta}(t)|\right).$$

On the other hand, we immediately have $|\mathcal{F}_W(t)|\le C\|W(t)\|^2_{H^1\times L^2}$. Since $W(S_n)=0$ and $\alpha_{+,\beta}(S_n)=0$, we deduce that $|C_n|\le Ce^{-3e_\beta S_n}\le Ce^{-3e_\beta t}$ for all $t\in[t^*_n,S_n]$.
\end{proof}

\begin{Cor}
There exists $C>0$ such that for $t_0$ sufficently large, we have for all $n$ such that $S_n\ge t_0$, for all $t\in[t_n^*,S_n]$:
$$|\mathcal{F}_W(t)|\le C\left(\|W(t)\|^3_{H^1\times L^2}+e^{-e_\beta t}\sup_{t'\in [t,S_n]}\|W(t')\|^2_{H^1\times L^2}+e^{-3e_\beta t}\right).$$
\end{Cor}

\begin{proof}
From Lemma \ref{lem_der_alpha_pm} we deduce that $$\forall\;t\in[t^*_n,S_n],\qquad \left|\frac{d}{dt}\left(e^{-e_\beta t}\alpha_{+,\beta}\right)\right|\le C\left(e^{-e_\beta t}\|W\|_{H^1\times L^2}^2+e^{-3e_\beta t}\right).$$
Given that $e^{-e_\beta S_n}\alpha_{+,\beta}(S_n)=0$, we deduce by integration of the preceding inequality that:
$$\forall\;t\in[t^*_n,S_n],\qquad\left|e^{-e_\beta t}\alpha_{+,\beta}(t)\right|\le C\left(e^{-e_\beta t}\sup_{t'\in [t,S_n]}\|W(t')\|_{H^1\times L^2}^2+e^{-3e_\beta t}\right).$$
Now the corollary follows from the combination of this last result with Lemma \ref{lem_control_F}.
\end{proof}

\begin{Lem}\label{est_alpha_der_alpha}
There exists $c>0$ such that for $\alpha_0>0$ sufficiently small and $t_0\ge 0$ sufficiently large, for all $n$ such that $S_n\ge t_0$, for all $t\in[t^*_n,S_n]$:
\begin{equation}\label{syst_est_alpha}
\begin{dcases}
\left|\frac{d}{dt}\alpha_{+,\beta}(t)-e_\beta\alpha_{+,\beta}(t)\right|\le c\min\left\{\left(\sup_{t'\in [t,S_n]}\left(\alpha_{+,\beta}^2(t')+\alpha_{-,\beta}^2(t')\right)+e^{-2e_\beta t}\right),\alpha_{+,\beta}^2(t)+\alpha_{-,\beta}^2(t)+e^{-e_\beta  t}\right\}\\
\left|\frac{d}{dt}\alpha_{-,\beta}(t)+e_\beta\alpha_{-,\beta}(t)\right|\le c\min\left\{\left(\sup_{t'\in [t,S_n]}\left(\alpha_{+,\beta}^2(t')+\alpha_{-,\beta}^2(t')\right)+e^{-2e_\beta t}\right),\alpha_{+,\beta}^2(t)+\alpha_{-,\beta}^2(t)+e^{-e_\beta  t}\right\}.
\end{dcases}
\end{equation}
\end{Lem}

\begin{proof}
Due to Proposition \ref{prop_coercivity}, we have on the one hand
\begin{equation}\label{est_lem_alpha_aux_1}
\|W\|_{H^1\times L^2}^2\le \frac{1}{\mu}\mathcal{F}_W(t)+\frac{1}{\mu^2}\left(\alpha_{+,\beta}^2+\alpha_{-,\beta}^2\right).
\end{equation}
On the other, 
\begin{equation}\label{est_lem_alpha_aux_2}
|\mathcal{F}_W(t)|\le C\left(\|W\|_{H^1\times L^2}^3+e^{-e_\beta t}\sup_{t'\in [t,S_n]}\|W(t')\|_{H^1\times L^2}^2+e^{-3e_\beta t}\right).
\end{equation}
Now, inserting \eqref{est_lem_alpha_aux_2} into \eqref{est_lem_alpha_aux_1}, we obtain
\begin{equation}\label{est_lem_alpha_aux_3}
\|W\|_{H^1\times L^2}^2\le \frac{2C}{\mu}e^{-3e_\beta t}+\frac{2}{\mu^2}\sup_{t'\in [t,S_n]}\left(\alpha_{+,\beta}^2(t')+\alpha_{-,\beta}^2(t')\right)
\end{equation}
provided $\alpha_0$ is chosen small enough such that $\frac{C}{\mu}\sup_{t'\in [t,S_n]}\|W(t')\|_{H^1\times L^2}\le \frac{1}{4}$ (for all $t$) and $t_0$ is chosen large enough such that $\frac{ C}{\mu}e^{-e_\beta t_0}\le \frac{1}{4}$. \\

Then, combining \eqref{est_lem_alpha_aux_3} and the estimates obtained in Lemma \ref{lem_der_alpha_pm}, we deduce:
\begin{align*}
\left|\frac{d\alpha_{\pm,\beta}}{dt}\mp e_\beta\alpha_{\pm,\beta}\right|&\le C\left(\|W\|_{H^1\times L^2}^2+e^{-2e_\beta t}\right)\\
&\le C\left(\sup_{t'\in [t,S_n]}\left(\alpha_{+,\beta}^2(t')+\alpha_{-,\beta}^2(t')\right)+e^{-2e_\beta t}\right).
\end{align*}
In order to avoid the supremum in front of the expression $\alpha_{+,\beta}^2+\alpha_{-,\beta}^2$ and hence to obtain the second way of estimating $\frac{d}{dt}\alpha_{\pm,\beta}\mp e_\beta\alpha_{\pm,\beta}$ which is described by Lemma \ref{est_alpha_der_alpha}, we can rewrite \eqref{est_lem_alpha_aux_2} more simply as follows
$$
|\mathcal{F}_W(t)|\le C\left(\|W\|_{H^1\times L^2}^3+e^{-e_\beta t}+e^{-3e_\beta t}\right).
$$
Then the analog of \eqref{est_lem_alpha_aux_3} takes the following form:
$$\|W\|_{H^1\times L^2}^2\le Ce^{-e_\beta t}+\left(\alpha_{+,\beta}^2(t)+\alpha_{-,\beta}^2(t)\right);$$ the last arguments remain unchanged. However note that, in this case, the obtained estimates are at the cost of a worse exponential term. 
\end{proof}

\begin{proposition}\label{prop_est_expo}
There exists $C>0$ such that for all $\alpha_0\ge 0$ sufficiently small, for all $t_0$ sufficiently large, for all $n$ such that $S_n\ge t_0$, for all $t\in[t^*_n,S_n]$, the following inequalities hold:
\begin{equation}\label{est_alpha_-+}
|\alpha_{-,\beta}(t)|+|\alpha_{+,\beta}(t)|\le Ce^{-2e_\beta t}
\end{equation}
and \begin{equation}\label{est_W_aux}
\|W(t)\|_{H^1\times L^2}\le Ce^{-\frac{3}{2}e_\beta t}.
\end{equation}
\end{proposition}

\begin{proof}
Let us start with the system \eqref{syst_est_alpha} obtained in Lemma \ref{est_alpha_der_alpha}. We have in particular
\begin{equation}\label{syst_est_alpha'}
\begin{dcases}
\left|\frac{d}{dt}\alpha_{+,\beta}-e_\beta\alpha_{+,\beta}\right|\le c\left(\alpha_{+,\beta}^2(t)+\alpha_{-,\beta}^2(t)+e^{-e_\beta  t}\right)\\
\left|\frac{d}{dt}\alpha_{-,\beta}+e_\beta\alpha_{-,\beta}\right|\le c\left(\alpha_{+,\beta}^2(t)+\alpha_{-,\beta}^2(t)+e^{-e_\beta  t}\right).
\end{dcases}
\end{equation}
Taking some inspiration in \cite[paragraph 4.4.2]{combet2010}, we will first show the existence of $M\ge 0$ such that for all $t_0$ large enough, for all $t\in[t_n^*,S_n]$,
\begin{equation}\label{est_alpha+_aux}
\left|\alpha_{+,\beta}(t)\right|\le M\left(\alpha_{-,\beta}^2(t)+e^{-e_\beta t}\right).
\end{equation}

In order to prove \eqref{est_alpha+_aux}, let us consider, for some positive constants $M$ and $\tilde{M}$, the function
$$h:t\mapsto \alpha_{+,\beta}(t)-M\alpha_{-,\beta}^2(t)-\tilde{M}e^{-e_\beta t}$$
and show that it is always negative on $[t^*_n,S_n]$, provided $M$ and $\tilde{M}$ are well chosen. \\
We compute
$$
h'(t)=\alpha_{+,\beta}'(t)-2M\alpha_{-,\beta}'(t)\alpha_{-,\beta}(t)+\tilde{M}e_\beta e^{-e_\beta t}.
$$

By \eqref{syst_est_alpha'}, we obtain
\begin{align*}
h'\ge &\; e_\beta\alpha_{+,\beta}-c\left(\alpha_{+,\beta}^2+\alpha_{-,\beta}^2+e^{-e_\beta t}\right)\\
&-2M\left(-e_\beta\alpha_{-,\beta}^2+c|\alpha_{-,\beta}|\left(\alpha_{+,\beta}^2+\alpha_{-,\beta}^2+e^{-e_\beta t}\right)\right)+\tilde{M}e_\beta e^{-e_\beta t}.
\end{align*}
Replacing $\alpha_{+,\beta}$ by its expression in terms of $h$, $\alpha_{-,\beta}^2$, and $e^{-e_\beta t}$ and using that
$$\left(h+M\alpha_{-,\beta}^2+\tilde{M}e^{-e_\beta t}\right)^2\le 2h^2+4M^2\alpha_{-,\beta}^4+4\tilde{M}^2e^{-2e_\beta t}$$
 lead to the following estimate of $h'$:
\begin{align*}
h'\ge &\; e_\beta h-2c\left(1+2M|\alpha_{-,\beta}|\right)h^2\\
&+\left(e_\beta\tilde{M}-c-2Mc|\alpha_{-,\beta}|-4\tilde{M}^2ce^{-e_\beta t}-8M\tilde{M}^2c|\alpha_{-,\beta}|e^{-e_\beta t}\right)e^{-e_\beta t}\\
&+\alpha_{-,\beta}^2\left(3Me_\beta-c-2Mc|\alpha_{-,\beta}|-4cM^2\alpha_{-,\beta}^2-8M^3c|\alpha_{-,\beta}|\alpha_{-,\beta}^2\right).
\end{align*}
Now we choose $M:=\frac{c}{e_\beta}$ and $\tilde{M}:=\frac{2c}{e_\beta}$ so that the expressions
$$3Me_\beta-c-2Mc|\alpha_{-,\beta}|-4cM^2\alpha_{-,\beta}^2-8M^3c|\alpha_{-,\beta}|\alpha_{-,\beta}^2$$ and
$$ e_\beta\tilde{M}-c-2Mc|\alpha_{-,\beta}|-4\tilde{M}^2ce^{-e_\beta t}-8M\tilde{M}^2c|\alpha_{-,\beta}|e^{-e_\beta t}$$ are positive for $\alpha_0$ sufficiently small and $t_0$ large enough.
Thus for such values of $\alpha_0$ and $t_0$, there exists $c_M>0$ such that for all $t\in[t_n^*,S_n]$,
$$h'(t)\ge e_\beta h(t)-c_Mh(t)^2.$$
At this stage, we can deduce that for all $t$ in $[t^*_n,S_n]$, $h(t)\le 0$. Assume for the sake of contradiction that there exists $\tilde{t}^*_n\in[t^*_n,S_n]$ such that $h(\tilde{t}^*_n)>0$. Then, we can define
$$T:=\sup\{t\in[\tilde{t}^*_n,S_n],\:h(t)>0\}.$$ We necessarily have $h(T)=0$. Indeed, $h(T)<0$ is excluded by continuity of $h$ in $T$ (on the left side) and by definition of the supremum; if $h(T)>0$, then $T<S_n$ (since $h(S_n)=-\tilde{M}e^{-e_\beta S_n}$) which leads once more to a contradiction, using the continuity of $h$ in $T$ (on the right side) and the definition of $T$ as a supremum. It follows that $h'(T)\ge 0$; thus $h$ is non-decreasing in the neighborhood of $T$ and in particular, $h(t)\le 0$ on $[T-\eta,T]$ for some $\eta>0$. This again contradicts the definition of $T$. \\
Hence, for all $t$ in $[t_n^*,S_n]$, $h(t)\le 0$. Given that $-\alpha_{+,\beta}$ satisfies the same differential system as $\alpha_{+,\beta}$, we finally obtain \eqref{est_alpha+_aux}. \\

Now, we have for all $t\in [t^*_n,S_n]$,
$$\sup_{t'\in[t,S_n]}|\alpha_{+,\beta}(t')|\le M\left(\sup_{t'\in[t,S_n]}\alpha_{-,\beta}^2(t')+e^{-e_\beta t}\right).$$
By using Lemma \ref{est_alpha_der_alpha} and even if it means reducing $\alpha_0$ and increasing $t_0$, we obtain that for all $t\in [t^*_n,S_n]$,
$$
|\alpha_{-,\beta}'(t)+e_\beta\alpha_{-,\beta}(t)|\leq C\left(\sup_{t'\in[t,S_n]}\alpha_{-,\beta}^2(t')+e^{-2e_\beta t}\right)\le \frac{e_\beta}{10}|\alpha_{-,\beta}(t)|+Ce^{-2e_\beta t}.
$$ 
We claim that this implies the estimates in Proposition \ref{prop_est_expo}. The preceding inequality rewrites as follows:
$$\left|\frac{d}{dt}(e^{e_\beta t}\alpha_{-,\beta})(t)\right|\le \frac{e_\beta}{10}e^{e_\beta t}\sup_{t'\in[t,S_n]}|\alpha_{-,\beta}(t')|+Ce^{-e_\beta t}.$$
For $t$ belonging to $[t_n^*,S_n]$, we obtain by integration on $[t,S_n]$:
$$
e^{e_\beta t}|\alpha_{-,\beta}(t)|\le Ce^{-e_\beta t}+\frac{e_\beta}{10}\int_t^{S_n}e^{e_\beta s}\sup_{t'\in[s,S_n]}|\alpha_{-,\beta}(t')|\;ds.
$$
Passing to the supremum and defining $y(t):=e^{e_\beta t}\sup_{t'\in[t,S_n]}|\alpha_{-,\beta}(t')|$ on $[t_n^*,S_n]$, this leads to
\begin{equation}\label{est_gronwall}
y(t)\le Ce^{-e_\beta t}+\frac{e_\beta}{10}\int_t^{S_n}y(s)\;ds.
\end{equation}
Now, a standard Grönwall argument allows us to see that $y(t)\le Ce^{-e_\beta t}$, which precisely provides the expected estimate of the parameter $\alpha_{-,\beta}$ in Proposition \ref{prop_est_expo}. Then, the similar estimate of the parameter $\alpha_{+,\beta}$ follows from the integration of the inequality
$$|\alpha_{+,\beta}'(t)-e_\beta\alpha_{+,\beta}(t)|\leq Ce^{-2e_\beta t}$$ and finally \eqref{est_W_aux} follows from \eqref{est_lem_alpha_aux_3} and \eqref{est_alpha_-+}. \\
For the reader's convenience, let us write explicitly the Grönwall argument. The function $\psi(t):=e^{\frac{e_\beta}{10} t}\int_t^{S_n}y(s)\;ds$ is $\mathscr{C}^1$ on $[t^*_n,S_n]$ and for all $t\in[t^*_n,S_n]$,
$$\psi'(t)=e^{\frac{e_\beta}{10} t}\left(-y(t)+\frac{e_\beta}{10}\int_t^{S_n}y(s)\;ds\right)\ge -Ce^{-\frac{9e_\beta}{10}t}$$
by \eqref{est_gronwall}. Observing that $\psi(S_n)=0$, it follows that $\psi(t)\le Ce^{-\frac{9e_\beta}{10}t}$. As a consequence, $$\int_t^{S_n}y(s)\;ds\le Ce^{-e_\beta t}$$ and thus, by \eqref{est_gronwall} again, $$\forall\;t\in[t^*_n,S_n],\qquad y(t)\le Ce^{-e_\beta t}.$$
\end{proof}

\subsubsection{Step 4: Improvement of the exponential decay rate}

The goal of this paragraph is to optimize the exponential decay rate in the estimate of $\|W\|_{H^1\times L^2}$.
We actually prove

\begin{proposition}\label{prop_est_W_improved}
The following estimate holds for all $t\in[t^*_n,S_n]$:
 $$\|W(t)\|_{H^1\times L^2}\le Ce^{-2e_\beta t}.$$
\end{proposition}

\begin{proof}
Let us consider this time the derivative of $\mathcal{F}_W$ (and not the functional $\mathcal{F}_W$ itself). From the definition \eqref{def_F_W} and the symmetric property of $\tilde{H}_\beta$ for $\langle\cdot,\cdot\rangle$, we immediately obtain:
$$
\frac{d}{dt}\mathcal{F}_W(t)=2\langle \tilde{H}_\beta W,\partial_tW\rangle+\beta\int_\RR\partial_x\tilde{Q}_\beta f''(\tilde{Q}_\beta )w_1^2\;dx,
$$
where $w_1$ (resp. $w_2$) is the first (resp. the second) component of $W$. Replacing $\partial_tW$ by its expression obtained in \eqref{eq_W} and noticing that
$$\left(\begin{array}{cc}
\beta\partial_x&Id\\
\partial_x^2-Id+f'(\tilde{Q}_\beta)&\beta\partial_x
\end{array}\right)\tilde{Y}_{+,\beta}=J\tilde{H}_\beta\tilde{Y}_{+,\beta}=J\tilde{Z}_{+,\beta},$$ we have:
\begin{align*}
\MoveEqLeft[2]
\langle \tilde{H}_\beta W,\partial_tW\rangle\\
=&\left\langle \tilde{H}_\beta W,\left(\begin{array}{cc}
0&Id\\
\partial_x^2-Id+f'(\tilde{Q}_\beta)&0
\end{array}\right)W\right\rangle +\left\langle \tilde{H}_\beta W,\dbinom{0}{f(u)-f(\tilde{Q}_\beta)-f'(\tilde{Q}_\beta)(u-\tilde{Q}_\beta)}\right\rangle \\
&+\langle \tilde{H}_\beta W,J\tilde{Z}_{+,\beta}\rangle +e_\beta Ae^{-e_\beta t}\langle W,\tilde{Z}_{+,\beta}\rangle+\dot{x}\langle \tilde{H}_\beta W,\partial_x\tilde{R}_\beta\rangle+Ae^{-e\beta t}\dot{x}\langle W,\tilde{H}\partial_x\tilde{Y}_{+,\beta}\rangle.
\end{align*}

Let us analyze each term appearing in the preceding decomposition. 
Since $\tilde{H}_\beta$ is self-adjoint, we infer:
\begin{equation}
\langle \tilde{H}_\beta W,\partial_x\tilde{R}_\beta\rangle=\langle W,\tilde{H}_\beta\partial_x\tilde{R}_\beta\rangle=0
\end{equation} and
\begin{equation}
\langle \tilde{H}_\beta W,J\tilde{Z}_{+,\beta}\rangle=-Ae^{-e_\beta t}\langle W,\tilde{\mathscr{H}_\beta}\tilde{Z}_{+,\beta}\rangle=-e_\beta Ae^{-e_\beta t}\langle W,\tilde{Z}_{+,\beta}\rangle.
\end{equation}
Moreover, 
\begin{equation}
\left|\left\langle \tilde{H}_\beta W,\dbinom{0}{f(u)-f(\tilde{Q}_\beta)-f'(\tilde{Q}_\beta)(u-\tilde{Q}_\beta)}\right\rangle \right|\le C\|W\|_{H^1\times L^2}\left(\|W\|_{H^1\times L^2}^2+e^{-2e_\beta t}\right).
\end{equation}
We notice that \eqref{est_x_dot} implies 
\begin{equation}
|Ae^{-e_\beta t}\dot{x}\langle W,\tilde{H}_\beta\partial_x\tilde{Y}_{+,\beta}\rangle|\le Ce^{-e_\beta t}\|W(t)\|^2_{H^1\times L^2}+Ce^{-3e_\beta t}\|W(t)\|_{H^1\times L^2}.
\end{equation}
Defining $\tilde{T}:=-\partial_x^2+Id-f'(\tilde{Q}_\beta)$, it remains us to examine 
\begin{align*}
\MoveEqLeft[4]
\left\langle \tilde{H}_\beta W,\left(\begin{array}{cc}
0&Id\\
-\tilde{T}&0
\end{array}\right)W\right\rangle  \\
&=\left\langle \left(\begin{array}{cc}
\tilde{T}&0\\
0&Id
\end{array}\right) W,\left(\begin{array}{cc}
0&Id\\
-\tilde{T}&0
\end{array}\right)W\right\rangle +\left\langle \left(\begin{array}{cc}
0&-\beta\partial_x\\
\beta\partial_x&0
\end{array}\right)W,\left(\begin{array}{cc}
0&Id\\
-\tilde{T}&0
\end{array}\right)W\right\rangle.
\end{align*}
On the one hand, we observe that
$$\left\langle \left(\begin{array}{cc}
\tilde{T}&0\\
0&Id
\end{array}\right) W,\left(\begin{array}{cc}
0&Id\\
-\tilde{T}&0
\end{array}\right)W\right\rangle =0$$ and on the other
\begin{align*}
\left\langle \left(\begin{array}{cc}
0&-\beta\partial_x\\
\beta\partial_x&0
\end{array}\right)W,\left(\begin{array}{cc}
0&Id\\
-\tilde{T}&0
\end{array}\right)W\right\rangle&=\int_\RR\left\{-\beta\partial_xw_2w_2+\beta\partial_xw_1(-\tilde{T}w_1)\right\}\;dx\\
&=\beta\int_\RR\partial_xw_1f'(\tilde{Q}_\beta)w_1\\
&=-\frac{\beta}{2}\int_\RR \partial_x\tilde{Q}_\beta f''(\tilde{Q}_\beta)w_1^2\;dx.
\end{align*}
Gathering the above lines, we deduce
$$
\left|\frac{d\mathcal{F}_W}{dt}\right|\le C\left(\|W\|_{H^1\times L^2}^3+e^{-2e_\beta t}\|W\|_{H^1\times L^2}+e^{-e_\beta t}\|W\|_{H^1\times L^2}^2\right).
$$
Using the estimate of $\|W\|_{H^1\times L^2}$ obtained in Proposition \ref{prop_est_expo}, we infer:
\begin{equation}\label{est_der_FW}
\left|\frac{d\mathcal{F}_W}{dt}\right|\le C\left(e^{-2e_\beta t}\|W\|_{H^1\times L^2}+e^{-e_\beta t}\|W\|_{H^1\times L^2}^2\right).
\end{equation}
By integration of \eqref{est_der_FW}, we obtain for all $t\in[t^*,S_n]$:
\begin{equation}
\begin{aligned}
\left|\mathcal{F}_W(t)\right|&\le C\sup_{\tau\in[t,S_n]}\|W(\tau)\|_{H^1\times L^2}\int_t^{S_n}e^{-2e_\beta \tau}\;d\tau+C\sup_{\tau\in[t,S_n]}\|W(\tau)\|^2_{H^1\times L^2}\int_t^{S_n}e^{-e_\beta \tau}\;d\tau\\
&\le C\sup_{\tau\in[t,S_n]}\|W(\tau)\|_{H^1\times L^2}e^{-2e_\beta t}+C\sup_{\tau\in[t,S_n]}\|W(\tau)\|^2_{H^1\times L^2}e^{-e_\beta t}.
\end{aligned}
\end{equation}
At this stage, using once again the coercivity property provided by Proposition \ref{prop_coercivity} and the estimates on $\alpha_{\pm,\beta}$ given in Proposition \ref{prop_est_expo}, and taking $t_0$ large enough so that $e^{-e_\beta t}<\frac{1}{2}$ for all $t\ge t_0$, we have:
$$\sup_{\tau\in[t,S_n]}\|W(\tau)\|^2_{H^1\times L^2}\le C\left(\sup_{\tau\in[t,S_n]}\|W(\tau)\|_{H^1\times L^2}e^{-2e_\beta t}+e^{-4e_\beta t}\right).$$
We now deduce the existence of $C>0$ such that for all $n$, for all $t\in [t^*_n,S_n]$,
$$\|W(t)\|_{H^1\times L^2}\le Ce^{-2e_\beta t}.$$

\end{proof}

Now, Proposition \ref{prop_boot} is obtained as a corollary of Proposition \ref{prop_est_W_improved}. The triangular inequality implies
$$\|U_n(t)-R_\beta(t)-Ae^{-e_\beta t}Y_{+,\beta}(t)\|_{H^1\times L^2}\le \|W_n(t)\|_{H^1\times L^2}+C|x(t)|$$
and since $x(S_n)=0$, the result follows from the integration of 
$$|\dot{x}(t)|\le C\left(\|W_n(t)\|_{H^1\times L^2}+e^{-2e_\beta t}\right)\le Ce^{-2e_\beta t}.$$

\section{Classification of the asymptotic soliton-like solutions}

Let us consider a solution $u$ of (NLKG), denote $U=\dbinom{u}{\partial_tu}$, and assume that
$$\|U(t)-R_\beta(t)\|_{H^1\times L^2}\to 0,\quad \text{as } t\to+\infty.$$
We want to show that $U$ equals to $U^A$, for some $A\in\RR$. In this section again, we consider the one-dimensional case.

\subsection{Modulation of $U$ and coercivity property}

\begin{Lem}\label{lem_modulation_2}
There exist $T\in\RR$, $C>0$, and $x:[T,+\infty)\to\RR$ of class $\mathscr{C}^1$ such that, denoting 
$$\tilde{R}_\beta(t,x):=\dbinom{Q_\beta(t,x-x(t))}{\partial_tQ_\beta(t,x-x(t))}\quad \text{and}\quad E:=U-\tilde{R}_\beta,$$
we have $\|E(t)\|_{H^1\times L^2}\to 0$ as $t\to+\infty$, $x(t)\to 0$ as $t\to+\infty$, and for all $t\ge T$, 
\begin{equation}\label{eq_modul}
\left\langle E(t),\partial_x\tilde{R}_\beta(t) \right\rangle = 0
\end{equation} and 
\begin{equation}\label{control_xdot}
|\dot{x}(t)|\le C\|E(t)\|_{H^1\times L^2}.
\end{equation}
\end{Lem}

\begin{proof}
Lemma \ref{lem_modulation_2} is proved similarly as Lemma \ref{lem_modulation}.
Defining $$\tilde{Q}_\beta(t,x):=Q_\beta(t,x-x(t)),$$ the equation satisfied by $E$ reads:
\begin{equation}
\begin{aligned}
\partial_tE=&\left(\begin{array}{cc}
0&Id\\
\partial_x^2-Id+f'(\tilde{Q}_\beta)&0\\
\end{array}\right)E+\dot{x}\partial_x\tilde{R}_\beta\\
&+\dbinom{0}{f(u)-f(\tilde{Q}_\beta)-f'(\tilde{Q}_\beta)(u-\tilde{Q}_\beta)}.
\end{aligned}
\end{equation}
\end{proof}

Let us denote $\varepsilon_1(t,x):=u(t,x)-Q_\beta(t,x-x(t))$ and $\varepsilon_2(t,x):=\partial_tu(t,x)-\partial_tQ_\beta(t,x-x(t))$ so that $E(t,x):=\dbinom{\varepsilon_1(t,x)}{\varepsilon_2(t,x)}$. \\
Then, introduce the functional $\mathcal{F}_E$ defined as follows: for all $t\ge T$,
\begin{equation}
\mathcal{F}_E(t):=\int_\RR\left\{\varepsilon_2^2+(\partial_x\varepsilon_1)^2+\varepsilon_1^2-f'(\tilde{Q}_\beta)\varepsilon_1^2+2\beta\varepsilon_2\partial_x\varepsilon_1\right\}(t,x)\;dx.
\end{equation}

With analogous notations as that employed in the previous section, we denote
$$\tilde{Z}_{\pm,\beta}(t,x):=Z_{\pm,\beta}(t,x-x(t)).$$

\begin{proposition}\label{lem_coercivity}
There exists $\mu>0$ such that for all $t\ge T$:
\begin{equation}
\mathcal{F}_E(t)\ge \mu\|E(t)\|_{H^1\times L^2}^2-\frac{1}{\mu}\left[\left\langle E(t),\tilde{Z}_{-,\beta}(t)\right\rangle^2+\left\langle E(t),\tilde{Z}_{+,\beta}(t)\right\rangle^2\right].
\end{equation}
\end{proposition}

\begin{proof}
We observe that $\mathcal{F}_E(t)=\left\langle \tilde{H}_\beta E(t),E(t)\right\rangle$ so that Proposition \ref{lem_coercivity} follows from Proposition \ref{prop_coercivity} and from \eqref{eq_modul}. 
\end{proof}

\subsection{Exponential control of $\|E(t)\|_{H^1\times L^2}$}

We improve the control of $\|E(t)\|_{H^1\times L^2}$, proceeding as in section \ref{sect_constr_one_par}.

\subsubsection{Control of the functional $\mathcal{F}_E$} 

\begin{proposition}\label{prop_control_F}
We have 
\begin{equation}
\mathcal{F}_E(t)=\mathrm{O}\left(\|E(t)\|_{H^1\times L^2}^3\right).
\end{equation}
\end{proposition}

\begin{proof}
Let us take again the proof of Lemma \ref{lem_control_F} (replacing $A$ in that proof by $0$ here). We obtain in the same way the existence of $K\in\RR$ such that

$$
\mathcal{F}_E(t)= K+2\int_\RR\left\{F(u)-F(\tilde{Q}_\beta)-f(\tilde{Q}_\beta)\varepsilon_1-\frac{1}{2}f'(\tilde{Q}_\beta)\varepsilon_1^2\right\}(t,x)\;dx
$$
By Taylor inequality,
$$
\int_\RR\left\{F(u)-F(\tilde{Q}_\beta)-f(\tilde{Q}_\beta)\varepsilon_1-\frac{1}{2}f'(\tilde{Q}_\beta)\varepsilon_1^2\right\}(t,x)\;dx=\mathrm{O}\left(\|E(t)\|_{H^1\times L^2}^3\right).
$$
It follows that 
$$\mathcal{F}_E(t)=K+\mathrm{O}\left(\|E(t)\|_{H^1\times L^2}^3\right).$$

On the other hand, $\mathcal{F}_E(t)=\mathrm{O}\left(\|E(t)\|_{H^1\times L^2}^2\right)$; thus $K=0$ and Proposition \ref{prop_control_F} is proved.
\end{proof}

\subsubsection{Control of the unstable directions} 

Let us denote $$\alpha_{\pm,\beta}(t):=\left\langle E(t),\tilde{Z}_{\pm,\beta}(t)\right\rangle.$$

\begin{Lem}\label{lem_der_alpha}
We have $$\left|\frac{d\alpha_{+,\beta}}{dt}-e_\beta\alpha_{+,\beta}\right|+\left|\frac{d\alpha_{-,\beta}}{dt}+e_\beta\alpha_{-,\beta}\right|=\mathrm{O}\left(\|E(t)\|_{H^1\times L^2}^2\right).$$
\end{Lem}

\begin{proof}
The proof follows the same lines as that of Lemma \ref{lem_der_alpha_pm} (by taking $A=0$).
We compute
\begin{align*}
\frac{d\alpha_{\pm,\beta}}{dt}=&\left\langle \partial_tE,\tilde{Z}_{\pm,\beta}\right\rangle+\left\langle E,\partial_t\tilde{Z}_{\pm,\beta}\right\rangle-\dot{x}\left\langle E(t),\partial_x\tilde{Z}_{\pm,\beta}\right\rangle\\
=&\left\langle E,\left(\begin{array}{cc}
0&Id\\
\partial_x^2-Id+f'(Q_\beta)&0\\
\end{array}\right)\tilde{Z}_{\pm,\beta}\right\rangle+\mathrm{O}\left(\|E\|_{H^1\times L^2}^2\right)\\
&+\dot{x}\left\langle \partial_x\tilde{R}_\beta,\tilde{Z}_{\pm,\beta}\right\rangle+\left\langle E,\left(\begin{array}{cc}
-\beta\partial_x&0\\
0&-\beta\partial_x\\
\end{array}\right)\tilde{Z}_{\pm,\beta}\right\rangle\\
=&\left\langle E,\mathscr{H}_\beta\tilde{Z}_{\pm,\beta}\right\rangle+\mathrm{O}\left(\|E\|_{H^1\times L^2}^2\right)\\
=&\pm e_\beta\alpha_{\pm,\beta}+\mathrm{O}\left(\|E\|_{H^1\times L^2}^2\right),
\end{align*}
which precisely yields Lemma \ref{lem_der_alpha}.
\end{proof}

\begin{proposition}\label{prop_est_alpha}
There exists $C\ge 0$ such that for $t$ large enough,
\begin{equation}\label{est_alpha}
|\alpha_{-,\beta}(t)|\le Ce^{-e_\beta t}\quad \text{and}\quad |\alpha_{+,\beta}(t)|\le Ce^{-2e_\beta t}.
\end{equation}
\end{proposition}

\begin{proof}
For $t$ large enough, we obtain as a consequence of Lemma \ref{lem_coercivity} and Proposition \ref{prop_control_F}:
\begin{equation}\label{est_E_alpha}
\|E(t)\|_{H^1\times L^2}^2\le C\left(\alpha_{+,\beta}^2+\alpha_{-,\beta}^2\right).
\end{equation} 
Now, we deduce from Lemma \ref{lem_der_alpha} the following differential system, for some constant $C\ge 0$ and for all $t$ large enough:
\begin{equation}\label{syst_diff_alpha}
\begin{dcases}
\left|\frac{d\alpha_{+,\beta}}{dt}-e_\beta\alpha_{+,\beta}\right|\leq C\left(\alpha_{+,\beta}^2+\alpha_{-,\beta}^2\right)\\
\left|\frac{d\alpha_{-,\beta}}{dt}+e_\beta\alpha_{-,\beta}\right|\leq C\left(\alpha_{+,\beta}^2+\alpha_{-,\beta}^2\right).
\end{dcases}
\end{equation} 
Then, the argument exposed in \cite[paragraph 4.4.2]{combet2010} shows that
$$|\alpha_{+,\beta}(t)|\le C\alpha_{-,\beta}^2(t)$$
and allows us to conclude to \eqref{est_alpha}.
\end{proof}

\begin{proposition}\label{prop_E}
There exists $C\ge 0$ such that for $t$ large enough,
\begin{equation}
\|E(t)\|_{H^1\times L^2}\le Ce^{-e_\beta t}.
\end{equation}
\end{proposition}

\begin{proof}
This is an immediate consequence of \eqref{est_alpha} and \eqref{est_E_alpha}.
\end{proof}

\subsubsection{Identification of $U$ with $U^A$ for some $A\in\RR$}

\begin{Lem}\label{lem_def_A}
There exist $A\in\RR$ and $C\ge 0$ such that for $t$ sufficiently large,
\begin{equation}\label{est_lem_def_A}
|e^{e_\beta t}\alpha_{-,\beta}(t)-A|\le Ce^{-e_\beta t}.
\end{equation}
\end{Lem}

\begin{proof}
We have from \eqref{syst_diff_alpha}: 
\begin{equation}\label{est_der_alpha_-}
\left|\frac{d}{dt}\left(e^{e_\beta t}\alpha_{-,\beta}(t)\right)\right|\le Ce^{-e_\beta t}.
\end{equation} Thus the derivative of $t\mapsto e^{e_\beta t}\alpha_{-,\beta}(t)$ is integrable in the neighborhood of $+\infty$. Hence, there exists $A\in\RR$ such that $e^{e_\beta t}\alpha_{-,\beta}(t)\to A$ as $t\to+\infty$. We finally obtain Lemma \ref{lem_def_A} by integration of \eqref{est_der_alpha_-}.
\end{proof}

\begin{Lem}\label{lem_def_xinfty}
There exists $x_\infty\in\RR$ such that for $t$ sufficiently large,
\begin{equation}\label{est_x}
|x(t)-x_\infty|\le Ce^{-e_\beta t}.
\end{equation}
\end{Lem}

\begin{proof}
As \eqref{est_lem_def_A} was a consequence of \eqref{est_der_alpha_-}, \eqref{est_x} is obtained by means of the estimate
$$|\dot{x}(t)|\le C\|E(t)\|_{H^1\times L^2}\le Ce^{-e_\beta t}.$$
\end{proof}

Now let $V(t,x):=U(t,x)-U^A(t,x-x_\infty)$ where $A$ and $x_{\infty}$ are defined in Lemma \ref{lem_def_A} and Lemma \ref{lem_def_xinfty} respectively.

\begin{Lem}\label{lem_est_v}
There exists $C\ge 0$ such that for $t$ large enough, $\|V(t)\|_{H^1\times L^2}\le Ce^{-e_\beta t}$.
\end{Lem}

\begin{proof}
Define $V^A(t):=U^A(t)-R_\beta (t)-Ae^{-e_\beta t}Y_{+,\beta}(t)$ and decompose $V(t,x)$ as follows:
\begin{equation}\label{decomp_V_E}
V(t,x)=E(t,x)+\tilde{R}_\beta (t,x)-R_\beta(t,x-x_\infty)-Ae^{-e_\beta t}Y_{+,\beta}(t,x-x_\infty)-V^A(t,x-x_\infty).
\end{equation}
Now, we have 
\begin{equation}\label{est_decomp_V_E}
\tilde{R}_\beta(t,x)-R_\beta(t,x-x_\infty)=(x(t)-x_\infty)\partial_xR_\beta(t,x-x_\infty)+\mathrm{O}\left(|x(t)-x_\infty|^2\right).
\end{equation}
Hence, Lemma \ref{lem_est_v} is obtained as a consequence of Proposition \ref{prop_E}, Lemma \ref{lem_def_xinfty}, and the estimate of $\|V^A(t)\|_{H^1\times L^2}$ given by \eqref{est_U^A_thm}, that is $\|V^A(t)\|_{H^1\times L^2}=\mathrm{O}\left(e^{-2e_\beta t}\right)$.
\end{proof}

Let us decompose $V(t,x)$ as follows:
\begin{equation}\label{decomp_V}
V(t,x)=\alpha_{+,\beta}^A(t)Y_{-,\beta}(t,x-x_\infty)+\alpha_{-,\beta}^A(t)Y_{+,\beta}(t,x-x_\infty)+\lambda(t)\partial_xR_\beta(t,x)+V_\perp(t,x),
\end{equation}
where $\alpha_{+,\beta}^A(t)=\langle V(t),Z_{+,\beta}(t,\cdot-x_\infty)\rangle$, $\alpha_{-,\beta}^A(t)=\langle V(t),Z_{-,\beta}(t,\cdot-x_\infty)\rangle$ and
$$\lambda(t):=\frac{1}{\|\partial_xR_\beta\|_{H^1\times L^2}}\langle V(t)-\alpha_{+,\beta}^A(t)Y_{-,\beta}(t,\cdot-x_\infty)-\alpha_{-,\beta}^A(t)Y_{+,\beta}(t,x-x_\infty),\partial_xR_\beta(t)\rangle.$$
Then, we have
$$\langle V_\perp(t),\partial_xR_\beta(t)\rangle=\langle V_\perp(t),Z_{+,\beta}(t,\cdot-x_\infty)\rangle=\langle V_\perp(t),Z_{-,\beta}(t,\cdot-x_\infty)\rangle=0.$$
Thus, by Proposition \ref{prop_coercivity}, there exists $C\ge 0$ such that
\begin{equation}\label{est_coerc_V}
\|V_\perp(t)\|_{H^1\times L^2}^2\le C\left\langle H_\beta V_\perp(t),V_\perp(t)\right\rangle.
\end{equation}

We claim

\begin{Lem}\label{lem_estimates_V}
The following assertions hold:
\begin{enumerate}
\item $\langle H_\beta V(t),V(t)\rangle=\langle H_\beta V_\perp(t),V_\perp(t)\rangle+2\alpha_{+,\beta}^A(t)\alpha_{-,\beta}^A(t)$.\\
\item  $\frac{d}{dt}\langle H_\beta V(t),V(t)\rangle=\mathrm{O}\left(e^{-e_\beta t}\|V(t)\|_{H^1\times L^2}^2\right)$.
\item $|\alpha_{+,\beta}^A(t)|\le Ce^{-e_\beta t}\|V(t)\|_{H^1\times L^2}$ and $|\alpha_{-,\beta}^A(t)|\le Ce^{-e_\beta t}\int_{t}^{+\infty}\|V(t')\|_{H^1\times L^2}\;dt'$
\item $\lambda'(t)=\mathrm{O}\left(e^{-e_\beta t}\|V(t)\|_{H^1\times L^2}+e^{-e_\beta t}\int_t^{+\infty}\|V(t')\|_{H^1\times L^2}\;dt'+\|V_\perp(t)\|_{H^1\times L^2}\right)$.
\end{enumerate}
\end{Lem}

\begin{proof}
The first assertion in Lemma \ref{lem_estimates_V} is obtained by the decomposition of $V$ in terms of $V_\perp$ \eqref{decomp_V} and by means of the following properties (see Proposition \ref{prop_spectral}):
$$\langle Z_{\pm,\beta},Y_{\pm,\beta}\rangle=0,\quad \langle Z_{\pm,\beta},Y_{\mp,\beta}\rangle=1,\quad\text{and}\quad H_\beta\left(\partial_xR_\beta\right)=0.$$
Let us now prove assertion (2). Let $v:=u-u^A$; recall that $U=\dbinom{u}{\partial_tu}$ and $U^A=\dbinom{u^A}{\partial_tu^A}$.\\
We observe that 
$$\langle H_\beta V(t),V(t)\rangle=\int_\RR\left\{\left(\partial_tv\right)^2+\left(\partial_xv\right)^2+v^2-f'(Q_\beta)v^2+2\beta\partial_tv\partial_xv\right\}(t,x)\;dx.$$
Using the fact that $u$ and $u^A$ satisfy \eqref{NLKG}, we obtain
\begin{align*}
\MoveEqLeft[2]
\frac{d}{dt}\langle H_\beta V(t),V(t)\rangle\\
=&-2\int_\RR\left[f(u)-f(u^A)-f'(Q_\beta)v\right]\partial_tv\;dx\\
&-2\beta\int_\RR\partial_xv\left(f(u)-f(u^A)\right)\;dx+\beta\int_\RR\partial_xQ_\beta f''(Q_\beta)v^2\;dx\\
=&\;\mathrm{O}\left(\|V(t)\|_{H^1\times L^2}^3\right)-2\beta\int_\RR\partial_xv\left(vf'(u^A)+\mathrm{O}(v^2)\right)\;dx+\beta\int_\RR\partial_xQ_\beta f''(Q_\beta)v^2\;dx\\
=&\;\mathrm{O}\left(\|V(t)\|_{H^1\times L^2}^3\right)-2\beta\int_\RR\partial_xv\left(vf'(Q_\beta)+\mathrm{O}(v^2+e^{-e_\beta t}v)\right)\;dx+\beta\int_\RR\partial_xQ_\beta f''(Q_\beta)v^2\;dx\\
=&\;\mathrm{O}\left(\|V(t)\|_{H^1\times L^2}^3+e^{-e_\beta t}\|V(t)\|_{H^1\times L^2}^2\right).
\end{align*}
By Lemma \ref{lem_est_v}, assertion (2) is thus proved. \\
In order to prove (3) and (4), let us write
\begin{equation}
\partial_tV=\left(\begin{array}{cc}
0&Id\\
\partial_x^2-Id+f'(Q_\beta)&0\\
\end{array}\right)V+\dbinom{0}{f(u)-f(u^A)}.
\end{equation}
Then, 
\begin{align*}
\frac{d}{dt}\alpha_{\pm,\beta}^A(t)&=\langle \partial_tV,Z_{\pm,\beta}(t,\cdot-x_\infty)\rangle + \langle V,\partial_tZ_{\pm,\beta}(t,\cdot-x_\infty)\rangle\\
&=\left\langle V,\left(\begin{array}{cc}
-\beta\partial_x&\partial_x^2-Id+f'(Q_\beta)\\
Id&-\beta\partial_x\\
\end{array}\right)Z_{\pm,\beta}(t,\cdot-x_\infty)\right\rangle+\mathrm{O}\left(e^{-e_\beta t}\|V(t)\|_{H^1\times L^2}\right)\\
&=\pm e_\beta\alpha_{\pm,\beta}^A(t)+\mathrm{O}\left(e^{-e_\beta t}\|V(t)\|_{H^1\times L^2}\right)
\end{align*}
Note that we have used that $\partial_tZ_{\pm,\beta}(t,\cdot-x_\infty)=-\beta\partial_xZ_{\pm,\beta}(t,\cdot-x_\infty)$. We then deduce, in a similar way as for Proposition \ref{prop_est_alpha}:
$$|\alpha_{+,\beta}^A(t)|\le Ce^{-e_\beta t}\|V(t)\|_{H^1\times L^2}$$ because $t\mapsto e^{-2e_\beta t}\|V(t)\|_{H^1\times L^2}$ is integrable in $+\infty$ and
$$\left|e^{-e_\beta t}\alpha_{+,\beta}^A(t)\right|\le e^{-e_\beta t}\|V(t)\|_{H^1\times L^2}\underset{t\to+\infty}{\to} 0.$$ \\
In addition we deduce 
$$\left|\frac{d}{dt}\left(e^{e_\beta t}\alpha_{-,\beta}^A(t)\right)\right|\le C\|V(t)\|_{H^1\times L^2}.$$
Arguing similarly as for $\alpha_{+,\beta}^A(t)$, we then obtain $\left|\alpha_{-,\beta}^A(t)\right|\le Ce^{-e_\beta t}\int_t^{+\infty}\|V(t')\|_{H^1\times L^2}\;dt'$. This is due to the fact that $t\mapsto\|V(t)\|_{H^1\times L^2}$ is integrable in $+\infty$ and the fact that, by \eqref{decomp_V_E}, \eqref{est_decomp_V_E}, and \eqref{est_lem_def_A},
\begin{align*}
\alpha_{-,\beta}^A(t)&=\alpha_{-,\beta}(t)-Ae^{-e_\beta t}\langle Y_{+,\beta}(t,\cdot-x_\infty),Z_{-,\beta}(t,\cdot-x_\infty)\rangle+\mathrm{O}(e^{-2e_\beta t})\\
&=\mathrm{O}(e^{-2e_\beta t}),
\end{align*}
which explains that $e^{e_\beta t}\alpha_{-,\beta}^A(t)\to 0$ as $t\to+\infty$.\\
It remains to prove (4). By definition of $\lambda(t)$ and the decomposition \eqref{decomp_V} of $V$,
\begin{align*}
\lambda'(t)=&\;\frac{1}{\|\partial_xR_\beta\|^2_{H^1\times L^2}}\left\langle \partial_tV-\alpha_{+,\beta}^A\partial_tY_{-,\beta}(t,\cdot-x_\infty)-\alpha_{-,\beta}^A\partial_tY_{+,\beta}(t,\cdot-x_\infty),\partial_xR_\beta\right\rangle\\
&-\frac{1}{\|\partial_xR_\beta\|^2_{H^1\times L^2}}\left\langle  (\alpha_{+,\beta}^A)'Y_{-,\beta}(t,\cdot-x_\infty)+(\alpha_{-,\beta}^A)'Y_{+,\beta}(t,\cdot-x_\infty),\partial_xR_\beta \right\rangle\\
&+\frac{1}{\|\partial_xR_\beta\|^2_{H^1\times L^2}}\langle V_\perp,\partial_{xt}R_\beta\rangle\\
=&\;\frac{1}{\|\partial_xR_\beta\|^2_{H^1\times L^2}}\left\langle \partial_tV-\alpha_{+,\beta}^A\partial_tY_{-,\beta}(t,\cdot-x_\infty)-\alpha_{-,\beta}^A\partial_tY_{+,\beta}(t,\cdot-x_\infty),\partial_xR_\beta\right\rangle\\
&+\frac{1}{\|\partial_xR_\beta\|^2_{H^1\times L^2}}\langle V_\perp,\partial_{xt}R_\beta\rangle.\\
\end{align*}
since $\langle Y_{\pm,\beta},\partial_xR_\beta\rangle=0$ (we refer to Proposition \ref{prop_spectral}). Hence,
\begin{align*}
\lambda'(t)=&\;\left\langle \left(\begin{array}{cc}
0&Id\\
\partial_x^2-Id+f'(Q_\beta)&0\\
\end{array}\right)V,\partial_xR_\beta\right\rangle+\mathrm{O}\left(e^{-e_\beta t}\|V\|_{H^1\times L^2}\right)+\frac{1}{\|\partial_xR_\beta\|^2_{H^1\times L^2}}\langle V_\perp,\partial_{xt}R_\beta\rangle\\
&+\frac{1}{\|\partial_xR_\beta\|^2_{H^1\times L^2}}\left\langle \beta\alpha_{+,\beta}^A\partial_xY_{-,\beta}(t,\cdot-x_\infty)+\beta\alpha_{-,\beta}^A\partial_xY_{+,\beta}(t,\cdot-x_\infty),\partial_xR_\beta\right\rangle.
\end{align*}
Now, using that $\partial_xR_\beta=JZ_0$ and $^t J\left(\begin{array}{cc}
\beta \partial_x&Id\\
\partial_x^2-Id+f'(Q_\beta)&\beta\partial_x\\
\end{array}\right)=H_\beta$, we have:

\begin{align*}
\MoveEqLeft[3]
\left\langle \left(\begin{array}{cc}
0&Id\\
\partial_x^2-Id+f'(Q_\beta)&0
\end{array}\right)V,\partial_xR_\beta\right\rangle \\
=&\;\left\langle \left(\begin{array}{cc}
\beta\partial_x&Id\\
\partial_x^2-Id+f'(Q_\beta)&\beta\partial_x\\
\end{array}\right)V,JZ_0\right\rangle-\left\langle \left(\begin{array}{cc}
\beta\partial_x&0\\
0&\beta\partial_x
\end{array}\right)V,\partial_xR_\beta\right\rangle\\
=&\;\langle H_\beta V,Z_0\rangle -\frac{1}{\|\partial_xR_\beta\|^2_{H^1\times L^2}}\left\langle\beta\alpha_{+,\beta}^A\partial_xY_{-,\beta}(t,\cdot-x_\infty)-\beta\alpha_{-,\beta}^A\partial_xY_{+,\beta}(t,\cdot-x_\infty),\partial_xR_\beta\right\rangle \\
&+\frac{1}{\|\partial_xR_\beta\|^2_{H^1\times L^2}}\langle V_\perp,\beta\partial_x^2R_\beta\rangle.
\end{align*}
Finally, let us notice that
\begin{align*}
\langle H_\beta V,Z_0\rangle &= \langle H_\beta V_\perp,Z_0\rangle+\mathrm{O}\left(\alpha_{+,\beta}^A+\alpha_{-,\beta}^A\right)\\
&=\langle V_\perp,H_\beta Z_0\rangle+\mathrm{O}\left(\alpha_{+,\beta}^A+\alpha_{-,\beta}^A\right)\\
&=\mathrm{O}\left(\alpha_{+,\beta}^A+\alpha_{-,\beta}^A+\|V_\perp\|_{H^1\times L^2}\right).
\end{align*}
Hence,
$$\lambda'(t)=\mathrm{O}\left(e^{-e_\beta t}\|V(t)\|_{H^1\times L^2}+|\alpha_{+,\beta}^A|+|\alpha_{-,\beta}^A|+\|V_\perp\|_{H^1\times L^2}\right)$$ and assertion (4) indeed holds.
\end{proof}

It follows from \eqref{est_coerc_V} and (1), (2), and (3) in Lemma \ref{lem_estimates_V} that for $t$ large,
\begin{multline}
\|V_\perp(t)\|_{H^1\times L^2}^2\le C\int_t^{+\infty}e^{-e_\beta t'}\|V(t')\|_{H^1\times L^2}^2\;dt'\\
+Ce^{-2e_\beta t}\|V(t)\|_{H^1\times L^2}\int_t^{+\infty}\|V(t')\|_{H^1\times L^2}\;dt'.
\end{multline}
Since $$\int_t^{+\infty}e^{-e_\beta t'}\|V(t')\|_{H^1\times L^2}^2\;dt'\le e^{-e_\beta t}\sup_{t'\ge t}\|V(t')\|_{H^1\times L^2}\int_t^{+\infty}\|V(t')\|_{H^1\times L^2}\;dt',$$ we deduce that for $t$ large,
\begin{equation}\label{est_Vperp}
\|V_\perp(t)\|_{H^1\times L^2}\le Ce^{-\frac{1}{2}e_\beta t}\sup_{t'\ge t}\|V(t')\|_{H^1\times L^2}^{\frac{1}{2}}\left(\int_t^{+\infty}\|V(t')\|_{H^1\times L^2}\;dt'\right)^{\frac{1}{2}}.
\end{equation}

\begin{Lem}[Estimate of $\|V(t)\|_{H^1\times L^2}$ in terms of $\int_t^{+\infty}\|V(t')\|_{H^1\times L^2}\;dt'$ and $\|V_\perp(t)\|_{H^1\times L^2}$]\label{lem_est_V}
We have the existence of $C\ge 0$ such that for $t$ sufficiently large:
\begin{multline}
\|V(t)\|_{H^1\times L^2}\le Ce^{-e_\beta t}\int_t^{+\infty}\|V(t')\|_{H^1\times L^2}\;dt'\\
+C\left(\|V_\perp(t)\|_{H^1\times L^2}+\int_t^{+\infty}\|V_\perp(t')\|_{H^1\times L^2}\;dt'\right).
\end{multline}
\end{Lem}

\begin{proof}
Using the decomposition \eqref{decomp_V} of $V$, we have:
$$\|V(t)\|_{H^1\times L^2}\le C\left(|\alpha_{+,\beta}^A(t)|+|\alpha_{-,\beta}^A(t)|+|\lambda(t)|+\|V_\perp(t)\|_{H^1\times L^2}\right).$$
By Lemma \ref{lem_estimates_V} which gives estimates of $|\alpha_{+,\beta}^A(t)|$, $|\alpha_{-,\beta}^A(t)|$, and $\lambda'(t)$ in terms of $\|V(t)\|_{H^1\times L^2}$ and $\|V_\perp(t)\|_{H^1\times L^2}$ , we obtain:
\begin{equation}
\lambda(t)=\mathrm{O}\left(e^{-e_\beta t}\int_t^{+\infty}\|V(t')\|_{H^1\times L^2}\;dt'+\int_t^{+\infty}\|V_\perp(t')\|_{H^1\times L^2}\;dt'\right).
\end{equation}
Thus
\begin{align*}
\|V(t)\|_{H^1\times L^2}\le &\;C\left(e^{-e_\beta t}\|V(t)\|_{H^1\times L^2}+e^{-e_\beta t}\int_t^{+\infty}\|V(t')\|_{H^1\times L^2}\;dt'+\|V_\perp(t)\|_{H^1\times L^2}\right)\\
&+C\int_t^{+\infty}\|V_\perp(t')\|_{H^1\times L^2}\;dt'.
\end{align*}
Hence, even if it means taking larger values of $t$, this finishes proving Lemma \ref{lem_est_V}.
\end{proof}

\begin{Lem}\label{est_V_integral}
We have 
\begin{equation}
\int_t^{+\infty}\|V(t')\|_{H^1\times L^2}\;dt'\le Ce^{-e_\beta t}\sup_{t'\ge t}\|V(t')\|_{H^1\times L^2}
\end{equation}
\end{Lem}

\begin{proof}
From \eqref{est_Vperp}, it follows that
\begin{equation}\label{est_Vperp_2}
\begin{aligned}
\int_t^{+\infty}\|V_\perp(t')\|_{H^1\times L^2}\;dt'&\le C\sup_{t'\ge t}\|V(t')\|_{H^1\times L^2}^{\frac{1}{2}}\left(\int_t^{+\infty}\|V(t')\|_{H^1\times L^2}\;dt'\right)^{\frac{1}{2}}\int_t^{+\infty}e^{-\frac{1}{2}e_\beta t'}\;dt'\\
&\le  Ce^{-\frac{1}{2}e_\beta t}\sup_{t'\ge t}\|V(t')\|_{H^1\times L^2}^{\frac{1}{2}}\left(\int_t^{+\infty}\|V(t')\|_{H^1\times L^2}\;dt'\right)^{\frac{1}{2}}.
\end{aligned}
\end{equation}
Gathering Lemma \ref{lem_est_V}, \eqref{est_Vperp}, and \eqref{est_Vperp_2}, we obtain:
\begin{multline*}
\|V(t)\|_{H^1\times L^2}\le Ce^{-e_\beta t}\int_t^{+\infty}\|V(t')\|_{H^1\times L^2}dt'\\
+C\left(e^{-\frac{1}{2}e_\beta t}\sup_{t'\ge t}\|V(t')\|_{H^1\times L^2}^{\frac{1}{2}}\left(\int_t^{+\infty}\|V(t')\|_{H^1\times L^2}\;dt'\right)^{\frac{1}{2}}\right).
\end{multline*}
Now integrating the preceding inequality (which is obviously possible) leads to:
\begin{multline*}
\int_t^{+\infty}\|V(t')\|_{H^1\times L^2}\;dt'\le C\int_t^{+\infty}\|V(t')\|_{H^1\times L^2}\;dt'.\int_t^{+\infty}e^{-e_\beta t'}dt'\\
+C\sup_{t'\ge t}\|V(t')\|_{H^1\times L^2}^{\frac{1}{2}}\left(\int_t^{+\infty}\|V(t')\|_{H^1\times L^2}\;dt'\right)^{\frac{1}{2}}\int_t^{+\infty}e^{-\frac{1}{2}e_\beta t'}dt'.
\end{multline*}
Or more simply
\begin{multline*}
\int_t^{+\infty}\|V(t')\|_{H^1\times L^2}\;dt'\le  Ce^{-e_\beta t}\int_t^{+\infty}\|V(t')\|_{H^1\times L^2}\;dt'\\
+Ce^{-\frac{1}{2}e_\beta t}\sup_{t'\ge t}\|V(t')\|_{H^1\times L^2}^{\frac{1}{2}}\left(\int_t^{+\infty}\|V(t')\|_{H^1\times L^2}\;dt'\right)^{\frac{1}{2}}.
\end{multline*}
Even if it means considering larger values of $t$, we obtain:
$$\int_t^{+\infty}\|V(t')\|_{H^1\times L^2}\;dt'\le  Ce^{-\frac{1}{2}e_\beta t}\sup_{t'\ge t}\|V(t')\|_{H^1\times L^2}^{\frac{1}{2}}\left(\int_t^{+\infty}\|V(t')\|_{H^1\times L^2}\;dt'\right)^{\frac{1}{2}},$$
which immediately implies the expected result.
\end{proof}

Now, let us show 

\begin{proposition}
For $t$ sufficiently large, $V(t)=0$.
\end{proposition}

\begin{proof}
Gathering Lemma \ref{lem_est_V} and Lemma \ref{est_V_integral}, we infer that for $t$ large: 
$$
\|V(t)\|_{H^1\times L^2}\le C\left(\sup_{t'\ge t}\|V_\perp(t')\|_{H^1\times L^2}+\int_t^{+\infty}\|V_\perp(t')\|_{H^1\times L^2}\;dt'\right).$$
From \eqref{est_Vperp} and \eqref{est_Vperp_2}, we deduce
$$\|V(t)\|_{H^1\times L^2}\le Ce^{-\frac{1}{2}e_\beta t}\sup_{t' \ge t}\|V(t')\|_{H^1\times L^2}^{\frac{1}{2}}\left(\int_t^{+\infty}\|V(t')\|_{H^1\times L^2}\;dt'\right)^{\frac{1}{2}}.$$
Now, it results from Lemma \ref{est_V_integral} again that 
$$\|V(t)\|_{H^1\times L^2}\le Ce^{-e_\beta t}\sup_{t'\ge t}\|V(t')\|_{H^1\times L^2}.$$
Thus we deduce that $\|V(t)\|_{H^1\times L^2}=0$ for large values of $t$.
\end{proof}

Finally let us observe below that $x_\infty=0$ so that $U=U^A$.

\begin{proposition}
There exists $t_0\ge 0$ such that for all $t\ge t_0$, $U(t)=U^A(t)$.
\end{proposition}

\begin{proof}
On the one hand, we have $\|U^A(t,\cdot-x_\infty)-R_\beta(t,\cdot-x_\infty)\|_{H^1\times L^2}\le Ce^{-e_\beta t}$.
On the other, we have $\|U(t)-R_\beta(t)\|_{H^1\times L^2}\to 0$ as $t\to+\infty$. Since we have $U(t)=U^A(t,\cdot-x_\infty)$, it follows from the triangular inequality that
$$\|R_\beta(t)-R_\beta(t,\cdot-x_\infty)\|_{H^1\times L^2}\to 0\quad\text{as } t\to+\infty.$$
Hence $x_\infty=0$ by the following claim, which is a consequence of Taylor formula.

\begin{Claim}
There exist $h_0>0$, $\epsilon_0>0$, and $\delta>0$ such that
\begin{enumerate}
\item if $|h|\le h_0$, then $\delta h^2\leq \|Q(\cdot+h)-Q\|_{H^1}^2\le 4\delta h^2$;
\item if $|h|>h_0$,
 then $\|Q(\cdot+h)-Q\|_{H^1}^2>\epsilon_0$.
\end{enumerate}
\end{Claim}

\end{proof}

\section{Appendix}

\subsection{Extension of the proofs to higher dimensions}

The main parts of the proofs remain obviously unchanged. Essentially three notable adaptations are to be made, passing from the one-dimensional case to higher dimensions. \\

In a first instance, one has to be careful about how establishing several estimates. Although all estimates we have proved in the previous sections (in dimension 1) are identical for general $d$, the way we establish them when $d\ge 2$ can be altered. \\
For example, we point out that it is no longer possible to use the Sobolev embedding $H^1\hookrightarrow L^\infty$ when $d\ge 2$. Particularly, multi-solitons in dimension $d\ge 2$ do not necessarily take values in $L^\infty(\RR^d)$ and in order to estimate a quantity like
$$f(u)-f(\varphi)-f'(\varphi)(u-\varphi)$$ (as for proving Claim \ref{claim_der_t_W}), we would proceed as follows: by \textbf{(H'1)}, we deduce that $|f''(r)|=p(p-1)|r|^{p-2}$, thus applying Taylor formula, we have for fixed time $t\in\RR$ and position $x\in\RR^d$,
\begin{align*}
\left|f(u)-f(\varphi)-f'(\varphi)(u-\varphi)\right|(t,x)&\le \frac{|u-\varphi|^2(t,x)}{2}\sup_{r\in [\varphi(t,x),u(t,x)]}|f''(r)|\\
&\le C|u-\varphi|^2(t,x)\left(|\varphi|^{p-2}(t,x)+|u|^{p-2}(t,x)\right).
\end{align*}
Now, for all $\psi\in L^\infty(\RR^d)$ and for all $z\in H^1(\RR^d)$, Hölder inequality yields
$$\left|\int_{\RR^d}|\psi||u(t)-\varphi(t)|^{2}|z|^{p-2}\;dx\right|\le C\|\psi\|_{L^\infty}\left(\int_{\RR^d}|u(t)-\varphi(t)|^{p}\;dx\right)^{\frac{2}{p}}\left(\int_{\RR^d}|z|^{p}\;dx\right)^{\frac{p-2}{p}}.$$
Finally, replacing $z$ by $\varphi(t)$ and $u(t)-\varphi(t)$, we obtain an estimate of $$\int_{\RR^d}\psi\left(f(u)-f(\varphi)-f'(\varphi)(u-\varphi)\right)\;dx$$ in terms of $\|u-\varphi\|_{H^1}$, $\|u\|_{H^1}$, and $\|\varphi\|_{H^1}$ due to the Sobolev embedding $H^1(\RR^d)\hookrightarrow L^p(\RR^d)$. \\

Secondly, in view of Proposition \ref{prop_spectral}, one has to take into account, for all $i=1,\dots,N$ the $d$ directions which generate the kernel of the operator $\mathscr{H}_{\beta_i}$ when we practice modulation in dimension $d$. For instance, in Lemma \ref{lem_mod_E}, we would define $E$ as follows:
$$E:=Z-\sum_{i=1}^Na_i\cdot\nabla R_i-\sum_{i=1}^Nb_iY_{+,i},$$ with $a_i(t)\in\RR^d$ and $b_i(t)\in\RR$ such that for all $i=1,\dots,N$ and for all $j=1,\dots,d$:
\begin{align}
\left\langle E(t),\partial_{x_j}R_i(t)\right\rangle&=0 \\
\left\langle E(t),Z_{-,i}(t)\right\rangle &=0. 
\end{align}
But each time this extension does not affect the sequel; in other words, the estimates we are supposed to obtain afterwards and their proofs are the same. \\

A third change to be done concerns the way we define the different Lyapunov functionals which are studied throughout the article. To deal with dimensions greater or equal than 2, we reduce the problem to the case of a one-dimensional variable.
For instance, let us explain how to generalize Step 4 in subsection \ref{subsect_improvement} to all dimensions. The subset $$\mathcal{M}:=\bigcup_{i\ne j}\left\{\ell\in\RR^d|\;\ell\cdot (\beta_j-\beta_i)=0\right\}$$ of $\RR^d$ is of zero Lebesgue measure. Hence, there exists $\ell\in\RR^d$ such that for all $i\ne j$, $$\ell\cdot (\beta_j-\beta_i)\ne 0.$$ In particular $\ell\neq 0$ and, even if it means considering $\frac{\ell}{|\ell|}$, we can assume that $|\ell|=1$, so that $\forall\;i=1,\dots,N$, $|\ell\cdot\beta_i|<1$. Now, defining $\tilde{\beta}_i:=\ell\cdot\beta_i$, and even if it means changing the permutation $\eta$, we have $$-1<\tilde{\beta}_{\eta(1)}<\tilde{\beta}_{\eta(2)}<\dots<\tilde{\beta}_{\eta(N)}<1.$$
Then, the direction described by $\ell$ is to be favored: we consider the following cut-off functions:
$$\psi_k(t)=\psi\left(\frac{1}{\sqrt{t}}\left(\ell\cdot x-\frac{\tilde{\beta}_{\eta(k)}+\tilde{\beta}_{\eta(k+1)}}{2}-\ell\cdot\frac{x_{\eta(k)}+x_{\eta(k+1)}}{2}\right)\right).$$
At this stage, the definition of the functions $\phi_k$ in terms of the $\psi_k$ is kept unchanged and the corresponding Lyapunov functional is to be written:
$$\mathcal{F}_{W}(t)=\sum_{k=1}^K\int_{\RR^d}\left(w_1^2+(\partial_xw_{1})^2+w_2^2-f'(Q_{\eta(k)})w_1^2+2\beta_{\eta(k)}\cdot\nabla w_1w_2\right)\phi_k\;dx.$$

\subsection{Proof of Corollary \ref{cor_special_solutions}}

The proof is an immediate adaptation of that of Proposition 4.12 in \cite{combet2010}. \\

Let $A>0$ and denote $t_A:=-\frac{\ln(A)}{e_\beta}$. In the sense of the $H^1\times L^2$-norm, we have:
\begin{align*}
U^1(t+t_A,\cdot +\beta t_A)&=R_\beta(t+t_A,\cdot+\beta t_A)+e^{-e_\beta(t+t_A)}Y_{+,\beta}(t+t_A,\cdot+\beta t_A)+\mathrm{O}\left(e^{-2e_\beta t}\right)\\
&=R_\beta(t)+Ae^{-e_\beta t}Y_{+,\beta}(t)+\mathrm{O}\left(e^{-2e_\beta t}\right).
\end{align*}
Then, $\|U^1(t+t_A,\cdot +\beta t_A)-R_\beta(t)\|_{H^1\times L^2}\underset{t\to +\infty}{\rightarrow}0$ so that there exist $\tilde{A}\in\RR$ and $t_0=t_0(\tilde{A})\in\RR$ such that for all $t\ge t_0$,
$$U^{\tilde{A}}(t)=U^1(t+t_A,\cdot +\beta t_A).$$
But on the other hand, 
$$U^{\tilde{A}}(t)=R_\beta(t)+\tilde{A}e^{-e_\beta t}Y_{+,\beta}(t)+\mathrm{O}\left(e^{-2e_\beta t}\right).$$
Hence, $$(A-\tilde{A})e^{-e_\beta t}Y_{+,\beta}(t)=\mathrm{O}\left(e^{-2e_\beta t}\right),$$
which implies $A=\tilde{A}$.
Consequently, $U^A(t)=U^1(t+t_A,\cdot +\beta t_A).$ \\

If $A<0$, we have just to repeat the above argument with $-A$ instead of $A$. \\

Lastly, let us identify $U^0$. Given that $R_\beta$ is a solution of \eqref{NLKG} which satisfies \eqref{cv_zero}, Theorem \ref{th_main_N} provides the existence of $A\in\RR$ and of $t_0\in\RR$ such that for all $t\ge t_0$, $U^A(t)=R_\beta(t)$. \\
Since $U^A$ satisfies \eqref{est_U^A_thm}, we deduce that
$$\left\|Ae^{-e_\beta t}Y_{+,\beta}(t)\right\|_{H^1\times L^2}\le Ce^{-2e_\beta t}.$$
Thus $A=0$ and $U^0=R_\beta$ is defined for all $t\in\RR$.

\subsection{A result of analytic theory of differential equations}

\begin{Lem}\label{lem_gen_diff}
Let $t_0\in\RR$, $\mathcal{A}:[t_0,+\infty)\to\RR$ be a $\mathscr{C}^1$ bounded function, and $\xi:[t_0,+\infty)\to\RR^+$ be continuous and integrable. \\
If, for some $\rho>0$,
$$\forall\;t\ge t_0,\qquad |\mathcal{A}'(t)+\rho\mathcal{A}(t)|\le \xi(t)\sup_{t'\ge t}|\mathcal{A}(t')|,$$
then there exists $c>0$ such that 
$$\forall\;t\ge t_0,\qquad |\mathcal{A}(t)|\le ce^{-\rho t}.$$
\end{Lem}

\begin{proof}
Let us assume that 
\begin{equation}\label{hyp_lem_gen_diff}
\forall\;t\ge t_0,\qquad |\mathcal{A}'(t)+\rho\mathcal{A}(t)|\le \xi(t)\sup_{t'\ge t}|\mathcal{A}(t')|,
\end{equation} for some $\rho>0$.
Then for all $t\ge t_0$,
$$|(e^{\rho t}\mathcal{A})'(t)|\le \xi(t)e^{\rho t}\sup_{t'\ge t}|\mathcal{A}(t')|.$$
Let us consider $t\ge t_0$. For $t'\ge t$, we obtain by integration
$$|e^{\rho t'}\mathcal{A}(t')- e^{\rho t}\mathcal{A}(t)|\le \int_t^{t'}\xi(s)e^{\rho s}\sup_{u\ge s}|\mathcal{A}(u)|\;ds.$$ This implies that, for $t'\ge t$,
$$e^{\rho t'}|\mathcal{A}(t')|\le e^{\rho t}|\mathcal{A}(t)|+\sup_{u\ge t}|\mathcal{A}(u)|e^{\rho t'}\int_t^{t'}\xi(s)\;ds.$$ From the preceding line, we deduce that for all $t'\ge t$,
\begin{equation}\label{ineq_lem_gen_diff}
|\mathcal{A}(t')|\le |\mathcal{A}(t)|+\sup_{u\ge t}|\mathcal{A}(u)|\int_t^{+\infty}\xi(s)\;ds.
\end{equation}
Now we consider $t_1\ge t_0$ such that $\int_{t_1}^{+\infty}\xi(s)\;ds<\frac{1}{2}$ (which is indeed possible given that $\int_t^{+\infty}\xi(s)\;ds\to 0$ as $t\to +\infty$).
By passing to the supremum on $t'$ in \eqref{ineq_lem_gen_diff}, we obtain for all $t\ge t_1$,
$$\sup_{t'\ge t}|\mathcal{A}(t')|\le 2|\mathcal{A}(t)|.$$
Consequently, assumption \eqref{hyp_lem_gen_diff} becomes 
\begin{equation}\label{hyp_lem_gen_diff1}
\forall\;t\ge t_1,\qquad |\mathcal{A}'(t)+\rho\mathcal{A}(t)|\le 2\xi(t)|\mathcal{A}(t)|.
\end{equation}
Let us define $y(t):=e^{\rho t}|\mathcal{A}(t)|$. By integration of \eqref{hyp_lem_gen_diff1}, we obtain
\begin{equation}\label{hyp_lem_gen_diff2}
\forall\;t\ge t_1,\qquad y(t)\le y(t_1)+\int_{t_1}^t2\xi(s)y(s)\;ds.
\end{equation}
By a standard Grönwall argument, we conclude to the existence of $C>0$ such that for all $t\ge t_1, y(t)\le C$, which implies the desired result. For the sake of completeness, let us explicit this argument.\\
We define $Y(t):=\exp\left(-\int_{t_1}^t2\xi(s)\;ds\right)\int_{t_1}^t2\xi(s)y(s)\;ds$ for $t\ge t_1$. The function $Y$ is $\mathscr{C}^1$ on $[t_1,+\infty)$ and for all $t\ge t_1$,
\begin{align*}
Y'(t)&=2\xi(t)\exp\left(-\int_{t_1}^t2\xi(s)\;ds\right) \left[y(t)-\int_{t_1}^t2\xi(s)y(s)\;ds\right]\\
&\le 2y(t_1)\xi(t)\exp\left(-\int_{t_1}^t2\xi(s)\;ds\right)
\end{align*}
by \eqref{hyp_lem_gen_diff2}.
Integrating the preceding inequality and observing that $Y(t_1)=0$, we have
$$Y(t)\le \int_{t_1}^t2\xi(s)y(t_1)\exp\left(-\int_{t_1}^s2\xi(u)\;du\right)\;ds.$$
We then infer
\begin{equation}\label{hyp_lem_gen_diff3}
\int_{t_1}^t2\xi(s)y(s)\;ds=\exp\left(\int_{t_1}^t2\xi(s)\;ds\right)Y(t)\le 2y(t_1)\int_{t_1}^t\xi(s)\exp\left(\int_s^t2\xi(u)\;du\right)\;ds.
\end{equation}
Lastly, we denote $\nu:=\int_{t_1}^{+\infty}\xi(s)\;ds$; gathering \eqref{hyp_lem_gen_diff2} and \eqref{hyp_lem_gen_diff3}, we obtain
$$\forall\;t\ge t_1,\qquad y(t)\le y(t_1)+2y(t_1)e^{2\nu}\nu.$$
This achieves the proof of Lemma \ref{lem_gen_diff}.
\end{proof}

\bibliographystyle{plain}
\bibliography{references_nlkg}

\end{document}